\documentclass[a4paper,reqno]{amsart}

\usepackage{amsfonts}
\usepackage{mathrsfs}
\usepackage{graphicx}
\usepackage[normalem]{ulem}
\usepackage{url}

\usepackage{pdftexcmds}
\usepackage{mathtools}

\usepackage{enumitem}

\usepackage{amsthm,amsfonts,amssymb,amsmath,esint}
\usepackage[colorlinks=true, urlcolor=black]{hyperref}
\usepackage{bbm}
\usepackage{appendix}
\usepackage{stackrel}

\usepackage{stackengine}
\setstackEOL{\\}
\setstackgap{L}{\normalbaselineskip}

\usepackage{fullpage}
\usepackage[all]{xy}
\usepackage{xspace}
\xspaceaddexceptions{]\}}

\newcommand{\RNum}[1]{\uppercase\expandafter{\romannumeral #1\relax}}

\newcommand{\sS}{\mathscr{S}}
\newcommand{\resp}{\text{resp.}\xspace}
\newcommand{\Log}{{\operatorname{Log}}}

\newcommand{\bG}{\mathbf{G}}
\newcommand{\bC}{\mathbf{C}}
\newcommand{\bGSp}{\mathbf{GSp}}
\newcommand{\bV}{\mathbf{V}}
\newcommand{\bE}{\mathbf{E}}
\newcommand{\bK}{\mathbf{K}}
\newcommand{\bZ}{\mathbf{Z}}
\newcommand{\bF}{\mathbf{F}}
\newcommand{\bh}{\mathbf{h}}
\newcommand{\rK}{\mathrm{K}}
\newcommand{\simtos}{\overset{\sim}{\ra}}
\newcommand{\bbf}{\mathbf f}
\newcommand{\GSpin}{{\operatorname{GSpin}}}

\DeclareMathOperator{\GU}{GU}

\newcommand{\sG}{\mathscr{G}}
\newcommand{\sB}{\mathscr{B}}

\newcommand{\BA}{{\mathbb {A}}} 
\newcommand{\BC}{{\mathbb {C}}} \newcommand{\BD}{{\mathbb {D}}}
\newcommand{\CD}{{\mathcal {D}}}
 \newcommand{\BF}{{\mathbb {F}}}
\newcommand{\BG}{{\mathbb {G}}} \newcommand{\BH}{{\mathbb {H}}}
\newcommand{\BI}{{\mathbb {I}}} 
 
\newcommand{\BM}{{\mathbb {M}}} 
\newcommand{\BN}{{\mathbb {N}}}
 
\newcommand{\BQ}{{\mathbb {Q}}} \newcommand{\BR}{{\mathbb {R}}}
\newcommand{\BS}{{\mathbb {S}}} 
 \newcommand{\BV}{{\mathbb {V}}}
\newcommand{\BW}{{\mathbb {W}}} 
 \newcommand{\BZ}{{\mathbb {Z}}}

\newcommand{\CA}{{\mathcal {A}}} \newcommand{\CB}{{\mathcal {B}}}
\newcommand{\CC}{{\mathcal {C}}} 
 \newcommand{\CF}{{\mathcal {F}}}
\newcommand{\CG}{{\mathcal {G}}} \newcommand{\CH}{{\mathcal {H}}}

\newcommand{\CO}{{\mathcal {O}}} \newcommand{\CP}{{\mathcal {P}}}
 
 \newcommand{\CT}{{\mathcal {T}}}
 \newcommand{\CV}{{\mathcal {V}}}

\newcommand{\RM}{{\operatorname {M}}} 

\newcommand{\RU}{{\operatorname {U}}}

\newcommand{\fa}{{\mathfrak{a}}}

\newcommand{\fm}{{\mathfrak{m}}}

\newcommand{\fM}{{\mathfrak{M}}}

\newcommand{\fS}{{\mathfrak{S}}}

\newcommand{\ab}{{\operatorname{ab}}}
\newcommand{\ad}{{\operatorname{ad}}}

\newcommand{\cris}{{\operatorname{cris}}}

\newcommand{\dR}{{\operatorname{dR}}}
\newcommand{\disc}{{\operatorname{disc}}}
\newcommand{\der}{{\operatorname{der}}}
\newcommand{\diag}{{\operatorname{diag}}}

\newcommand{\Frac}{\operatorname{Frac}}

\newcommand{\Gal}{{\operatorname{Gal}}} \newcommand{\GL}{{\operatorname{GL}}}
\newcommand{\GSp}{{\operatorname{GSp}}}
\newcommand{\Hom}{{\operatorname{Hom}}}

\newcommand{\Isom}{{\operatorname{Isom}}}

\newcommand{\Lie}[1]{\operatorname{Lie} #1} 
\newcommand{\loc}{{\operatorname{loc}}}

\renewcommand{\mod}{\operatorname{mod}}

\newcommand{\Res}{{\operatorname{Res}}}

\newcommand{\Sh}{{\operatorname{Sh}}}

\newcommand{\Spec}{\operatorname{Spec}} \newcommand{\Spf}{\operatorname{Spf}}

\newcommand{\tr}{\operatorname{tr}}
\newcommand{\univ}{{\operatorname{univ}}}
\newcommand{\ur}{{\operatorname{ur}}}

\newcommand{\wt}[1]{{\widetilde {#1}}}
\newcommand{\wh}[1]{{\widehat {#1}}}

\newcommand{\pair}[1]{\langle {#1} \rangle}

\newcommand{\ol}[1]{{\overline{#1}}}

\newcommand{\lra}{\longrightarrow}

\newcommand{\ra}{\rightarrow}

\newcommand{\ud}[1]{{\underline{#1}}}

\newcommand{\rbra}[1]{\left({#1}\right)}

\newcommand{\cbra}[1]{\left\{{#1}\right\}}

\newcommand{\tcbra}[1]{\{{#1}\}}

\makeatletter\def\varW@#1#2{%
\vtop{\m@th\ialign{##\cr
\hfil$#1 \operatorname{colim} $\hfil\cr
\noalign{\nointerlineskip\kern1.5\ex@}#2\cr
\noalign{\nointerlineskip\kern-\ex@}\cr}}
}
\makeatother
\makeatletter
\def\colim{%
\mathop{\mathpalette\varW@{}}\nmlimits@
}\makeatother

\theoremstyle{plain}
\newtheorem{thm}{Theorem}[section] 
\newtheorem{corollary}[thm]{Corollary}
\newtheorem{lem}[thm]{Lemma}  \newtheorem{prop}[thm]{Proposition}
\newtheorem {conj}[thm]{Conjecture} 
\newtheorem{lemma}[thm]{Lemma}

\theoremstyle{remark} \newtheorem{remark}[thm]{Remark}
\theoremstyle{definition}  
\theoremstyle{definition} \newtheorem{example}[thm]{Example}  
\newtheorem{defn}[thm]{Definition}

\newcommand{\adeles}{ad\`{e}les\xspace}

\newcommand{\etale}{\'{e}tale\xspace}
\newcommand{\et}{\'{e}t\xspace}
\renewcommand{\et}{{\text{\'{e}t}}\xspace}

\newcommand{\Teichmuller}{Teichm\"{u}ller\xspace}

\newcommand{\Dieudonne}{Dieudonn\'{e}\xspace}

\numberwithin{equation}{section}

\newcommand{\inverse}{^{-1}}

\newcommand{\sset}{\subset}

\newcommand{\cross}{^\times}

\newcommand{\rk}{\operatorname{rk}}

\newcommand{\Tr}{\operatorname{Tr}}

\newcommand{\Fil}{\operatorname{Fil}}

\newcommand{\simto}{\overset{\sim}{\lra}}

\newcommand{\Gr}{\operatorname{Gr}}

\newcommand{\dfn}[1]{\emph{#1}}

\begin{document}
\title{$2$-adic integral models of some Shimura varieties with parahoric level structure}
\author{Jie Yang}
\address{Yau Mathematical Sciences Center, Tsinghua University, Haidian District, Beijing, 100084, China}
\email{jie-yang@mail.tsinghua.edu.cn}

\begin{abstract}
	 We construct integral models over $p=2$ for some Shimura varieties of abelian type with parahoric level structure, extending the previous work of Kim--Madapusi \cite{kim20162} and Kisin, Pappas, and Zhou \cite{kisin2018integral, kisin2024independenceellfrobeniusconjugacy, kisin2024integralmodelsshimuravarieties}. These integral models are canonical in the sense of Pappas--Rapoport \cite{pappas2021p}.
\end{abstract}

\maketitle

\setcounter{tocdepth}{1}
\tableofcontents

\section{Introduction}

\subsection{Background}
Let $p$ be a prime number. Let $\BA_f$ denote the ring of finite \adeles over $\BQ$, and $\BA^p_f$ denote the ring of prime-to-$p$ finite \adeles over $\BQ$. Let $(\bG,X)$ be a Shimura datum. For a sufficiently small open compact subgroup $\rK\sset\bG(\BA_f)$, the corresponding Shimura variety $$\Sh_\rK(\bG,X)=\bG(\BQ)\backslash X\times \bG(\BA_f)/\rK$$
is naturally a quasi-projective smooth scheme over the reflex field $\bE=\bE(\bG,X)$. 

Let $v|p$ be a place of $\bE$ and $E$ be the completion of $\bE$ at $v$. Denote by $\CO_{\bE,(v)}$ the localization of $\CO_\bE$ at $v$. Denote by $k_E$ the residue field of $E$ and by $k$ the algebraic closure of $k_E$. Let $\CG$ be a Bruhat--Tits stabilizer group scheme (see \S \ref{subsec-localmod}) over $\BZ_p$ for $\bG_{\BQ_p}$ with neutral component $\CG^\circ$. Set $\rK_p\coloneqq \CG(\BZ_p)$ and $\rK_p^\circ\coloneqq \CG^\circ(\BZ_p)$. Suppose $\rK^\circ\sset\bG(\BA_f)$ is of the form $\rK^\circ=\rK_p^\circ\rK^p$, where $\rK^p\sset \bG(\BA_f^p)$ is a sufficiently small open compact subgroup.  

Suppose that $(\bG,X)$ is of abelian type. If $\rK^\circ_p$ is hyperspecial, Kisin \cite{kisin2010integral} and Kim--Madapusi \cite{kim20162} (when $p=2$) constructed (smooth) canonical integral models over $p$ of $\Sh_{\rK^\circ}(\bG,X)$, which are uniquely characterized by Milne's extension property in the sense of \cite[2.3.7]{kisin2010integral}. If $p>2$, Kisin, Pappas and Zhou \cite{kisin2024independenceellfrobeniusconjugacy, kisin2024integralmodelsshimuravarieties}, following earlier work of Kisin--Pappas \cite{kisin2018integral}, constructed integral models over $p$ of $\Sh_{\rK^\circ}(\bG,X)$ with arbitrary parahoric level structure.  Using the theory of $p$-adic shtukas, Pappas--Rapoport \cite{pappas2021p} (and Daniels \cite{daniels2023canonical})  made the following conjecture about the existence of the \dfn{canonical integral model} of $\Sh_{\rK^\circ}(\bG,X)$ with parahoric level structure for any Shimura datum $(\bG,X)$.

\begin{conj}[{\cite[Conjecture 4.2.2]{pappas2021p},\cite[Conjecture 4.5]{daniels2023canonical}}] \label{conj}
	There exists a unique system $\{\sS_{\rK^\circ} \}_{\rK^p}$ of normal flat schemes over $\CO_E$, extending $\{\Sh_{\rK^\circ}(\bG,X)\}_{\rK^p}$ and equipped with a $p$-adic shtuka satisfying the axioms in \textit{loc. cit.}.
\end{conj}
Conjecture \ref{conj} holds if $(\bG,X)$ is of Hodge type (\cite[Theorem 1.3.2]{pappas2021p}, \cite[Theorem 4.2.3]{daniels2024conjecture}); or $(\bG,X)$ is of abelian type and $p>2$ (\cite[Theorem A]{daniels2024canonical}). It follows that the integral models constructed by Kisin--Pappa--Zhou are canonical and independent of all auxiliary choices made in the construction.

Assuming the existence of $\sS_{\rK^\circ}$ as in Conjecture \ref{conj}, Pappas and Rapoport also conjectured (at least when $\bG$ satisfies the assumption \cite[(4.1.1)]{pappas2021p}) that $\sS_{\rK^\circ}$ fits into a \dfn{scheme-theoretic local model diagram} in the sense of \cite[Definition 4.9.1]{pappas2021p}. Specifically, there should exist a diagram of $\CO_E$-schemes \begin{equation} 
       \begin{gathered}
       	  \xymatrix{
       	   &\sS_{\rK^\circ} & \wt{\sS}_{\rK^\circ}\ar[l]_-{\pi}\ar[r]^-{q} &\BM^\loc_{\CG^\circ,\mu_h},
		    }
       \end{gathered}  \label{diamlocal}
	\end{equation}
    where $\mu_h$ denotes the geometric cocharacter of $\bG_{\BQ_p}$ corresponding to the Hodge cocharacter attached to $(\bG,X)$, the $\CO_E$-scheme $\BM^\loc_{\CG^\circ,\mu_h}$ denotes the scheme local model used in \cite[\S 4.9.2]{pappas2021p}, $\pi$ is a $\CG^\circ$-torsor, and $q$ is $\CG^\circ$-equivariant and smooth of relative dimension $\dim\bG$, such that the compatibility conditions in \cite[Definition 4.9.1]{pappas2021p} are satisfied. Here, the compatibility conditions require that the generic fiber of the diagram \eqref{diamlocal} is given by the Borel embedding construction, and that the associated v-sheaves recover the Pappas--Rapoport v-sheaf local model diagram in \cite[(4.9.4)]{pappas2021p}; see more details in \cite[Definition 4.9.1]{pappas2021p}. For $p>2$, a scheme-theoretic local model diagram is known to exist if $(\bG,X)$ is of Hodge type and satisfies the assumptions in \cite[Theorem II]{daniels2024conjecture}, which build upon the work of Kisin, Pappas and Zhou \cite{kisin2018integral, kisin2024independenceellfrobeniusconjugacy, kisin2024integralmodelsshimuravarieties}.   
    
    In this paper, we extend results in \cite{daniels2024canonical, daniels2024conjecture, kisin2024integralmodelsshimuravarieties} to the case $p=2$ under some assumptions, which we now explain.

\subsection{Main results} \label{intro-results}
Assume $p=2$ and that $(\bG,X)$ is a Shimura datum of abelian type. We will construct $2$-adic integral models for $\Sh_{\rK^\circ}(\bG,X)$ under one of the following assumptions: 
    \begin{enumerate}
    	\item [(A)] $(\bG^\ad,X^\ad)$ has no factor of type $D^\BH$, $\bG_{\BQ_p}$ is unramified, and $\rK^\circ_p$ is contained in some hyperspecial subgroup;
    	\item [(B)] $\bG=\GU(n-1,1)$ is the unitary similitude group over $\BQ$ of signature $(n-1,1)$ for some odd integer $n\geq 3$, $\bG_{\BQ_p}$ is (wildly) ramified, and $\rK_p^\circ$ is a special parahoric subgroup.
    \end{enumerate}

\begin{thm} \label{thm1}
	Assume that either (A) or (B) holds. 	
	\begin{enumerate}[label=(\arabic*)] 
        \item The $\bE$-scheme \[\Sh_{\rK_p^\circ}(\bG,X)\coloneqq \varprojlim_{\rK^p}\Sh_{\rK^\circ_p\rK^p}(\bG,X) \] admits a $\bG(\BA_f^p)$-equivariant extension to a flat normal $\CO_{E}$-scheme $\sS_{\rK^\circ_p}(\bG,X)$. Any sufficiently small $\rK^p\sset \bG(\BA_f^p)$ acts freely on $\sS_{\rK_{p}^\circ}(\bG,X)$, and the quotient $$\sS_{\rK^\circ}(\bG,X)\coloneqq \sS_{\rK^\circ_{p}}(\bG,X)/\rK^p$$ is a flat normal $\CO_{E}$-scheme extending $\Sh_{\rK^\circ}(\bG,X)$. 
		\item For any discrete valuation ring $R$ of mixed characteristic $0$ and $p$, the map $$\sS_{\rK_{p}^\circ}(\bG,X)(R)\ra \sS_{\rK_{p}^\circ}(\bG,X)(R[1/p])$$ is a bijection.
		\item There exists a diagram of $\CO_{E}$-schemes  
		    $$\xymatrix{
		   &\wt{\sS}_{\rK_{p}^\circ}^\ad\ar[ld]_{\pi}\ar[rd]^q &\\ \sS_{\rK_{p}^\circ}(\bG,X) & &\BM^\loc_{\CG^\circ,\mu_{h}},
		}$$ where $\pi$ is a $\bG(\BA_f^p)$-equivariant $G_{\BZ_p}^{\ad\circ}$-torsor, $q$ is $G_{\BZ_p}^{\ad\circ}$-equivariant, and for any sufficiently small $\rK^p\sset \bG(\BA_f^p)$, the map $\wt{\sS}_{\rK_{p}^\circ}^\ad/\rK^p \ra \BM^\loc_{\CG^\circ,\mu_{h}}$ induced by $q$ is smooth of relative dimension $\dim \bG^\ad$. 
		\item If $\kappa$ is a finite extension of $k_E$ and $y\in\sS_{\rK_{p}^\circ}(G,X)(\kappa)$, then there exists $z\in \BM^\loc_{\CG^\circ,\mu_{h}}(\kappa)$ such that we have an isomorphism of henselizations $$\CO_{\sS_{\rK_{p}^\circ}(\bG,X), y}^h\simeq \CO_{\BM^\loc_{\CG^\circ,\mu_{h}},z}^h.$$ 
	\end{enumerate}
\end{thm}

Here in (3), $G^{\ad\circ}_{\BZ_p}$ denotes the parahoric group scheme over $\BZ_p$ with generic fiber $\bG^\ad_{\BQ_p}$, defined by $\CG^\circ$ using the map $\CB(\bG_{\BQ_p},\BQ_p)\ra \CB(\bG^\ad_{\BQ_p},\BQ_p)$ between extended Bruhat--Tits buildings, see \S \ref{subsect-cond}. The proof of Theorem \ref{thm1} will be given in \S \ref{caseA} and \S \ref{sec-application}.

\begin{remark}
    \begin{enumerate}
    	% \item When $\rK_p^\circ$ is hyperspecial and $(\bG,X)$ is of abelian type, Kim--Madapusi \cite[Theorem 4.11]{kim20162} construct the integral canonical model $\sS_{\rK_p^\circ}(\bG,X)$ of $\Sh_{\rK^\circ}(\bG,X)$, which is formally smooth over $\CO_{\bE,(v)}$ and uniquely characterized by Milne's extension property (as stated in \textit{loc. cit.}). By construction and the proofs given there, it is not hard to see that $\sS_{\rK_p^\circ}(\bG,X)$ satisfies the assertions of Theorem \ref{thm1}; we omit the details. 
        \item The system of integral model $\sS_{\rK_p^\circ}(\bG,X)$ in Theorem \ref{thm1} is canonical, i.e., it satisfies the conditions in Conjecture \ref{conj}; see Corollary \ref{coro-abelian} and Remark \ref{hodgecan}.
        \item In Case (B) of Theorem \ref{thm1}, where $(\bG,X)$ is of  Hodge type, we can construct a local model diagram as in Theorem \ref{thm-hodgemodel} (3), which is scheme-theoretic by Remark \ref{hodgecan}.
    	\item We expect that the results of van Hoften \cite{van2020mod} and Gleason--Lim--Xu \cite{gleason2022connected} can be extended to the $2$-adic models constructed in this paper. 
    \end{enumerate}
\end{remark}

Let us give two interesting cases in which Theorem \ref{thm1} can be applied to obtain integral models over $\BZ_{(2)}$ for $\Sh_{\rK_2^\circ\rK^2}(\bG,X)$ when $\rK_2^\circ$ is a parahoric subgroup contained in some hyperspecial subgroup. Let $F$ be a totally real number field which is unramified at primes over $2$. 
  \begin{enumerate}[label=(\roman*)]
	\item $\bG=\Res_{F/\BQ}\GSpin(V,Q)$, where $\GSpin(V,Q)$ is the spin similitude group over $F$ attached to a quadratic space $(V,Q)$ of signature $(n,2)$ at each real place (assume $\GSpin(V,Q)$ is unramified over $F_v$, $v|2$) and $X$ is (a product of) the space of oriented negative definite planes;
	\item $\bG=\Res_{F/\BQ}\operatorname{GU}$, where $\operatorname{GU}$ is the unitary similitude group over $F$ that is unramified over $F_v$, $v|2$. We note that this case is also known by \cite[Appendix A]{rapoport2021shimura}.
\end{enumerate} 

As in \cite[Corolary 0.3]{kisin2018integral}, Theorem \ref{thm1} together with \cite[Theorem B]{He_Schremmer_Yu_2026} imply the following.

\begin{corollary}
	With the assumptions as in Theorem \ref{thm1}, the integral model $\sS_{\rK^\circ_p}(\bG,X)$ is normal and Cohen--Macaulay with reduced special fiber, and the strict henselizations of the local rings on $\sS_{\rK_p^\circ}(\bG,X)\otimes k_E$ have irreducible components which are normal and Cohen--Macaulay.
	
	If $\rK_p^\circ$ is associated to a point $x\in\CB(\bG_{\BQ_p},\BQ_p)$ which is a special vertex in $\CB(\bG_{\BQ_p},\BQ_p^\ur)$, then the special fiber $\sS_{\rK_p^\circ}(\bG,X)\otimes k_E$ is normal and Cohen--Macaulay.
\end{corollary}

\subsection{Strategy of the proof}
We now explain the idea to prove Theorem \ref{thm1}. The overall strategy follows that of \cite{kisin2018integral} and \cite{kisin2024integralmodelsshimuravarieties}. As in \textit{loc. cit.}, the crucial case is when $(\bG,X)$ is of Hodge type. A key step in this case involves identifying the formal neighborhood of $\sS_\rK(\bG,X)$ with that of the local model $\BM^\loc_{\CG,\mu_h}$. For $p>2$, this identification is obtained in \cite{kisin2018integral, kisin2024integralmodelsshimuravarieties}  by constructing a universal deformation of $p$-divisible groups (equipped with a family of crystalline tensors) over the formal neighborhood of the local model. The construction of this universal deformation uses Zink's theory of \Dieudonne displays that classify $p$-divisible groups. For $p=2$, we modify Zink's theory by using Lau's results from \cite{lau2014relations}, and obtain a similar deformation theory for $2$-divisible groups. A technical requirement arising in this step is that we need to find a Hodge embedding $$\iota: (\bG,X)\hookrightarrow (\bGSp(V,\psi),S^\pm),$$ where $V$ is a $\BQ$-vector space of dimension $2g$ equipped with a perfect alternating pairing $\psi$, such that $\iota_{\BQ_p}$ extends to a \dfn{very good} integral Hodge embedding $(\CG,\mu_h)\hookrightarrow (\GL(\Lambda),\mu_g)$, where $\Lambda\sset V_{\BQ_p}$ is a \dfn{self-dual} $\BZ_p$-lattice with respect to $\psi$. 

The concept of very good integral Hodge embeddings was introduced in \cite[\S 5.2]{kisin2024integralmodelsshimuravarieties} for $p>2$, refining the notion of \dfn{good} integral Hodge embeddings in \cite[Definition 3.1.6]{kisin2024independenceellfrobeniusconjugacy}. We generalize the concept to the case $p=2$ (see Definition \ref{defn-very}). Roughly speaking, a good integral integral Hodge embedding is an integral Hodge embedding $$\wt{\iota}: (\CG,\mu_h)\hookrightarrow (\GL(\Lambda),\mu_g)$$ extending $\iota_{\BQ_p}$ such that $\wt{\iota}$ induces a closed immersion \begin{flalign*}
	     \BM^\loc_{\CG,\mu_h}\hookrightarrow \BM^\loc_{\GL(\Lambda),\mu_g}\otimes_{\BZ_p}\CO_E=\Gr(g,\Lambda)\otimes_{\BZ_p}\CO_E
\end{flalign*} of local models, where $\Gr(g,\Lambda)$ denotes the Grassmannian of rank $g$ subspaces of $\Lambda$. The key idea behind very good Hodge embeddings is that certain collection of tensors $(s_\alpha)$ in the tensor algebra $\Lambda^\otimes$, cutting out $\CG$ in $\GL(\Lambda)$, should satisfy a ``horizontal" condition under the natural connection isomorphism. We refer to \S \ref{subsec-localmod} for more details. 
For a good integral Hodge embedding $\wt{\iota}$, Kisin--Pappas--Zhou proved in  \cite[Proposition 5.3.11, Lemma 5.3.2]{kisin2024integralmodelsshimuravarieties} that this horizontality condition is satisfied in the following two cases (including for $p=2$): 
\begin{enumerate}
	\item For any $x\in \BM^\loc_{\CG,\mu_h}(k)$, the image of the natural map \begin{flalign*}
		    \cbra{f\in \BM^\loc_{\CG,\mu_h}(k[[t]]) \ |\ f\mod (t)=x }\ra T_x\BM^\loc_{\CG,\mu_h}
	\end{flalign*} spans, as a $k$-vector space, the tangent space $T_x\BM^\loc_{\CG,\mu_h}$. 
	\item The tensors $(s_\alpha)\sset \Lambda^\otimes$ are in $\Lambda\otimes_{\BZ_p}\Lambda^\vee$.
\end{enumerate}  
Using this, they can produce sufficiently many very good Hodge embeddings when $p>2$.

When $p=2$, it is in general difficult to find a very good integral Hodge embedding $\wt{\iota}$ for a Shimura datum of Hodge type. 
In the present paper, we establish the existence of very good Hodge embeddings under the assumption (A) or (B). 

For Case (A), by applying \cite[Proposition 5.3.11, Lemma 5.3.2]{kisin2024integralmodelsshimuravarieties}, we are reduced to presenting the stabilizer group scheme $\CG$ as $(\Res_{\CO_F/\BZ_p}\CH)^\Gamma$, where $F/\BQ_p$ is a tame Galois extension with Galois group $\Gamma$ and $\CH$ is a reductive group over $\CO_F$. For Case (B), we directly prove that the tangent space of the local model $\BM^\loc_{\CG,\mu_h}$ at any closed point is spanned by formal curves (see Lemma \ref{lem-spantc}), using the explicit description of the (local) coordinate rings of the unitary local models in \cite[Theorem 1.2, 1.3]{yang24}.

\subsection{Organization}
We now give an overview of the paper.

In \S \ref{sec-lau}, we review Lau's results from \cite{lau2014relations}, which generalizes Zink's theory of \Dieudonne displays so that we can classify $2$-divisible groups over $2$-adic rings (see Theorem \ref{thm-classify}). A new feature of the theory of \Dieudonne displays in the case $p=2$ is the modified Verschiebung map for the Zink ring (see Lemma \ref{modifiedV}). In \S \ref{Dpairs}, we construct the natural ``connection isomorphisms" for \Dieudonne pairs when $p=2$ (see Lemma \ref{lem-connection}), generalizing \cite[Lemma 5.1.15]{kisin2024integralmodelsshimuravarieties} for $p>2$.  In \S \ref{subsubsec-compBK}, we compare Lau's classification of $p$-divisible groups with Breuil--Kisin's classification. This comparison is later used in \S \ref{subsubsec-hodgetensors} to construct $(\CG_W,\mu_y)$-adapted deformations of $p$-divisible groups in the sense of Definition \ref{defn-adapted}. 

In \S \ref{sec-defm}, we apply Lau's theory to construct a universal deformation of $p$-divisible groups, extending results from \cite[\S 3]{kisin2018integral} to the case $p=2$. We also generalize the concept of very good Hodge embeddings, introduced in \cite{kisin2024integralmodelsshimuravarieties}, to $p=2$. This is used to construct universal deformations of $p$-divisible groups with crystalline tensors (see Proposition \ref{prop-torsor}). In Proposition \ref{prop-adapted}, we establish a criterion for determining when a deformation is $(\CG_W,\mu_y)$-adapted, extending \cite[Proposition 4.7]{zhou2020mod} to $p=2$. 

In \S \ref{sec-intmodelhodge}, we apply results in \S \ref{sec-defm} to construct $2$-adic integral models of Shimura varieties of abelian type under certain assumptions (see Theorem \ref{thm-main}). The overall strategy follows that of \cite{kisin2018integral, kisin2024integralmodelsshimuravarieties}. We first treat the case of Shimura varieties of Hodge type and then extend to Shimura varieties of abelian type by finding suitable Hodge type lifts while closely following \cite{kisin2018integral}. In \S \ref{caseA} and \S \ref{sec-application}, we complete the proof of Theorem \ref{thm1} by verifying that the assumptions in Theorem \ref{thm-main} or Theorem \ref{thm-hodgemodel} are satisfied in Case (A) or (B).

In Appendix \ref{appsec}, we show that, for an unramified group $G$ over a $2$-adic field $F$, if a stabilizer group scheme $\CG$ satisfies $\CG(\CO_F)\sset H$ for some hyperspecial subgroup $H$ of $G(F)$, then $\CG$ can be written as the tame Galois fixed points of the Weil restriction of scalars of a reductive group scheme. This result is used in the construction of very good integral Hodge embeddings in Case (A).

Very often we will refer the readers to corresponding arguments in \cite{kisin2018integral,kisin2024integralmodelsshimuravarieties} that are similar or can be directly extended to the case $p=2$ without repeating the proofs.

\subsection*{Acknowledgments}
I would like to thank my advisor G. Pappas for suggesting this problem to me, and for his support and encouragement. I am especially grateful to him for his explanation of several arguments in \cite{kisin2018integral, kisin2024integralmodelsshimuravarieties}, as well as reading a draft of the paper. I thank K. Madapusi, P. Wu and R. Zhou for interesting conversations. I also thank the referees for their careful reading of the paper and for their valuable comments and suggestions. This work is partially supported by the Shuimu Tsinghua Scholar Program of Tsinghua University.

\subsection*{Notations}\label{subsec-notation}
Let $p$ be a prime number. If $R$ is a commutative (not necessarily unitary) ring, we denote by $W(R)$ the ring of ($p$-typical) Witt vectors of $R$, with ghost maps $w_n: W(R)\ra R$. We let $\wh{W}(R)$ denote the group of all elements $(a_0,a_1,\ldots)\in W(R)$ such that $a_i=0$ for almost all $i$.  For $x\in R$, we use $[x]\in W(R)$ to denote the \Teichmuller representative of $x$.  Let $k$ be a perfect field of characteristic $p$. For most of the paper, $k$ is an algebraic closure of a finite field. Set $W\coloneqq W(k)$ and $K_0\coloneqq \Frac W$. 

If $F/\BQ_p$ is a non-archimedean local field, we let $\ol{F}$ denote a fixed choice of algebraic closure of $F$. Denote by $\breve F$ the completion of the maximal unramified extension of $F$ in $\ol{F}$.

\section{\texorpdfstring{$p$}{p}-divisible groups and Lau's classification} \label{sec-lau}
In this section, we review Lau's work \cite{lau2014relations} on the classification of $2$-divisible groups in terms of \Dieudonne displays. We generalize the construction of the natural ``connection isomorphisms" for \Dieudonne pairs in \cite{kisin2024integralmodelsshimuravarieties} to the case $p=2$. We also compare Lau's classification of $p$-divisible groups with Breuil--Kisin's classfication. 

\subsection{\Dieudonne displays and \texorpdfstring{$p$}{p}-divisible groups}

\subsubsection{The Zink ring}
Let $(R,\fm_R,k)$ be an artinian local ring (or more generally an \dfn{admissible ring} in the sense of \cite[\S 1]{lau2014relations}) with residue field $k$. 
Denote by $W(R)$ its associated Witt ring equipped with Frobenius $\varphi$ and Verschiebung $V$. By \cite[\S 1B]{lau2014relations}, the exact sequence $$0\ra W(\fm_R)\ra W(R)\ra W(k)\ra 0$$ has a unique ring homomorphism section $s: W(k)\ra W(R)$, which is $\varphi$-equivariant.

\begin{defn}[{\cite{zink2001dieudonne}}]  \label{defn-zink}
	The \dfn{Zink ring} of $R$ is $\BW(R)=sW(k)\oplus \wh{W}(\fm_R)$, where $\wh{W}(\fm_R)\sset W(\fm_R)$ consists of elements $(x_0,x_1,\ldots)\in  W(\fm_R)$ such that $x_i=0$ for almost all $i$. 
\end{defn}

%Note that if $R$ is artinian with perfect residue field of characteristic $p$, the Zink ring $\BW(R)$ recovers the definition introduced in \cite{zink2001dieudonne} (but under the notation $\wh{W}(R)$). 
The Zink ring $\BW(R)$ is a $\varphi$-stable subring of $W(R)$. If $p=2$, $\BW(R)$ is in general not stable under the Verschiebung $V$. We need to modify $V$ as follows. The element $p-[p]\in W(\BZ_p)$ lies in the image of $V$ because it maps to zero in $\BZ_p$. Moreover, the element $V\inverse(p-[p])\in W(\BZ_p)$ is a unit, since it maps to $1$ in $W(\BF_p)$. Define \begin{flalign}
	   u_0\coloneqq \begin{cases}
	   	  V\inverse(2-[2])\quad &\text{if $p=2$},\\ 1 &\text{if $p\geq 3$}.
	   \end{cases}   \label{eq-u0}
\end{flalign}
The image of $u_0\in W(\BZ_p)\cross $ in $W(R)\cross$ is also denoted by $u_0$. For $x\in W(R)$, set $$\BV(x)\coloneqq V(u_0 x).$$  

\begin{lemma}[{\cite[Lemma 1.7]{lau2014relations}}] \label{modifiedV}
	The map $\BV: W(R)\ra W(R)$ satisfies $\BV(\BW(R))\sset\BW(R)$. In particular, $pu_0=\varphi V(u_0)=\varphi \BV(1)\in \BW(R)$. Moreover, there is an exact sequence $$0\ra \BW(R)\xrightarrow{\BV}\BW(R)\xrightarrow{w_0}R\ra 0. $$
\end{lemma}

\begin{remark}
	We will call the map $$\BV: \BW(R)\ra \BW(R)$$ the \dfn{modified Verschiebung} for $\BW(R)$. Many statements about $\BW(R)$ in the case $p=2$ are proven by adapting the corresponding proofs for $p>2$, with adjustments for the modified Verschiebung map.
\end{remark}

Now we recall the logarithmic coordinates of the Witt ring, see \cite[\S 1C]{lau2014relations}. Let $(S\ra R,\delta)$ be a divided power extension of rings  with kernel $\fa\sset S$. Denote by $\fa^\BN$ the additive group $\prod_{i\in\BN}\fa$, equipped with a $W(S)$-module structure $$x[a_0,a_1,\ldots]\coloneqq [w_0(x)a_0, w_1(x)a_1,\ldots] $$
for $x\in W(S)$ and $[a_0,a_1,\ldots]\in \prod_{i\in\BN}\fa$.
 Then the $\delta$-divided Witt polynomials $w_n'$ define an isomorphism of $W(S)$-modules \begin{flalign*}
	   \Log: W(\fa) &\simto \fa^\BN\\ \ud{a}=(a_0,a_1,\ldots) &\mapsto [w_0'(\ud{a}), w_1'(\ud{a}),\ldots]
\end{flalign*} where $w_n'(X_0,\ldots,X_n)=(p^n-1)!\delta_{p^n}(X_0)+ (p^{n-1}-1)!\delta_{p^{n-1}}(X_1)+\cdots+ X_n $. For $x\in W(\fa)$, we call $\Log(x)$ the \dfn{logarithmic coordinate} of $x$. In terms of logarithmic coordinates, the Frobenius and Verschiebung of $W(\fa)$ act on $\fa^\BN$ as \begin{flalign}
	    \varphi([a_0,a_1,\ldots])=[pa_1,pa_2,\ldots], \quad V([a_0,a_1,\ldots])=[0,a_0,a_1,\ldots]. \label{eq-Frob} 
\end{flalign} 
Moreover, $\Log$ induces an injective map $$\Log: \wh{W}(\fa)\hookrightarrow \fa^{(\BN)},$$ which is bijective when the divided powers $\delta$ are nilpotent. Here, the group $\wh{W}(\fa)$ denotes the set of elements $(a_0,a_1,\ldots)\in W(\fa)$ such that $a_i=0$ for almost all $i$, and $\fa^{(\BN)}\sset\fa^\BN$ denotes $\bigoplus_{i\in\BN}\fa$. The ideal $\fa\sset W(S)$ is by definition the set of elements whose logarithmic coordinates are of the form $[a,0,0,\ldots]$, $a\in \fa$. 

\begin{defn}
	For a (Noetherian) complete local ring $R$ with residue field $k$, we set $$\BW(R)\coloneqq \varprojlim_n \BW(R/\fm_R^n).$$ 
\end{defn}
For a complete local ring $R$, we can define the modified Verschibung $\BV$ on $\BW(R)$ by passing to the limit. Then $\BW(R)$ is a subring of $W(R)\coloneqq \varprojlim_nW(R/\fm_R^n)$, which is stable under $\varphi$ and $\BV$. We also have  $\BW(R)/\BV(\BW(R))\simeq R$, see \cite[\S 1E]{lau2014relations}. 
Note that $\BW(R)$ is $p$-adically complete by \cite[Proposition 1.14]{lau2014relations}. 

\subsubsection{Frames and windows}
Here, we introduce notions of frames and windows following \cite[\S 2]{lau2010frames} and \cite[\S 2]{lau2014relations}.
\begin{defn}
	\begin{enumerate}
	    \item A \dfn{frame} is a quintuple $\CF=(S,I,R,\sigma,\sigma_1)$, where $S$ and $R=S/I$ are rings, $\sigma: S\ra S$ is a ring endomorphism with $\sigma(a)\equiv a^p\mod pS$, $\sigma_1: I\ra S$ is a $\sigma$-linear map of $S$-modules whose image generates $S$ as an $S$-module, and $I+pS$ lies in the Jacobson radical of $S$. A frame is called a \dfn{lifting frame} if all projective $R$-modules of finite type can be lifted to projective $S$-modules. 
		\item A homomorphism of frames $$\alpha: \CF\lra \CF'=(S',I',R',\sigma',\sigma_1')$$ is a ring homomorphism $\alpha: S\ra S'$ with $\alpha(I)\sset I'$ such that $\sigma'\alpha=\alpha\sigma$ and $\sigma_1'\alpha=u\cdot \alpha \sigma_1$ for a unit $u\in S'$, which is then determined by $\alpha$. We say that $\alpha$ is a frame $u$-homomorphism. If $u=1$, then $\alpha$ is called \dfn{strict}.
		\item Let $\CF$ be a frame. A \dfn{window over $\CF$} (or $\CF$-window) is a quadruple $$\CP=(M,M_1,F,F_1),$$ where $M$ is a projective $S$-module of finite type with a submodule $M_1$ such that there exists a decomposition of $S$-modules $M=L\oplus T$ with $M_1=L\oplus IT$, called a \dfn{normal decomposition}, and where $F: M\ra M$ and $F_1: M_1\ra M$ are $\sigma$-linear maps of $S$-modules with $$F_1(ax)=\sigma_1(a)F(x)$$ for $a\in I$ and $x\in M$, and $F_1(M_1)$ generates $M$ as an $S$-module.
	\end{enumerate}
\end{defn}

\begin{remark} \label{rmk-frame}
	If $\CF$ is a lifting frame, then the existence of a normal decomposition in (3) of the above definition is equivalent to that $M/M_1$ is a projective $R$-module. A frame is a lifting frame if $S$ is local or $I$-adic, see \cite[Remark 2.4]{lau2010frames}. 
\end{remark}

% Given $(M,M_1)$ together with a normal decomposition $M=L\oplus T$, giving $\sigma$-linear maps $(F,F_1)$ which make $(M,M_1,F,F_1)$ a window $\CP$ over $\CF$ is equivalent to giving a $\sigma$-linear isomorphism $\Psi: L\oplus T\simto M$ defined by $F_1$ on $L$ and by $F$ on $T$. The triple $(L,T,\Psi)$ is called a \dfn{normal representation of $\CP$}. 

A $u$-homomorphism $\alpha: \CF\ra \CF'$ induces a base change functor \begin{flalign}
	\alpha_*: \text{(windows over $\CF$)}\lra \text{(windows over $\CF'$)} \label{eq-basechange}
\end{flalign}  from the category of windows over $\CF$ to the category of windows over $\CF'$. In terms of normal representations, the functor $\alpha_*$ is given by $$(L,T,\Psi)\mapsto (S'\otimes_SL, S'\otimes_ST, \Psi')$$ with $\Psi'(s'\otimes l)=u\sigma'(s')\otimes \Psi(l)$ and $\Psi'(s'\otimes t)=\sigma'(s')\otimes \Psi(t)$. 
\begin{defn}
	A frame homomorphism $\alpha: \CF\ra \CF'$ is called \dfn{crystalline} if the functor $\alpha_*$ is an equivalence of categories.
\end{defn}

Note that for a frame $\CF=(S,I,R, \sigma,\sigma_1)$, there is a unique element $\theta\in S$ such that $\sigma(a)=\theta\sigma_1(a)$ for all $a\in I$. For an $S$-module $M$, we write $M^{(\sigma)}=S\otimes_{\sigma,S}M$. Then for a window $\CP=(M,M_1,F,F_1)$ over $\CF$, by \cite[Lemma 2.3]{lau2014relations}, there exists a unique $S$-linear map \begin{flalign}
	   V^\sharp: M\lra M^{(\sigma)}  \label{eq-Vsharp}
\end{flalign}  such that $V^\sharp (F_1(x))=1\otimes x$ for $x\in M_1$. It satisfies $F^\# V^\sharp=\theta$ and $V^\sharp F^\#=\theta$, where $F^\#: M^{(\sigma)}\ra M $ is the linearization of $F$.

\begin{example} \label{eg-frame}
	For a complete local ring $R$ with perfect residue field, we will be interested in the following (lifting) frames (see \cite[Lemma 2.5, \S 6]{lau2014relations} and Remark \ref{rmk-frame}):  
    \begin{enumerate}
		\item the \dfn{\Dieudonne frame} \[ \CD_R\coloneqq  (\BW(R), \BI_R, R, \varphi, \varphi_1),\] where $\BI_R=\ker (w_0: \BW(R)\ra R)$, $\varphi: \BW(R)\ra \BW(R)$ is the Frobenius map, and $\varphi_1:  \BI_R\ra \BW(R)$ is the inverse of $\BV$;
		\item assume $R=\CO_K$ for some finite extension $K$ of $\BQ_p$ with residue field $k$, choose a presentation $R=\fS/E\fS$, where $\fS=W(k)[[u]]$ and $E\in \fS$ is an Eisenstein polynomial with constant term $p$. Define the \dfn{Breuil--Kisin frame} \[\CB\coloneqq (\fS, E\fS, R, \varphi,\varphi_1),\] where $\varphi: \fS\ra \fS$ acts on $W(k)$ as usual Frobenius and sends $u$ to $u^p$, and $\varphi_1(Ex)\coloneqq \varphi(x)$ for $x\in \fS$.
	\end{enumerate}
    Unfortunately, we use the same symbol $\varphi$ for the Frobenius on $\BW(R)$ and on $\fS$ (as in \cite{kisin2018integral,kisin2024integralmodelsshimuravarieties}). The intended meaning should be clear from context.
\end{example}

\subsubsection{\Dieudonne displays and \Dieudonne pairs} \label{Dpairs}
%Note that when $R$ has odd residue characteristic, i.e., $p>2$, windows over $\CD_R$ are the same as the \Dieudonne displays over $R$ used in \cite[3.1.3]{kisin2018integral}, and the ring $\BW(R)$ here is denoted by $\wh{W}(R)$ in \loccit More explicitly, 
Let $R$ be a complete local ring with perfect residue field of characteristic $p$. By Remark \ref{rmk-frame},  a window over $\CD_R$ (also called a \dfn{\Dieudonne display over $R$} later) is a tuple $(M,M_1,F,F_1)$, where \begin{enumerate}[label=(\roman*)]
	\item $M$ is a finite free $\BW(R)$-module,
	\item $M_1\sset M$ is a $\BW(R)$-submodule such that $\BI_RM\sset M_1\sset M$ and $M/M_1$ is a projective $R$-module, 
	\item $F: M\ra M$ is a $\varphi$-linear map,
	\item $F_1: M_1\ra M$ is a $\varphi$-linear map, whose image generates $M$ as a $\BW(R)$-module, and which satisfies \begin{flalign}
		    F_1(\BV(w)m)=wF(m)   \label{eq-FF1}
              \end{flalign}
              for any $w\in\BW(R)$ and $m\in M$.
	\end{enumerate}
\begin{remark}\label{invertp}
   \begin{enumerate}
       \item For $p>2$, windows over $\CD_R$ are the same as the \Dieudonne displays over $R$ used in \cite[3.1.3]{kisin2018integral}, and the ring $\BW(R)$ here is denoted by $\wh{W}(R)$ in \textit{loc. cit.}.
       \item Since $p\in \BI_R$, we have $M_1[1/p]=M[1/p]$, and $F,F_1$ become bijections after inverting $p$.
   \end{enumerate}
\end{remark}
%	Note that any finite projective module over $\BW(R)$ (a local ring) is free. 
	For a \Dieudonne display $(M,M_1,F,F_1)$, by taking $w=1$ and $m\in M_1$ in the equation \eqref{eq-FF1}, we get \begin{flalign*}
		   F(m)=F_1(\BV(1)m)=\varphi\BV(1)F_1(m) =pu_0 F_1(m).
	\end{flalign*}  Recall that $u_0\in W(R)\cross$ is defined by (\ref{eq-u0}). In particular, we can consider the condition \begin{enumerate}
		\item [(iv$^\prime$)] $F_1: M_1\ra M$ is a $\varphi$-linear map, whose image generates $M$ as a $\BW(R)$-module, and which satisfies \begin{flalign*}
		    F_1(\BV(w)m)=wpu_0F_1(m)   
              \end{flalign*}
              for any $w\in\BW(R)$ and $m\in M_1$.
	\end{enumerate}

Let $\wt{M}_1$ be the image of the homomorphism $$\varphi^*(i): \varphi^*M_1=\BW(R)\otimes_{\varphi,\BW(R)}M_1 \ra \varphi^*M=\BW(R)\otimes_{\varphi,\BW(R)}M$$ induced by the inclusion $i: M_1\hookrightarrow M$. Note that $\wt{M}_1$ and the notion of a normal decomposition depend only on $M$ and $M_1$, not on $F$ and $F_1$. 

\begin{lemma}
	\label{lem-Psi} Suppose $\BW(R)$ is $p$-torsion free (e.g. if $R$ is $p$-torsion free, or $pR=0$ and $R$ is reduced).  \begin{enumerate}
	    \item Giving a \Dieudonne display  $(M,M_1, F,F_1)$ over $R$ is the same as giving $(M,M_1,F_1)$ satisfying (i), (ii) and (iv$^\prime$).  In this case, we also refer to the tuple $(M,M_1,F_1)$ as a \Dieudonne display over $R$.
		\item For a \Dieudonne display $(M,M_1,F_1)$ over $R$, the linearization $F_1^\#$ of $F_1$ factors as $$\varphi^*M_1\ra \wt{M}_1\xrightarrow{\Psi}M$$ with $\Psi$ a $\BW(R)$-module isomorphism.
		\item Given an isomorphism $\Psi: \wt{M}_1\ra M$, there exists a unique \Dieudonne display $(M,M_1,F_1)$  over $R$, which produces the given $(M,M_1,\Psi)$ via the construction in $(2)$. 
	\end{enumerate}
\end{lemma}
\begin{proof}
	The proof closely follows \cite[\S 3.1.3, Lemma 3.1.5]{kisin2018integral}, with adjustments for the modified Verschiebung $\BV$. We take this lemma as an example to illustrate how we modify the arguments concerning \Dieudonne displays in \cite{kisin2018integral} to deal with the case $p=2$. 
	
	(1) Given the tuple $(M,M_1,F_1)$, set $F(m)\coloneqq F_1(\BV(1)m)$ for $m\in M$. Clearly $F: M\ra M$ is $\varphi$-linear. Then for $w\in\BW(R)$ and $m\in M$, we have $$pu_0F_1(\BV(w)m)=F_1(\BV(1)\BV(w)m)= F(\BV(w)m)=\varphi\BV(w)F(m)=pu_0wF(m). $$ Since $u_0\in W(R)\cross$ and $\BW(R)$ is $p$-torsion free, we obtain that $\BW(R)$ is $(pu_0)$-torsion free, and hence $$F_1(\BV(w)m)=wF(m).$$ In particular, $(M,M_1,F,F_1)$ is a \Dieudonne display.
	
	(2) Let $M=L\oplus T$ be a normal decomposition for $M$. Since $\varphi(\BI_R)=pu_0\BW(R)$ and $\BW(R)$ is $pu_0$-torsion free, we have $$\wt{M}_1=\varphi^*(L)\oplus pu_0\varphi^*(T)\simeq \BW(R)^d,$$ where $d=\rk_{\BW(R)} M$. Firstly, we show that $F_1^\#$ factors through $\wt{M}_1$. Let $K$ denote the kernel of $\varphi^*(i): \varphi^*M_1\ra \varphi^*M$. Note that $F|_{M_1}=pu_0 F_1$, and so $pu_0 F_1^\#=F^\#\circ\varphi^*(i)$. In particular, $pu_0 F_1^\#$ vanishes on $K$. Since $\BW(R)$ is $pu_0$-torsion free, we conclude that $F_1^\#$ vanishes on $K$, and hence $F_1^\#$ factors through $\wt{M}_1$. Since $F_1^\#$ is surjective by definition, we obtain a surjective map $\Psi: \wt{M}_1\ra M$ between free $\BW(R)$-modules of the same rank. Hence, $\Psi$ is an isomorphism. 
	
	(3) Define $F_1: M_1\ra M$ by $$F_1(m_1)\coloneqq \Psi(1\otimes m_1),$$ where $1\otimes m_1$ denotes the image of $1\otimes m_1\in \BW(R)\otimes_{\varphi,\BW(R)}M_1=\varphi^*M_1$ in $\varphi^*M$. Then $F_1$ is clearly $\varphi$-linear and its linearization $F_1^\#$ is surjective. Thus, we obtain a \Dieudonne display $(M,M_1,F_1)$.
\end{proof}

\begin{defn}[{\cite[\S 1.1]{hoff2023parahoric}}]
    Let $R$ be a complete local ring.
	\begin{enumerate}
		\item A \dfn{\Dieudonne pair of type $(n,d)$} over $R$ is a pair $(M,M_1)$ of $\BW(R)$-modules such that $M$ is a finite free $\BW(R)$-module of rank $n$, $M_1$ is a $\BW(R)$-submodule of $M$ and $M/M_1$ is a finite free $R$-module of rank $d$. Sometimes, we simply say that $(M,M_1)$ is a \Dieudonne pair. 
		\item A morphism between two \Dieudonne pairs $(M,M_1)$ and $(M',M_1')$ is a homomorphism of $\BW(R)$-modules $f: M\ra M'$ such that $f(M_1)\sset M_1'$.
	\end{enumerate}
\end{defn}

Let $R\ra R'$ be a local morphism of complete local rings. Let $(M,M_1)$ be a \Dieudonne pair over $R$. The \dfn{base change} $(M',M_1')$ of $(M,M_1)$ to $R'$ is characterized by (cf. \cite[3.1.6]{kisin2018integral}) \begin{flalign*}
	M'= M\otimes_{\BW(R)}\BW(R') \text{\ and\ } M_1'/\BI_{R'}M'= (M_1/\BI_RM)\otimes_RR'.
\end{flalign*}

\begin{lemma}   \label{lem-pairs}
	There exists a functor $$\CF\colon (M,M_1)\mapsto \wt{M}_1,$$ from the category of \Dieudonne pairs over $R$ of type $(n,d)$ to the category of finite free $\BW(R)$-modules of rank $n$, such that $\CF$ is compatible with base change in $R$ and there is a natural isomorphism $\wt{M}_1[1/p]=(\varphi^*M)[1/p]$. If $\BW(R)$ is $p$-torsion free, then $\wt{M}_1$ is given by the construction in Lemma \ref{lem-Psi}.
\end{lemma}
\begin{proof}(cf. \cite[\S 5.1.2]{kisin2024integralmodelsshimuravarieties}.)
	  Let $(M,M_1)$ be a \Dieudonne pair of type $(n,d)$. Choose a normal decomposition $M=L\oplus T$ and a basis $\sB=(e_1,\ldots,e_n)$ of $M$ such that $(e_1,\ldots,e_d)$ is a basis of $L$ and $(e_{d+1},\ldots,e_n)$ is a basis of $T$. Such a basis $\sB$ is said to be \dfn{adapted} to the normal decomposition $M=L\oplus T$.  Set $$\CF((M,M_1))=\wt{M}_1=(\varphi^*L)\oplus(\varphi^*T),$$ which is a free $\BW(R)$-module of rank $n$. We denote by $\wt{\sB}=(\varphi^*{e}_1,\ldots,\varphi^*{e}_n)$ the basis of $\wt{M}_1$. 
	  
	  Let $(M',M_1')$ be a second \Dieudonne pair with a normal decomposition $M'=L'\oplus T'$ and an adapted basis $\sB'=(e_1',\ldots,e_n')$. Let $f$ be  a morphism between $(M,M_1)$ and $(M',M_1')$. Using the normal decompositions, we may express $f$ as a block matrix $$\begin{pmatrix}
	  	    A &B\\ C &D
	  \end{pmatrix}\in M_n(\BW(R))$$ with respect to the bases $\sB$ and $\sB'$, where the entries of $C$ are in $\BI_R$. Then we define $\CF(f)$ to be the morphism $\wt{f}: \wt{M}_1\ra \wt{M}'_1$ given by the block matrix $$\begin{pmatrix}
	  	  \varphi(A) &pu_0\varphi(B)\\ \BV\inverse(C) &\varphi(D)
	  \end{pmatrix}$$
	  in terms of the bases $\wt{\sB}$ and $\wt{\sB}'$. Using $pu_0\BV\inverse=\varphi$, it is straightforward to check that $\CF$ is a well-defined functor. By construction, $\CF$ is compatible with base change in $R$. 
	  
	  There is a natural isomorphism \begin{flalign*}
	  	      \wt{M}_1[1/p]=(\varphi^*L)[1/p] \oplus(\varphi^*T)[1/p]&\simto (\varphi^*M)[1/p]=(\varphi^*L)[1/p] \oplus(\varphi^*T)[1/p]\\ l+t &\mapsto l +pu_0t.
	  \end{flalign*}
	  When $\BW(R)$ is $p$-torsion free, the above isomorphism restricts to an injective map \begin{flalign*}
	  	      \wt{M}_1\hookrightarrow \varphi^*M,
	  \end{flalign*}
	  and we recover the construction of $\wt{M}_1$ in Lemma \ref{lem-Psi}.
\end{proof}

\begin{lemma}[{cf. \cite[Lemma 5.1.15]{kisin2024integralmodelsshimuravarieties}}]  \label{lem-connection}
    Let $R$ be a complete local ring with residue field $k$. Suppose that $\BW(R)$ is $p$-torsion free. 
    Let $(M,M_1)$ be a \Dieudonne pair over $R$ with reduction $(M_0,M_{0,1})$ over $k$. 
	Set $\fa_R\coloneqq \fm_R^2+pR$. Then there exists a natural isomorphism $$c: \wt{M}_{0,1}\otimes_{W(k)}\BW(R/\fa_R) \simto \wt{M}_1\otimes_{\BW(R)}\BW(R/\fa_R) , $$ which is called the ``connection isomorphism", fitting into a canonical commutative diagram $$\xymatrix{
         \wt{M}_1\otimes_{\BW(R)}\BW(R/\fa_R)\ar[r]  &\varphi^*(M_{R/\fa_R})\ar@{=}[d] \\ \wt{M}_{0,1}\otimes_{W(k)}\BW(R/\fa_R)\ar[r]\ar[u]^{c}_{\simeq} &\varphi^*(M_0)\otimes_{W(k)}\BW(R/\fa_R),
}$$ where $M_{R/\fa_R}\coloneqq M\otimes_{\BW(R)}\BW(R/\fa_R)$ and horizontal maps are induced by taking the base change of the natural maps $\wt{M}_{0,1}\ra \varphi^*(M_0)$ and $\wt{M}_1\ra\varphi^*(M)$.
\end{lemma}
\begin{proof}
	Using Lemma \ref{lem-pairs}, the construction of $c$ and the proof of \cite[Lemma 5.1.15]{kisin2024integralmodelsshimuravarieties} (replacing $V\inverse$ by $\BV\inverse$) also work for $p=2$.
\end{proof}

\begin{remark}\label{rmk-basechange}
    The connection isomorphism $c$ in Lemma \ref{lem-connection} is compatible with base change in $R$ in the following sense.
	Let $R\ra R'$ be a local homomorphism of complete local rings lifting the identity map on the same residue field $k$. Let $(M,M_1)$ be a \Dieudonne pair as above. Denote by $(M',M_1')$ the base change to $R'$. In particular, $(M',M_1')$ also has reduction $(M_0,M_{0,1})$ over $k$. Since the functor $\CF$ in Lemma \ref{lem-pairs} is compatible with base change in $R$, we have the canonical identification \begin{flalign*}
		\wt{{M'}} \simeq \wt{M}\otimes_{\BW(R)}\BW(R').
	\end{flalign*}
	Write $c_R$ for the connection isomorphism constructed by Lemma \ref{lem-connection}. Using the above identification, the base change $c\otimes_RR'$ is the connection isomorphism $c_{R'}$. This follows from the construction of $c$ in the proof of \cite[Lemma 5.1.15]{kisin2024integralmodelsshimuravarieties}; see also the explicit description of $c$ in Equation (5.1.18) of \textit{loc. cit.}.
\end{remark}

\subsubsection{Lau's classification}
One of the main results in \cite{lau2014relations} is the following.
\begin{thm}
	\label{thm-classify} Let $R$ be a complete local ring with perfect residue field of characteristic $p$. 
	\begin{enumerate}[label=(\arabic*)]
		\item There is an anti-equivalence of exact categories \begin{flalign*}
		    \Theta_R: \text{($p$-divisible groups over $R$) }\simto \text{(\Dieudonne displays over $R$) },
	\end{flalign*} which is compatible with base change in $R$. 
	    \item For any $p$-divisible group $\sG$ over $R$, there is a natural isomorphism $$\Theta_R(\sG)/\BI_R\Theta_R(\sG)\simeq \BD(\sG)(R),$$ where $\BD(\sG)$ denotes the contravariant \Dieudonne crystal of $\sG$.
	    \item Let $\sG$ be a $p$-divisible group over $R$. Write $\Theta_R(\sG)=(M,M_1,F,F_1)$. The Hodge filtration of $\Theta_R(\sG)$ is defined as \[M_1/\BI_RM\sset M/\BI_RM.\]
	        Then the isomorphism in (2) respects the Hodge filtrations on both sides. 
	\end{enumerate}  
\end{thm}
\begin{remark}
	For $p>2$, the functor $\Theta_R$ recovers the anti-equivalence used in \cite[3.1.7]{kisin2018integral} by sending a $p$-divisible group $\sG$ over $R$ to $\BD(\sG)(\BW(R))$. Note that when $p>2$, $\BW(R)\ra R$ has divided powers on $\BI_R$ by \cite[Lemma 1.16]{lau2014relations}. For $p=2$, $\Theta_R$ is not as explicit as in the case $p>2$, but see the case when $R$ is a ring of $p$-adic integers in \S \ref{subsubsec-compBK}.
\end{remark}
\begin{proof}
	 (1) For any $p$-divisible group $\sG$ over $R$, set $$\Theta_R(\sG)\coloneqq \Phi_R(\sG^*),$$ where $\sG^*$ denotes the Cartier dual of $\sG$ and $\Phi_R$ denotes the equivalence in \cite[Corollary 5.4]{lau2014relations}. Then we see that $\Theta_R$ is an anti-equivalence of exact categories. It commutes with base change in $R$ by \cite[Theorem 3.9, 4.9]{lau2014relations}.
	 
	 (2) and (3) follow from \cite[Corollary 3.22, 4.10]{lau2014relations}. Note that we use \dfn{contravariant} \Dieudonne crystals following \cite{kisin2018integral}, while Lau uses \dfn{covariant} \Dieudonne crystals in \cite{lau2014relations}. One can switch between contravariant and covariant \Dieudonne crystals by taking Cartier duals.
	 
\end{proof}

\subsection{Comparison with Breuil--Kisin's classification}
\label{subsubsec-compBK}
Here the notation is as in Example \ref{eg-frame} (2). In particular, we let $R$ denote the ring $\CO_K$ of integers for some finite extension $K$ of $\BQ_p$ with residue field $k$. Let $\pi$ be a uniformizer of $\CO_K$ satisfying $E(\pi)=0$. Then there is a Frobenius-equivariant ring homomorphism $$\kappa: \fS=W(k)[[u]] \ra  W(\CO_K)$$ sending $u$ to $[\pi]$, lifting the quotient map $\fS\ra \CO_K$. Here $[\cdot]$ denotes the \Teichmuller map $\CO_K\ra W(\CO_K)$. Moreover, the image of $\kappa$ lies in $\BW(\CO_K)$, see \cite[Remark 6.3]{lau2014relations}. Recall that $\CB$ denotes the Breuil--Kisin frame in Example \ref{eg-frame} (2). By \cite[Theorem 6.6]{lau2014relations}, $\kappa$ induces a crystalline homomorphism $$\kappa: \CB\ra \CD_{\CO_K}.$$ That is, the induced functor $\kappa_*$ as in (\ref{eq-basechange})  gives an equivalence \begin{flalign*}
	    \kappa_*: \text{(windows over $\CB$)}\simto \text{(windows over $\CD_{\CO_K}$)}=\text{(\Dieudonne displays over $\CO_K$)}.
\end{flalign*}

Denote by $\mathrm{BT}_\fS^\varphi$ the category of Breuil--Kisin modules $(\fM,\varphi_\fM)$ of $E$-height one, i.e., $\fM$ is a finite free $\fS$-module and $\varphi_\fM: \varphi^*\fM\ra \fM $ is an $\fS$-module homomorphism whose cokernel is killed by $E$. 

\begin{lem}
	There is an equivalence $\CF: \mathrm{BT}_\fS^\varphi \lra \text{(windows over $\CB$)}$.  
\end{lem}
\begin{proof}
	The lemma is implicitly contained in \cite[\S 6, 7]{lau2014relations} (see also \cite[\S 2]{kim20162}). To a Breuil--Kisin module $(\fM,\varphi_\fM)$ in $\mathrm{BT}_\fS^\varphi$, we can associate a triple $(M, M_1, F_1)$, where $M\coloneqq \varphi^*\fM$; $M_1\coloneqq \fM$, viewed as a submodule of $M$ via the unique map $V_\fM: \fM\ra \varphi^*\fM$ whose composition with $\varphi_\fM$ is the multiplication by $E(u)$; and $F_1: M_1\ra M$ is given by $x\in\fM\mapsto 1\otimes x\in \varphi^*\fM$. Then we see \begin{flalign}
		   E(u)M\sset M_1\sset M.    \label{incl}
	\end{flalign} 
	Define $F: M\ra M$ by sending $m\in M$ to $F_1(E(u)m)$. Then $(M,M_1,F,F_1)$ defines a window over $\CB$. Hence, we obtain a functor $$\CF: \mathrm{BT}_\fS^\varphi\lra \text{(windows over $\CB$)}.$$ The functor $\CF$ is an equivalence (cf. \cite[Lemma 8.2, 8.6]{lau2010frames}). Its inverse can be described as follows. Let $(M,M_1,F,F_1)$ be a window over $\CB$. The $\fS$-module $M_1$ is necessarily free, and hence the surjection $F_1^\#: \varphi^*M_1\ra M$ is an isomorphism. Let $\phi: M_1\hookrightarrow \varphi^*M_1$ denote the composition of the inclusion $M_1\hookrightarrow M$ with the inverse of $F_1^\#$. There is a unique $\fS$-linear map $\psi: \varphi^*M_1\ra M_1$ such that $\psi\phi=E(u)$. Then $(M_1,\psi)$ defines an object in $\mathrm{BT}_\fS^\varphi$ and the functor $(M,M_1,F,F_1)\mapsto (M_1,\psi)$ is the inverse of $\CF$.
\end{proof}

The contravariant functor \begin{flalign*}
    \fM(-)\coloneqq \CF\inverse\circ\kappa_*\inverse\circ\Theta_{\CO_K}: \text{($p$-divisible groups over $\CO_K$)}\simto \mathrm{BT}_\fS^\varphi.  
\end{flalign*}
is an anti-equivalence by Theorem \ref{thm-classify}; see \cite[Theorem 2.12]{kim20162}. By the proof of \cite[Theorem 2.12]{kim20162}, for a $p$-divisible group $H$ over $\CO_K$, the Breuil--Kisin module $\fM(H)$ is canonically isomorphic to $\fM^{\text{Kisin}}(T_p(H)^\vee)$, where $T_p(H)^\vee\coloneqq \Hom_{\BZ_p}(T_p(H),\BZ_p)$ and $\fM^{\text{Kisin}}$ is defined in \cite[Theorem 1.2.1]{kisin2010integral}.

% \begin{remark}
% 	The anti-equivalence of the category of $p$-divisible groups over $\CO_K$ and $\mathrm{BT}_\fS^\varphi$ is due to \cite{lau2014relations}. See also independent proofs in \cite{kim2012classification} and \cite{liu2013correspondence}.  
% \end{remark}

\begin{defn}
	For $(\fM,\varphi_{\fM})\in \mathrm{BT}_\fS^\varphi$, the Hodge filtration of $\varphi^*\fM$ is defined as \[\fM/E(u)\varphi^*\fM \sset \varphi^*\fM/E(u)\varphi^*\fM,\]
	where the inclusion is induced by \eqref{incl}.
\end{defn}

\begin{corollary}\label{coro-Mequi}
	Let $\sG$ be a $p$-divisible group over $\CO_K$. 
	\begin{enumerate}
		\item There exists a natural isomorphism  \begin{flalign*}
	   \Theta_{\CO_K}(\sG)\simeq \varphi^* \fM(\sG) \otimes_{\fS,\kappa} \BW(\CO_K)
\end{flalign*} as \Dieudonne displays over $\CO_K$.
        \item There exists a natural isomorphism \begin{flalign*}
        	    \BD(\sG)(\CO_K)\simeq \varphi^*\fM(\sG)\otimes_{\fS}\CO_K=\varphi^*\fM/E(u)\varphi^*\fM,
        \end{flalign*}
        which respects the Hodge filtrations on both sides. 
	\end{enumerate}   
\end{corollary}
\begin{proof}
	(1) It follows from the equality $\Theta_{\CO_K}=\kappa_*\circ\CF\circ\fM(-)$ and the definition of base change of \Dieudonne displays.  
	
	(2) Denote by $\psi$ the isomorphism in (1). By base change of $\psi$ along the natural surjection $\BW(\CO_K)\ra \CO_K$, we obtain an isomorphism \[\Theta_{\CO_K}(\sG)/\BI_{\CO_K}\Theta_{\CO_K}(\sG) \simeq \varphi^*\fM(\sG)/E(u)\varphi^*\fM.\]
	    Since $\psi$ is an isomorphism of \Dieudonne displays, the above isomorphism respects the Hodge filtrations. By Theorem \ref{thm-classify} (2) and (3), we obtain an isomorphism \begin{flalign*}
        	    \BD(\sG)(\CO_K)\simeq \varphi^*\fM(\sG)\otimes_{\fS}\CO_K=\varphi^*\fM/E(u)\varphi^*\fM
        \end{flalign*} respecting the Hodge filtrations. 

\end{proof}

\section{Deformation theory} \label{sec-defm}
In this section, we extend the deformation theory of $p$-divisible groups in \cite[\S 3]{kisin2018integral} to the case $p=2$. We also generalize the notion of very good Hodge embeddings for $p=2$, allowing us to construct universal deformation of $p$-divisible groups with crystalline tensors (see Proposition \ref{prop-torsor}). In Proposition \ref{prop-adapted}, we establish a criterion for determining when a deformation is $(\CG_W,\mu_y)$-adapted in the sense of Definition \ref{defn-adapted}.

\subsection{Universal deformations of \texorpdfstring{$p$}{p}-divisible groups} \label{subsec-versaldef}
The notations are as in \S \ref{sec-lau}. In this subsection, we aim to extend the construction of the universal deformation space of $p$-divisible groups in \cite[\S 3.1]{kisin2018integral} to the case $p=2$. 

Firstly we generalize \cite[Theorem 3,4]{zink2001dieudonne}, which deals with the case when $R$ has residue characteristic $p>2$ or $2R=0$. 
\begin{thm}\label{thm-34}
	 Let $k$ be a perfect field of characteristic $p$. Let $(S\ra R, \delta)$ be a nilpotent divided power extension of artinian local rings of residue field $k$, i.e., the kernel $\fa$ of the surjection $S\ra R$ is equipped with nilpotent divided powers $\delta$. 
	\begin{enumerate}[label=(\arabic*)]
		\item Let $\CP=(M,M_1,F,F_1)$ be a \Dieudonne display over $S$ and $\ol{\CP}=(\ol{M},\ol{M}_1,F,F_1)$ be the reduction of $\CP$ over $R$. Denote by $\wh{M}_1$ the inverse image of $\ol{M}_1$ under the homomorphism $$M\ra \ol{M}=\BW(R)\otimes_{\BW(S)}M.$$ Then $F_1:M_1\ra M$ extends uniquely to a $\BW(S)$-module homomorphism $$\wh{F}_1: \wh{M}_1\ra M$$ such that $\wh{F}_1 (\fa M)=0$. Therefore, $\wh{F}_1$  restricted to $\wh{W}(\fa)M$ is given by $$\wh{F}_1([a_0,a_1,\ldots]x)=[w_0(u_0\inverse)a_1, w_1(u_0\inverse)a_2,\ldots]F(x)$$ in logarithmic coordinates. 
		\item Let $\CP=(M,M_1,F,F_1)$ (resp. $\CP'=(M',M'_1,F',F_1')$) be a \Dieudonne display over $S$. Denote by $\ol{\CP}$ (resp. $\ol{\CP}'$) the reduction of $\CP$ (resp. $\CP'$) over $R$. Assume that $\ol{u}: \ol{\CP}\ra \ol{\CP}'$ is a morphism of \Dieudonne displays over $R$. Then there exists a unique morphism of quadruples $$u: (M,\wh{M}_1, F, \wh{F}_1) \ra (M',\wh{M}'_1, F', \wh{F}'_1)  $$ lifting $\ol{u}$. Hence, we can associate a crystal to a \Dieudonne display as follows: Let $\CP=(M,M_1,F,F_1)$ be a \Dieudonne display over $R$, $(T\ra R, \delta)$ be a divided power extension, then define the \dfn{\Dieudonne crystal} $\BD(\CP)$ evaluated at $(T\ra R,\delta)$ as $$\BD(\CP)(T)\coloneqq T\otimes_{w_0,\BW(T)}\wt{M}, $$ where $\wt{\CP}=(\wt{M}, \wt{M}_1, \wt{F}, \wt{{F}}_1)$ is any lifting of $\CP$ over $T$. 
		\item Let $\CC$ be the category of all pairs $(\CP, Fil)$, where $\CP$ is a \Dieudonne display over $R$ and $Fil\sset \BD(\CP)(S)$ is a direct summand lifting the Hodge filtration $M_1/\BI_R M\hookrightarrow M/\BI_R M$ of $\BD(\CP)(R)$. Then the category $\CC$ is canonically isomorphic to the category of \Dieudonne displays over $S$. 
	\end{enumerate}
\end{thm}
\begin{remark}
	The above theorem has a reformulation in terms of \dfn{relative \Dieudonne displays} as in \cite[\S 2D, 2F]{lau2014relations}: the quadruple $(M,\wh{M}_1,F,\wh{F}_1)$ defines a window over the \dfn{relative \Dieudonne frame} $\CD_{S/R}$. 
\end{remark}

\begin{proof}
    The proof adapts arguments in \cite[Theorem 3,4]{zink2001dieudonne} and \cite[Lemma 38, 42]{zink2002display}, with adjustments for $\BV$.
    
	(1) Choose a normal decomposition $M=L\oplus T$. Then $$\wh{M}_1=\wh{W}(\fa)M+M_1=\fa T\oplus L\oplus \BI_ST.$$ Here $\fa T$ is considered as a $\BW(S)$-submodule via the logarithmic coordinates. Using this decomposition, we can extend $F_1$ by setting $\wh{F}_1(\fa T)=0$.  We claim that $\wh{F}_1(\fa L)=0$. Note that by formula (\ref{eq-Frob}), we have $\varphi(\fa)=0$. Since $F_1$ is $\varphi$-linear, we have $\wh{F}_1(\fa L)=\varphi(\fa)\wh{F}_1(L)=0$. Thus, the extension $\wh{F}_1$ satisfies $\wh{F}_1(\fa M)=0$. It is unique since $\wh{M}_1=\wh{W}(\fa)M+M_1=\fa M+M_1$. For any $[a_0,a_1,\ldots]\in \wh{W}(\fa)$ and $x\in M$, we have \begin{flalign*}
		   \wh{F}_1([a_0,a_1,\ldots]x) &=\wh{F}_1([a_0,0,0,\ldots]x)+\wh{F}_1(V[a_1,a_2,\ldots]x) \\ &=0+ F_1(\BV(u_0\inverse[a_1,\ldots])x)= F_1(\BV([w_0(u_0\inverse)a_1, w_1(u_0\inverse)a_2,\ldots])x) \\ &= [w_0(u_0\inverse)a_1, w_1(u_0\inverse)a_2,\ldots] F(x).
	\end{flalign*}

	(2) For the uniqueness of $u$, it is enough to consider the case $\ol{u}=0$. Recall that for a \Dieudonne display $(M,M_1,F,F_1)$ over $S$, we have defined the map $V^\sharp: M\ra \BW(S)\otimes_{\varphi,\BW(S)}M$ in (\ref{eq-Vsharp}). For any integer $N\geq 1$, we define $(V^N)^\sharp: M\ra M\otimes_{\varphi^N,\BW(S)}M$ as the composite \begin{flalign*}
		    M\xrightarrow{V^\sharp} \BW(S)\otimes_{\varphi,\BW(S)}M \xrightarrow{1\otimes V^\sharp} \BW(S)\otimes_{\varphi^2,\BW(S)}M\ra\cdots\ra \BW(S)\otimes_{\varphi^N,\BW(S)}M.
	\end{flalign*} Similarly, we can define maps $(F_1^N)^\#$ and $(\wh{F}_1^N)^\#$.
	As in the proof of \cite[Theorem 3]{zink2001dieudonne}, we have a commutative diagram $$\xymatrix{
	     M\ar[d]_{(V^N)^\sharp} \ar[r]^-u &\wh{W}(\fa)M' \\ \BW(S)\otimes_{\varphi^N,\BW(S)}M \ar[r]^-{1\otimes u} &\BW(S)\otimes_{\varphi^N,\BW(S)}\wh{W}(\fa) M' \ar[u]_{(\wh{F}_1^{'N} )^\#}
	}$$  By (1), for $[a_0,a_1,\ldots]\in \wh{W}(\fa)$ and $x\in M'$, we have $$\wh{F}_1^{'N}([a_0,\ldots]x)=[\prod_{i=0}^{N-1}w_i(u_0\inverse) a_N,\prod_{i=1}^Nw_i(u_0\inverse) a_{N+1},\ldots]F'^{N}(x).$$ Since $a_i=0$ for almost all $i$, $\wh{W}(\fa)M'$ is annihilated by $\wh{F}_1^{'N}$ for sufficiently large $N$. This shows $u=0$ as desired. 
	
	For the existence of $u$, we can repeat the proof of \cite[Theorem 3]{zink2001dieudonne}.  
	
	(3) Clearly we can get a lifting of the Hodge filtration of $\BD(\CP)(R)$ from a \Dieudonne display over $S$.  On the other hand, given $(\CP, Fil)\in\CC$, any lifting of $\CP$ to $S$ gives a unique quadruple $(M,\wh{M}_1,F,\wh{F}_1)$ by (2).
	 Let $M_1\sset \wh{M}_1$ be the inverse image of $Fil\sset M/\BI_SM $ under the projection $M\ra M/\BI_S{M}$, then we obtain a \Dieudonne display $(M,{M}_1,{F}, \wh{F}_1|_{{M}_1})$ over $S$. By (2), these two constructions are mutually inverse.
\end{proof}

Now we fix a $p$-divisible group $\sG_0$ over $k$, and let $(\BD, \BD_1,F,F_1)$ be the corresponding \Dieudonne display. Note that $\BD$ is given by $\BD(\sG_0)(W)$, see \cite[Corollary 2.34]{lau2014relations}. By Lemma \ref{lem-Psi}, the \Dieudonne display $(\BD,\BD_1,F,F_1)$ corresponds to a triple $(\BD,\BD_1,\Psi_0)$ for an isomorphism $\Psi_0: \wt{\BD}_1\simto \BD$. Next we will construct a universal deformation space of $\sG_0$, equivalently a universal deformation space of the \Dieudonne display $(\BD,\BD_1,\Psi_0)$.

Recall there is a canonical Hodge filtration on $\BD\otimes_W k=\BD(\sG_0)(k)$: $$0\ra \Hom_k(\Lie\sG_0, k)\ra \BD\otimes_W k \ra \Lie{\sG_0^*}\ra 0.$$ We think of $\BD\otimes_W k$ as a filtered $k$-module by setting $\Fil^0(\BD\otimes_W k)=\BD\otimes_W k$, $\Fil^1(\BD\otimes_W k)= \Hom_k(\Lie\sG_0, k)$.
This filtration corresponds to a parabolic subgroup $P_0\sset \GL(\BD\otimes_W k)$. Fix a lifting of $P_0$ to a parabolic subgroup $P\sset \GL(\BD)$. Write \begin{flalign}
	    M^\loc=\GL(\BD)/P \text{\ and\ } \wh{M}^\loc=\Spf R,  \label{eq-Mloc}
\end{flalign} where $\wh{M}^\loc$ is the completion of $\GL(\BD)/P$ along the image of the identity in $\GL(\BD\otimes_W k)$. Then $R$ is a power series ring over $W$. 

Set $M=\BD\otimes_W\BW(R)$, and let $\ol{M}_1\sset M/\BI_RM$ be the direct summand corresponding to the parabolic subgroup $gPg\inverse\sset \GL(\BD)$ over $\wh{M}^\loc$, where $g\in (\GL(\BD)/P)(R)$ is the universal point. Let $M_1\sset M$ be the preimage of $\ol{M}_1$ in $M$ and $\Psi:\wt{M}_1\overset{\sim}{\ra} M $ be a $\BW(R)$-module isomorphism reducing to $\Psi_0$ modulo $\fm_R$, where $\wt{M}_1$ is defined as in Lemma \ref{lem-Psi}. Then the triple $$(M,M_1,\Psi)$$ gives a \Dieudonne display over $R$ reducing to $(\BD,\BD_1,\Psi_0)$. By Theorem \ref{thm-classify}, the \Dieudonne display $(M,M_1,\Psi)$ corresponds to a $p$-divisible group $\sG_R$ over $R$, which is a deformation of $\sG_0$. 

Set $\fa_R\coloneqq \fm_R^2+pR$. By Lemma \ref{lem-connection}, there exists a natural connection isomorphism $$c: \wt{\BD}_1\otimes_W\BW(R/\fa_R)\simto \wt{M}_1\otimes_{\BW(R)}\BW(R/\fa_R).$$

\begin{defn}
	The map $\Psi$ is said to be \dfn{constant modulo $\fa_R$} if the composite map $$\wt{\BD}_1\otimes_W\BW(R/\fa_R)\xrightarrow{c} \wt{M}_1\otimes_{\BW(R)}\BW(R/\fa_R)\xrightarrow{\Psi\otimes 1} M_{R/\fa_R}\simeq \BD\otimes_W\BW(R/\fa_R) $$ is equal to $\Psi_0\otimes 1$. 
\end{defn} 
\begin{lemma} \label{lem-versal}
	If $\Psi$ is constant modulo $\fa_R$, then the deformation $\sG_R$ of $\sG_0$ is universal.
\end{lemma}
\begin{proof}
	Recall that there exists a universal deformation ring $R^\univ$ for $\sG_0$, which is a power series ring over $W$ of the same dimension as $R$, cf. \cite[Proposition 3.11]{lau2014relations}. The deformation $\sG_R$ is induced by a map $R^\univ\ra R$. We want to show this is an isomorphism. It suffices to prove that the induced map on tangent spaces is an isomorphism. 
	
	We have two \Dieudonne displays over $R/\fa_R$. One is obtained from $(M,M_1,F_R,F_{R,1})$ (the \Dieudonne display corresponding to $(M,M_1,\Psi)$) by the base change along $R\ra R/\fa_R$, the other is obtained from $(\BD,\BD_1,F,F_1)$ by the base change along $k\ra R/\fa_R$.  

	If $\Psi$ is constant modulo $\fa_R$, then as in the proof of \cite[Lemma 3.1.12]{kisin2018integral}, we know $\wh{F}_{R,1}=\wh{F}_{1}$ on $\wh{M}_{R/\fa_R,1}$, see the notation in Theorem \ref{thm-34}. Hence, these two \Dieudonne displays give rise to the same quadruple $$(M_{R/\fa_R}, \wh{M}_{R/\fa_R,1},F_{R/\fa_R}, \wh{F}_{R/\fa_R,1}).$$
	
	Let $\sG$ be a deformation of $\sG_0$ over the ring $k[\epsilon]$ of dual numbers. Since $k[\epsilon] \ra k$ has trivial divided powers, it is a nilpotent divided power extension, then by Theorem \ref{thm-34} (1) and (2), the base change of $(\BD,\BD_1,F,F_1)$ along $k\ra k[\epsilon]$ gives rise to a quadruple $(M_{k[\epsilon]},\wh{M}_{k[\epsilon],1},F_{k[\epsilon]}, \wh{F}_{k[\epsilon],1})$. By the proof of Theorem \ref{thm-34} (3), the \Dieudonne display corresponding to $\sG$ is of the form $$ (M_{k[\epsilon]},\wt{\Fil},F_{k[\epsilon]}, \wh{F}_{k[\epsilon],1}),$$ where $\wt{\Fil}\sset\wh{M}_{k[\epsilon],1}$ is the preimage of a certain lifting $\Fil\sset (\BD\otimes_Wk)\otimes_kk[\epsilon]$ of the Hodge filtration of $\BD$. Since the filtration $\ol{M}_1\sset \BD\otimes_W R$ is induced by the universal point in $(\GL(\BD)/P)(R)$ by construction, there is a map $\alpha: R\ra k[\epsilon]$ (which necessarily factors through $R/\fa_R$) such that the induced map $\BD\otimes_W R\ra \BD\otimes_Wk[\epsilon]$ sends $\ol{M}_1$ to $\Fil$. Then by the discussion in the previous paragraph, $(M_{k[\epsilon]},\wt{\Fil},F_{k[\epsilon]}, \wh{F}_{k[\epsilon],1})$ is the base change of $(M,M_1,F_R,F_{R,1})$ along $\alpha$. Thus, $\sG$ is the base change of $\sG_R$ along $\alpha$. In particular, $R^\univ\ra R$ induces an isomorphism of tangent spaces.  Hence, we proved that $\sG_R$ is universal.
\end{proof}

\begin{lemma}\label{formsmooth}
	Denote by $\CF\coloneqq \ud{\Isom}(\wt{M}_1,M)$ the functor of isomorphisms of finite free $\BW(R)$-modules between $\wt{M}_1$ and $M$. Then $\CF$ is a $\GL(M)$-torsor over $\BW(R)$. Then the natural map $$\CF(\BW(R))\ra  \CF(\BW(R/\fa_R))$$ is surjective. In particular, an isomorphism $\Psi$, which is constant modulo $\fa_R$, always exists.
\end{lemma}
\begin{proof}
	It is clear that $\CF$ is a $\GL(M)$-torsor. In particular, $\CF$ is affine and (formally) smooth over $\BW(R)$. Since $\fa_R\sset \fm_R$ and $R$ is a Noetherian complete local ring, we have $R=\varprojlim R/\fa_R^n$ by \cite[090T]{stacks-project}.  Note that $\BW(R)=\varprojlim\BW(R/\fa_R^n)$ by definition. By the universal property of inverse limits, we have \begin{flalign*}
		\CF(\BW(R))=\Hom(\CO(\CF),\varprojlim \BW(R/\fa_R^n))=\varprojlim\Hom(\CO(\CF),\BW(R/\fa_R^n))=\varprojlim\CF(\BW(R/\fa_R^n)),
	\end{flalign*}  
	where $\CO(\CF)$ denotes the coordinate ring of $\CF$ with $\Spec\CO(\CF)=\CF$.
	Since each transition map $\BW(R/\fa_R^{n+1})\twoheadrightarrow \BW(R/\fa_R^n)$ is a surjection with nilpotent kernel, 
	by the (formal) smoothness of $\CF$, we obtain a surjection $\CF(\BW(R))\twoheadrightarrow \CF(\BW(R/\fa_R))$.
\end{proof}

%\begin{remark}\label{formsmooth}
%	Note that the functor $\CF\coloneqq \ud{\Isom}(\wt{M}_1,M)$ of isomorphisms of finite free $\BW(R)$-modules between $\wt{M}_1$ and $M$ is a $\GL(M)$-torsor over $\BW(R)$. In particular, $\CF$ is affine and (formally) smooth over $\BW(R)$. Since $\fa_R\sset \fm_R$ and $R$ is a Noetherian complete local ring, we have $R=\varprojlim R/\fa_R^n$ by \cite[090T]{stacks-project}.  Note that $\BW(R)=\varprojlim\BW(R/\fa_R^n)$ by definition. By the universal property of inverse limits, we have \begin{flalign*}
%		\CF(\BW(R))=\Hom(\CO(\CF),\varprojlim \BW(R/\fa_R^n))=\varprojlim\Hom(\CO(\CF),\BW(R/\fa_R^n))=\varprojlim\CF(\BW(R/\fa_R^n)),
%	\end{flalign*}  
%	where $\CO(\CF)$ denotes the coordinate ring of $\CF$ with $\Spec\CO(\CF)=\CF$.
%	Since each transition map $\BW(R/\fa_R^{n+1})\twoheadrightarrow \BW(R/\fa_R^n)$ is a surjection with nilpotent kernel, 
%	by the (formal) smoothness of $\CF$, we obtain a surjection $\CF(\BW(R))\twoheadrightarrow \CF(\BW(R/\fa_R))$. This implies that an isomorphism $\Psi$, which is constant modulo $\fa_R$, always exists. 
%\end{remark}

\subsection{Local models and local Hodge embeddings} \label{subsec-localmod}
Before discussing the deformation of $p$-divisible groups with crystalline tensors, we will make a digression into local models and local Hodge embeddings in this subsection.

%\subsubsection{Local Hodge embeddings}
Let $p$ be a prime number and $F$ be a finite field extension of $\BQ_p$ or $\breve\BQ_p$. 
%Here, we denote by $\breve K$ the (completion of) maximal unramified field extension of some finite extension $K$ of $\BQ_p$.  
For a connected reductive group $G$ over $F$, let $\CB(G,F)$ denote the associated (extended) Bruhat--Tits building, which carries an action of $G(F)$.  For $x\in \CB(G,F)$, the associated \dfn{Bruhat--Tits stabilizer group scheme} $\CG_x$, in the sense of \cite{bruhat1984groupes}, is a smooth affine group scheme  over $\CO_F$ such that the generic fiber of $\CG_x$ is $G$ and $\CG_x(\CO_{ F})$ is the stabilizer subgroup of $x$ in $G(F)$. By definition, the neutral component $\CG_x^\circ$ is the \dfn{parahoric group scheme} associated to $x$.  Recall that a smooth affine group scheme $\CG$ over $\CO_F$ is \dfn{quasi-parahoric} if the neutral component of $\CG$ is a parahoric group scheme and $\CG_x^\circ(\CO_{\breve F})\sset \CG(\CO_{\breve F})\sset \CG_x(\CO_{\breve F})$ for some Bruhat--Tits stabilizer group scheme $\CG_x$, see \cite[\S 2.2]{Pappas_Rapoport_2025}.
\begin{defn}
	A \dfn{local model triple} over $F$ is a triple $(G,\cbra{\mu},\CG)$, where $G$ is a connected reductive group over $F$, $\cbra{\mu}$ is the $G(\ol{F})$-conjugacy class of a \dfn{minuscule} cocharacter $\mu: \BG_{m,\ol{F}}\ra G_{\ol{F}}$, and $\CG$ is a quasi-parahoric group scheme for $G$.
\end{defn}

We will often write $(\CG,\mu)$ (resp. $(G,\mu)$) for $(G,\cbra{\mu},\CG)$ (resp. $(G,\cbra{\mu})$). A morphism of local model triples $(\CG,\mu)\ra (\CG',\mu')$ is a group scheme homomorphism $\CG\ra \CG'$ taking $\tcbra{\mu}$ to $\tcbra{\mu'}$. 

Let $(G,\cbra{\mu},\CG)$ be a local model triple over $F$. Denote by $E$ the reflex field of $\tcbra{\mu}$. As conjectured by Scholze--Weinstein \cite[Conjecture 21.4.1]{scholze2020berkeley} and proven in \cite[Theorem 1.2]{anschutz2022p}, \cite[Corollary 1.4]{gleason2022tubular}, we have the following theorem.
\begin{thm} \label{thm-localmodels}
	 There exists a unique projective flat normal $\CO_E$-scheme $\BM^\loc_{\CG,\mu}$, which represents the v-sheaf local model in the sense of Scholze--Weinstein. Furthermore, the special fiber of $\BM^\loc_{\CG,\mu}$ is reduced.
\end{thm}
 Note that the generic fiber of $\BM^\loc_{\CG,\mu}$ is isomorphic to the homogeneous space $X_{G,\mu}\coloneqq G/P_\mu$ over $E$, where $P_\mu$ is the parabolic subgroup defined by \begin{flalign*}
	    P_\mu\coloneqq \tcbra{g\in G\ |\ \lim_{t\ra \infty} \mu(t)g\mu(t)\inverse \text{\ exists} }.
\end{flalign*}
We also have $\BM^\loc_{\CG,\mu}=\BM^\loc_{\CG^\circ,\mu}$ by \cite[Proposition 21.4.3]{scholze2020berkeley}. By functoriality,  any morphism $(\CG,\mu)\ra (\CG',\mu')$ of local model triples induces a natural morphism $\BM^\loc_{\CG,\mu}\ra \BM^\loc_{\CG',\mu'}$ of local models. 

\begin{defn}[{\cite[Definition 3.1.2]{kisin2024integralmodelsshimuravarieties}}]  \label{scalars}
   Let $(G,\cbra{\mu},\CG)$ be a local model triple over $F$.
   \begin{enumerate}
   	\item A pair $(G,\mu)$ is of \dfn{(local) Hodge type} if there is a closed immersion $\rho: G\hookrightarrow \GL(V)$, where $V$ is an $F$-vector space of dimension $h$, such that \begin{enumerate}[label=(\roman*)]
		\item $\rho$ is a minuscule representation in the sense of \cite[\S 1.2.9]{kisin2018integral}.
		\item $\rho\circ\mu$ is conjugate to the standard minuscule cocharacter $\mu_d$ of $\GL(V_{\ol{F}})$, where $$\mu_d(t)\coloneqq \diag(t^{(d)},1^{(h-d)}),\ t\in \ol{F}.$$
		\item $\rho(G)$ contains the scalars.
	\end{enumerate}
	Such a $\rho$ will be said to give a (local) Hodge embedding $\rho: (G,\mu)\hookrightarrow (\GL(V),\mu_d)$. 
	\item An \dfn{integral Hodge embedding} for $(\CG,\mu)$ is a closed immersion $\rho: \CG\hookrightarrow \GL(\Lambda)$ over $\CO_F$, where $\Lambda$ is a finite free $\CO_F$-module, such that the base change $\rho\otimes_{\CO_F}F$ is a Hodge embedding for $(G,\mu)$.
   \end{enumerate}
\end{defn}

\begin{lemma} \label{lem-closure}
	Let $(G,\cbra{\mu},\CG)$ be a local model triple over $F$. Suppose $\rho: (\CG,\mu)\hookrightarrow (\GL(\Lambda),\mu_d)$ is an integral Hodge embedding. Then $\rho$ induces a closed immersion $$X_{G,\mu}=G/P_\mu\hookrightarrow X_{\GL(V),\mu_d}\otimes_{F}E=\Gr(d,V)_E,$$ where $\Gr(d,V)$ denotes the Grassmannian classifying subspaces of $V$ of rank $d$. Let $\ol{X}_{G,\mu}$ be the (reduced) Zariski closure of $X_{G,\mu}\sset \Gr(d,V)_E$ in $\Gr(d,\Lambda)_{\CO_E}$. 
	
	If $\ol{X}_{G,\mu}$ is normal, then $\ol{X}_{G,\mu}$ is isomorphic to $\BM^\loc_{\CG,\mu}$ and the embedding $\ol{X}_{G,\mu}\hookrightarrow \Gr(d,\Lambda)_{\CO_E}$ is identified with the natural morphism $\BM^\loc_{\CG,\mu}\ra \BM^\loc_{\GL(\Lambda),\mu_d}\otimes_{\CO_F}\CO_E$ induced by $\rho$. 
\end{lemma}
\begin{proof}
	The proof of \cite[Lemma 3.4.1]{kisin2024integralmodelsshimuravarieties} goes through without assuming that $\ol{X}_{G,\mu}$ has reduced special fiber. In fact, once the isomorphism $\ol{X}_{G,\mu}\simeq \BM^\loc_{\CG,\mu}$ has been established, the reducedness of the special fiber follows from \cite[Corollary 1.4]{gleason2022tubular}.
\end{proof}

%\subsubsection{Very good Hodge embeddings}

\begin{defn}[{\cite[Definition 3.4.4]{kisin2024integralmodelsshimuravarieties}}]  \label{goodrho}
	Let $\rho: (\CG,\mu)\hookrightarrow (\GL(\Lambda),\mu_d)$ be an integral Hodge embedding over $\CO_F$. We say that $\rho$ is a \dfn{good} Hodge embedding, if the morphism $$\BM^\loc_{\CG,\mu}\lra \BM^\loc_{\GL(\Lambda),\mu_d}\otimes_{\CO_F}\CO_E$$
	induced by $\rho$ is a closed immersion. 
\end{defn}
By Lemma \ref{lem-closure}, $\rho$ is good if the Zariski closure of $X_{G,\mu}$ in $\Gr(d,\Lambda)_{\CO_E}$ is normal. 

From now on, we suppose that $F/\BQ_p$ is unramified and $\rho: (\CG,\mu)\hookrightarrow (\GL(\Lambda),\mu_d)$ is a good integral Hodge embedding over $\CO_F$. In particular, we have a closed immersion $\BM^\loc_{\CG,\mu}\hookrightarrow \Gr(d,\Lambda)_{\CO_E}$.  

For any $x\in \BM^\loc_{\CG,\mu}(k)$, where $k=\ol{\BF}_p$, we let $R_G=R_{G,x}$ (resp. $R_E$) denote the completion of $\BM^\loc_{\CG,\mu}$ (resp. $\Gr(d,\Lambda)_{\CO_E}$) at $x$. By our assumptions, $R_E$ is isomorphic to a power series ring over $\CO_EW(k)$ and $R_G$ is a (normal) quotient ring of $R_E$. Then $\BW(R_E)$ and $\BW(R_G)$ are $p$-torsion free rings. Set $$M\coloneqq\Lambda\otimes_{\CO_F} \BW(R_E).$$ Let $\ol{M}_1\sset M/\BI_{R_E}M=\Lambda\otimes_{\CO_F}R_E$ be the direct summand corresponding to the universal $R_E$-valued point of $\Gr(d,\Lambda)$. Set $$M_1\coloneqq \text{\ the preimage of $\ol{M}_1$ in $M$. }$$
Then $(M,M_1)$ is a \Dieudonne pair over $R_E$. By the base change along $R_E\twoheadrightarrow R_G$, we obtain a \Dieudonne pair $(M_{R_G},M_{R_G,1})$ over $R_G$. By Lemma \ref{lem-pairs}, we can associate a free $\BW(R_G)$-module $\wt{M}_{R_G,1}$ with $$\wt{M}_{R_G,1}[1/p]=(\varphi^*M_{R_G})[1/p].$$

\begin{defn} \label{defn-tensors}
	For any ring $A$ and a finite free $A$-module $N$, we denote by $N^\otimes$ the direct sum of all $A$-modules which can be formed from $N$ by using the operations of taking tensor products, duals, symmetric and exterior powers. If $N$ is equipped with a filtration, then $N^\otimes$ is equipped with a filtration accordingly. 
	
	If $(s_\alpha)\sset N^\otimes$ and $G\sset \GL(N)$ is the pointwise stabilizer of $s_\alpha$, we say $G$ is the group scheme \dfn{cut out} by the tensors $s_\alpha$.
\end{defn}

\begin{lemma}[{\cite[Proposition 1.3.2]{kisin2010integral}}] \label{lem-tensors}
	 Suppose that $A$ is a discrete valuation ring of mixed characteristic and $N$ is a finite free $A$-module. If $G \sset \GL(N)$ is a closed $A$-flat subgroup whose generic fiber is reductive, then $G$ is cut by a finite collection of tensors in $N^\otimes$.
\end{lemma}
\begin{remark}
	By an argument of Deligne (\cite[\S 3.2.1]{pappas2023integral}), the tensors in Lemma \ref{lem-tensors} can be taken in the submodule $\oplus_{m,n\geq 0}N^{\otimes m}\otimes_A(N^\vee)^{\otimes n}$. Here, $N^\vee$ denotes the $A$-dual module $\Hom_A(N,A)$. 
\end{remark}

Let $\rho: \CG\hookrightarrow \GL(\Lambda)$ be a Hodge embedding. Then  $\CG\sset \GL(\Lambda)$ (via $\rho$) is cut out by a set of tensors $(s_\alpha)\sset\Lambda^\otimes$ by Lemma \ref{lem-tensors}. Set \begin{flalign*}
	   \wt{s}_\alpha\coloneqq s_\alpha\otimes 1\in \Lambda^\otimes\otimes_{\CO_F}\varphi^*\BW(R_G)=(\varphi^*M_{R_G})^\otimes.
\end{flalign*}
We may view $(\wt{s}_\alpha)$ as tensors in $(\varphi^*M_{R_G})^\otimes[1/p]=\wt{M}_{R_G,1}^\otimes[1/p]$. By \cite[\S 5.2.1, p. 53]{kisin2024integralmodelsshimuravarieties} (and \cite[Corollary 3.2.11]{kisin2018integral}), we have the following proposition. 

\begin{prop}\label{prop-tens}
	Suppose that $F/\BQ_p$ is unramified and $\rho: (\CG,\mu)\hookrightarrow (\GL(\Lambda),\mu_d)$ is a good integral Hodge embedding over $\CO_F$. Then $\wt{s}_\alpha\in \wt{M}_{R_G,1}^\otimes$. 
\end{prop}

Denote by $\wt{s}_{\alpha,0}$ the reduction of $\wt{s}_\alpha$ in $\wt{M}_{0,1}^\otimes$, where $\wt{M}_{0,1}=\wt{M}_{R_G,1}\otimes_{\BW(R_G)}W(k)$. By Lemma \ref{lem-connection}, we have a connection isomorphism \begin{flalign}
		   c_\CG: \wt{M}_{0,1}\otimes_{W(k)}\BW(R_G/\fa_{R_G})\simto \wt{M}_{R_G,1}\otimes_{\BW(R_G)}\BW(R_G/\fa_{R_G}).  \label{eqcG}
	\end{flalign}
	
\begin{defn}\label{defn-very}
	Under the assumptions in Proposition \ref{prop-tens}, we say that $\rho$ is \dfn{very good} at $x\in \BM^\loc_{\CG,\mu}(k)$, if $c_\CG(\wt{s}_{\alpha,0}\otimes 1)=\wt{s}_{\alpha}\otimes 1$. In this case, we say that the tensors $(\wt{s}_\alpha)$ are \dfn{horizontal} at $x$.
	
	We say $\rho$ is a very good (integral) Hodge embedding if $\rho$ is very good at every $x\in \BM^\loc_{\CG,\mu}(k)$. 
\end{defn} 

We will use the following lemma in \S \ref{sec-application}. Recall that (see \cite[Definition 4.1.4]{kisin2024integralmodelsshimuravarieties}) for a scheme $X$ over $k$ and $x\in X(k)$, we say that the tangent space $T_xX$ of $X$ at $x$ is spanned by smooth formal curves if the images of the tangent spaces by $k$-morphisms $\Spec k[[t]]\ra X$ with the closed point mapping to $x$ generate the $k$-vector space $T_xX$.

\begin{lemma}[{\cite[Proposition 5.3.11]{kisin2024integralmodelsshimuravarieties}}] \label{lem-verygood}
	Assume $\rho: (\CG,\mu)\hookrightarrow (\GL(\Lambda),\mu_d)$ is a good integral Hodge embedding over $\BZ_p$. Let $x\in \BM^\loc_{\CG,\mu}(k)$ be a closed point. If the tangent space of the special fiber $\BM^\loc_{\CG,\mu}\otimes_{\CO_E}k$ at $x$ is spanned by smooth formal curves, then $\rho$ is very good at $x$. 
\end{lemma}

We refer to \cite[\S 5.3]{kisin2024integralmodelsshimuravarieties} for more properties of very good Hodge embeddings.

\subsection{Deformations with crystalline tensors} \label{subsec-crystensor}
We continue to use the notation in \S \ref{subsec-versaldef}, and as in \cite[\S 3.2, 3.3]{kisin2018integral}, we may assume $k$ is algebraically closed for simplicity. 

For a \Dieudonne display $(N,N_1,F,F_1)$ over a complete local ring with perfect residue field of characteristic $p$, the Frobenius $F$ induces a map $N[1/p]^\vee\ra N[1/p]^\vee$ sending $f\in N^\vee$ to the map $(x\mapsto \varphi f(F\inverse(x)), x\in N$. Here we use that $F$ becomes a bijection on $N[1/p]$; see Remark \ref{invertp} (2). Then $F$ induces a map $F[1/p]^\otimes: (N[1/p])^\otimes\ra (N[1/p])^\otimes.$  We say that a tensor $t$ in $N^\otimes\sset (N[1/p])^\otimes$ is \dfn{Frobenius invariant} if $t$ is fixed by $F[1/p]^\otimes$; cf. \cite[(1.5.4), footnote]{kisin2010integral}.

As in \S \ref{subsec-versaldef}, we fix a $p$-divisible group $\sG_0$ over $k$, which corresponds to a \Dieudonne display $(\BD,\BD_1,F,F_1)$ or equivalently a triple $(\BD,\BD_1,\Psi_0:\wt{\BD}_1\simto\BD)$ with $\BD=\BD(\sG_0)(W)$. Let $(s_{\alpha,0})\sset \BD^\otimes$ be a collection of Frobenius invariant tensors whose images in $\BD(\sG_0)(k)^\otimes$ lie in $\Fil^0 \BD(\sG_0)(k)^\otimes $.
We assume the following conditions (cf. \cite[\S 3.2]{kisin2024independenceellfrobeniusconjugacy}): \begin{enumerate}
	\item [(A1)] there exists a free $\BZ_p$-module $\Lambda$ with tensors $(s_\alpha)\sset \Lambda^\otimes$ and an isomorphism $\Lambda\otimes_{\BZ_p}W\simeq \BD$ such that each $s_{\alpha,0}\in \BD^\otimes$ is the image of $s_\alpha$;
	\item [(A2)] the stabilizer group scheme $\CG\sset\GL(\Lambda)$ cut out by $(s_{\alpha})\sset\Lambda^\otimes$ has reductive generic fiber $G$ and $\CG=\CG_x$ is a Bruhat--Tits stabilizer group scheme associated to some $x\in \CB(G,\BQ_p)$.
\end{enumerate}  

Note that the base change $\CG_W\coloneqq \CG\otimes_{\BZ_p}W\sset \GL(\BD)$ is cut out by $(s_{\alpha,0})\sset\BD^\otimes$. 
In \eqref{eq-Mloc} of \S \ref{subsec-versaldef}, we have defined $M^\loc$ 	and $\wh{M}^\loc=\Spf R$. Let $K'/K_0$ be a finite extension and $y: R\ra K'$ be a map such that $s_{\alpha,0}\in \Fil^0(\BD\otimes_W K')^\otimes$ for the filtration induced by $y$ on $\BD\otimes_W K'$. By \cite[Lemma 1.4.5]{kisin2010integral}, the filtration is induced by a $G$-valued cocoharacter $\mu_y$. We further impose the following assumption:
\begin{enumerate}
	\item [(A3)] there is a very good Hodge embedding $(\CG,\mu_y\inverse)\hookrightarrow (\GL(\Lambda),\mu_d)$ for $d=\dim_k\Lie{\sG_0}$.
\end{enumerate}

Denote by $E\sset K'$ the local reflex field of the $G$-conjugacy class of cocharacters $\{\mu_y\}$. Write $M_{G,y}^\loc$ for the closure of the $G$-orbit $G.y\sset M^\loc\otimes_WE$ in $M^\loc\otimes_W\CO_E$. By assumption (A3) and Lemma \ref{lem-closure}, the scheme $M_{G,y}^\loc$ is isomorphic to the local model $\BM^\loc_{\CG,\mu_y\inverse}$ attached to the local model triple $(G,\tcbra{\mu_y\inverse},\CG)$, and hence $M^\loc_{G,y}$ is normal and only depends on the $G$-conjugacy class $\cbra{\mu_y\inverse}$ (not on $y$). 
We denote by $\wh{M}_{G,y}^\loc=\Spf R_G$ the completion of $M_{G,y}^\loc$ along (the image of) the identity in $\GL(\BD\otimes_W k)$. Then $R_G$ is a normal quotient ring of $R\otimes_W\CO_E$. 

Recall in \S \ref{subsec-versaldef}, we constructed a universal deformation $\sG_R$ over $R$ corresponding to a \Dieudonne display $(M,M_1,\Psi)$ over $R$ with $M=\BD\otimes_W\BW(R)$, where $\Psi$ is constant modulo $\fa_R$. Set $$M_{R_E}\coloneqq  M\otimes_{\BW(R)}\BW(R_E),\quad M_{R_G}\coloneqq M_{R_E}\otimes_{\BW(R_E)}\BW(R_G).$$ The tensors $(s_{\alpha,0})\sset \BD^\otimes$ induce, by base change, tensors in  $M_{R_G}^\otimes$, which we still denote by $(s_{\alpha,0})$. As in Proposition \ref{prop-tens}, we have $(s_{\alpha,0})\sset \wt{M}_{R_G,1}^\otimes$. Set $$\fa_{R_E}\coloneqq \fm_{R_E}^2+\pi_ER_E \text{\ and\ }\fa_{R_G}\coloneqq \fm_{R_G}^2+\pi_ER_G, $$ where $\pi_E\in \CO_E$ is a uniformizer. In particular, $R_E/\fa_{R_E}\simeq R/\fa_R$.   

%Denote by $$c_\CG: \wt{\BD}_1\otimes_W\BW(R_G/\fa_{R_G})\simto \wt{M}_{R_G,1}\otimes_{\BW(R_G)}\BW(R_G/\fa_{R_G}), $$ the base change of the connection isomorphism  \begin{flalign*}
%	    c: \wt{\BD}_1\otimes_W\BW(R/\fa_{R})\simto \wt{M}_1\otimes_{\BW(R)}\BW(R/\fa_R)=\wt{M}_{R_E,1}\otimes_{\BW(R_E)}\BW(R_E/\fa_{R_E}).
%\end{flalign*} constructed in Lemma \ref{lem-connection}. We assume that: 
%
%\begin{enumerate}
%	\item [(A5)] The isomorphism $c_{\CG}$ respects tensors $s_{\alpha,0}$ on both sides. 
%\end{enumerate}
%As in \cite[\S 5.2]{kisin2024integralmodelsshimuravarieties}, we say that $(s_{\alpha,0})$ are ``horizontal" if the assumption (A5) holds. 

\begin{prop}
	[{cf. \cite[\S 3.2.12]{kisin2018integral}}]\label{prop-torsor} Assume (A1) to (A3).  \begin{enumerate}
		\item The scheme $$\CT\coloneqq \ud{\Isom}_{(s_{\alpha,0})}(\wt{M}_{R_G,1}, M_{R_G} )$$ consisting of isomorphisms respecting tensors $s_{\alpha,0}$ is a trivial $\CG$-torsor over $\BW(R_G)$.
		\item There exists an isomorphism $\Psi_{R_G}: \wt{M}_{R_G,1}\simto M_{R_G}$ respecting $s_{\alpha,0}$ which lifts to an isomorphism $\Psi_{R_E}: \wt{M}_{R_E,1}\ra M_{R_E}$ that is constant modulo $\fa_{R_E}$. Moreover, the $p$-divisible group $\sG_{R_E}$ over $R_E$ corresponding to the \Dieudonne display $(M_{R_E},M_{R_E,1},\Psi_{R_E})$ is a universal deformation of $\sG_0$. 
	\end{enumerate} 
\end{prop}
\begin{proof}
	(1) This follows from \cite[Corollary 3.2.11]{kisin2018integral} and \cite[Proposition 10.3]{anschutz2022extending}.
	
	(2) By Lemma \ref{lem-Psi} (2), the linearization $F_1^\#$ factors as $\varphi^*\BD_1\ra \wt{\BD}_1\xrightarrow{\Psi_0}\BD$. Since $\varphi$ is an isomorphism in this case, we obtain $\varphi^*\BD_1=\wt{\BD}_1\sset \varphi^*\BD$ and $F^\#_1$ is an isomorphism. By definition, we have $F(m)=pu_0F_1(m)$ for $m\in \BD_1$. Then \begin{flalign*}
	    F_1^\#[1/p] = (pu_0)\inverse F^\#[1/p]: (\varphi^*\BD_1)[1/p]=\wt{\BD}[1/p]\ra \BD[1/p].
	\end{flalign*}  
    By assumption (A3), $G$ contains the scalars (see Definition \ref{scalars} (iii)). Hence, $(s_{\alpha,0})$ are fixed by $(pu_0)\inverse$.
    Since $(s_{\alpha,0})$ are Frobenius invariant, $F_1^\#[1/p]$, and hence $\Psi_0$, preserves the tensors $(s_{\alpha,0})$. 
    Consider the isomorphism $\Psi_{R_G/\fa_{R_G}}$ \begin{flalign*}
\wt{M}_{R_G,1}\otimes_{\BW(R_G)}\BW(R_G/\fa_{R_G})\xrightarrow[\eqref{eqcG}]{c\inverse_\CG} \wt{\BD}_1\otimes_W\BW(R_G/\fa_{R_G}) \xrightarrow{\Psi_0\otimes 1}\BD\otimes_W\BW(R_G/\fa_{R_G})=M_{R_G/\fa_{R_G}}.
	\end{flalign*}
    By assumption (A3) and Definition \ref{defn-very}, $c_\CG\inverse$ preserves the tensors $s_{\alpha,0}$. Hence $\Psi_{R_G/\fa_{R_G}}$ defines a point in $\CT(\BW(R_G/\fa_{R_G}))$. Since $\CT$ is a $\CG$-torsor, we can lift the point to a point in $\CT(\BW(R_G))$; cf. Lemma \ref{formsmooth}. Then we obtain an isomorphism $\Psi_{R_G}: \wt{M}_{R_G,1}\simto M_{R_G}$ respecting $s_{\alpha,0}$. Let \begin{flalign*}
    	c\colon \wt{\BD}_1\otimes_W\BW(R_E/\fa_{R_E})\simto \wt{M}_{R_E,1}\otimes_{\BW(R_E)}\BW(R_E/\fa_{R_E})
    \end{flalign*} denote the connection isomorphism obtained from Lemma \ref{lem-connection}. By Remark \ref{rmk-basechange}, $c_\CG$ is the reduction of $c$ induced by the reduction $R_E\twoheadrightarrow R_G$. Hence, $\Psi_{R_G/\fa_{R_G}}$ is the reduction of the isomorphism $\Psi_{R_E/\fa_{R_E}}$ \begin{flalign*}
		    \wt{M}_{R_E,1}\otimes_{\BW(R_E)}\BW(R_E/\fa_{R_E})\xrightarrow{c\inverse} \wt{\BD}_1\otimes_W\BW(R_E/\fa_{R_E}) \xrightarrow{\Psi_0\otimes 1}\BD\otimes_W\BW(R_E/\fa_{R_E})=M_{R_E/\fa_{R_E}} .
	\end{flalign*} 	
	 Write $R_G=R_E/I_G$. Then $R_G/\fa_{R_G}\simeq R_E/(I_G+\fa_{R_E}).$ By the following Lemma \ref{lemzink}, we have $$\BW(R_G)\times_{\BW(R_G/\fa_{R_G})}\BW(R_E/\fa_{R_E})\simeq \BW((R_E/I_G)\times_{R_G/\fa_{R_G}}R_E/\fa_{R_E})\simeq\BW(R_E/(I_G\cap \fa_{R_E})). $$ 
	Denote by $\CF$ the $\GL(M_{R_E})$-torsor $\ud{\Isom}(\wt{M}_{R_E,1},M_{R_E})$ over $\BW(R_E)$. Then isomorphisms $\Psi_{R_G}$ and $\Psi_{R_E/\fa_{R_E}}$ define a point of $\CF$ valued in $\BW(R_G)\times_{\BW(R_G/\fa_{R_G})}\BW(R_E/\fa_{R_E})\simeq \BW(R_E/(I_G\cap \fa_{R_E}))$.
	As in Lemma \ref{formsmooth}, we can lift this point to a $\BW(R_E)$-valued point of $\CF$, which corresponds to an isomorphism $\Psi_{R_E}: \wt{M}_{R_E,1}\simto M_{R_E}$. Hence, $\Psi_{R_E}$ is constant modulo $\fa_{R_E}$. By Lemma \ref{lem-versal} and the discussion in \cite[\S 3.2.12]{kisin2018integral}, the \Dieudonne display $(M_{R_E},M_{R_E,1},\Psi_{R_E})$ is universal.
\end{proof}

\begin{lemma} \label{lemzink}
	Let $R$ be a Noetherian complete local ring with residue field $k$. Let $I,J\sset \fm_R$ be two ideals of $R$. 
	\begin{enumerate}[label=(\arabic*)]
		\item $R/(I\cap J)\simeq (R/I)\times_{R/(I+J)}(R/J)$.
		\item $\BW(R/(I\cap J)) \simeq \BW(R/I)\times_{\BW(R/(I+J))}\BW((R/J))$.
	\end{enumerate}
\end{lemma}
\begin{proof}
	(1) Note that the fiber product $(R/I)\times_{R/(I+J)}(R/J)$ equals  \begin{flalign*}
		\cbra{(a,b)\in R/I\times R/J\ |\ a\equiv b\mod I+J }.
	\end{flalign*} For $x\in R$, denote by $\ol{x}$ the reduction of $x$ in a quotient ring of $R$. It is clear that the natural map $x\mapsto (\ol{x},\ol{x})$ induces an injection $R/(I\cap J)\hookrightarrow (R/I)\times_{R/(I+J)}(R/J)$. We claim that it is surjective. For any $(\ol{x},\ol{y})\in (R/I)\times_{R/(I+J)}(R/J)$ with lifts $(x,y)\in R^2$, we have $x-y\in I+J$. Write $x-y=i-j$ for some $i\in I$ and $j\in J$. Set $z\coloneqq x-i (=y-j)$. Then we obtain that $z\mod I=\ol{x}$ and $z\mod J=\ol{y}$.
	
	(2) Since the formation of Witt vectors is a right adjoint (see e.g. \cite[Theorem 2]{joyal1985delta}), the functor $W(-)$ commutes with limits, and hence, we have a natural isomorphism \begin{flalign*}
		W\rbra{(R/I)\times_{R/(I+J)}(R/J)} \simeq W(R/I)\times_{W(R/(I+J))}W(R/J).
	\end{flalign*} Then the isomorphism in (1) induces an isomorphism \begin{flalign*}
		W(R/(I\cap J)) \simeq W(R/I)\times_{W(R/(I+J))}W(R/J).
	\end{flalign*}
	Hence,
	by construction of Zink rings, the above isomorphism restricts to an isomorphism as in (2). 
	
%	(2) By the formation of Witt rings (see, e.g., \cite[Chapter II, \S 6, Theorem 7]{serre1979local}), the isomorphism in (1) induces an isomorphism \begin{flalign*}
%		W(R/(I\cap J)) \simeq W(R/I)\times_{W(R/(I+J))}W(R/J).
%	\end{flalign*}
%	Note that the map of underlying sets is simply the infinite product of the bijection in (1). Hence,
%	by construction of Zink rings, the above isomorphism restricts to an isomorphism as in (2). 
%	\begin{flalign*}
%		\BW(R/(I\cap J))\simeq \BW(R/I)\times_{\BW(R/(I+J))}\BW(R/J).
%	\end{flalign*} 
\end{proof}

Following \cite[\S 4]{zhou2020mod}, we make the following definition. 

\begin{defn}\label{defn-adapted}
	Let $\sG$ be a $p$-divisible group over $\CO_K$ deforming $\sG_0$. Then $\sG$ corresponds to a \Dieudonne display $\Theta_{\CO_K}(\sG)$ over $\CO_K$; see Theorem \ref{thm-classify}. We say that $\mathscr{G}$ is \dfn{$(\CG_W,\mu_y)$-adapted} if the tensors $s_{\alpha,0}$ lift to Frobenius invariant tensors $\wt{s}_{\alpha}\in \Theta_{\CO_K}(\sG)^\otimes$ such that\footnote{For a \Dieudonne display $\CP=(M,M_1,F,F_1)$ over $R$, we use $\CP^\otimes$ to denote the $\BW(R)$-module $M^\otimes$ defined as in Definition \ref{defn-tensors}. } the following two conditions hold: \begin{enumerate}
		\item There is an isomorphism $\Theta_{\CO_K}(\sG)\simeq \BD\otimes_W\BW(\CO_K)$ sending $\wt{s}_\alpha$ to $s_{\alpha,0}\otimes 1$.
		\item Under the canonical isomorphism $\BD(\sG)(\CO_K)\otimes_{\CO_K}K\simeq \BD\otimes_WK$, the filtration on $\BD\otimes_WK$ is induced by a $G$-valued cocharacter $G$-conjugate to $\mu_y$. 
	\end{enumerate}
\end{defn}

\begin{prop}\label{prop-adapted}
	Assume (A1) to (A3). View $\Spf R_E$ as the universal deformation space of $\sG_0$ by the construction in Proposition \ref{prop-torsor} (2). Then for any finite extension $K/E$, a map $\xi: R_E\ra \CO_K$ factors through $R_G$ if and only if the $p$-divisible group $\sG_\xi=\xi^*\sG_{R_E}$ is $(\CG_W,\mu_y)$-adapted.
\end{prop}
\begin{proof}
    The proof goes as in \cite[Proposition 3.2.7]{kisin2024independenceellfrobeniusconjugacy} (see also \cite[Proposition 4.8]{zhou2020mod} and \cite[Proposition 3.2.17]{kisin2018integral}). 
\end{proof}

In \S \ref{sec-intmodelhodge}, we will construct $(\CG_W,\mu_y)$-adapted deformations of $p$-divisible groups associated to closed points in integral models of Shimura varieties, and apply Proposition \ref{prop-adapted} to describe the local structure of integral models of Shimura varieties.

\section{Integral models of Shimura varieties of abelian type} \label{sec-intmodelhodge}
In this section, we will prove Theorem \ref{thm1}. Following the strategy of \cite{kisin2018integral,kisin2024integralmodelsshimuravarieties}, we first consider Shimura varieties $\Sh_\rK(\bG,X)$ of Hodge type. We construct their integral models $\sS_{\rK}(\bG,X)$ by using the Hodge embeddings into Siegel modular varieties, as in \textit{loc. cit.}. Under certain assumptions (see Theorem \ref{thm-hodgemodel}), we apply the deformation theory developed in \S \ref{sec-defm} to identify the formal neighborhood of $\sS_\rK(\bG,X)$ with that of the local model. Then we extend this construction of integral models to the case of Shimura varieties of abelian type by choosing suitable Hodge type lifts under certain conditions (see Theorem \ref{thm-main}). We complete the proof of Theorem \ref{thm1} by showing that these conditions are satisfied in Case (A) or (B). 

\subsection{Shimura varieties of Hodge type} \label{subsubsec-hodge}
Let $(\bG,X)$ be a Shimura datum, that is, $\bG$ is a reductive group over $\BQ$ and $X$ is a $\bG(\BR)$-conjugacy class of $$h: \BS\coloneqq \Res_{\BC/\BR}\BG_m\ra \bG_\BR$$ satisfying axioms 2.1.1.1-2.1.1.3 in \cite[\S 2.1]{deligne1979varietes}. Denote by $\mu_h: \BG_{m\BC}\ra \bG_{\BC}$ the associated Hodge cocharacter, defined  by $\mu_h(z)=h_{\BC}(z,1)$. Set $w_h\coloneqq \mu_h\inverse\mu_h^{c-1}$ (the weight homomorphism), where $c$ denotes the complex conjugation.

Fix a $\BQ$-vector space $V$ of dimension $2g$ with a perfect alternating pairing $\psi: V\times V\ra \BQ$. Let $\bGSp=\bGSp(V,\psi)$ be the corresponding symplectic similitude group over $\BQ$, and let $S^\pm=S^\pm(V,\psi)$ be the Siegel double space consisting of maps $h:\BS\ra \bGSp_\BR$ such that \begin{enumerate}
	\item The map $\BS\xrightarrow{h}\bGSp_\BR\hookrightarrow\GL(V_\BR)$ gives rise to a Hodge structure of type $(-1,0), (0,-1)$ on $V_\BR$, i.e., $V_\BC=V^{-1,0}\oplus V^{0,-1}$.
	\item The pairing $(x,y)\mapsto \psi(x,h(i)y)$ is (positive or negative) definite on $V_\BR$.
\end{enumerate}
Then $(\bGSp,S^\pm)$ is a Shimura datum, which is called a \dfn{Siegel Shimura datum}. 

For the rest of the subsection, we assume $(\bG,X)$ is of Hodge type, i.e., there exists an embedding of Shimura data $$\iota: (\bG,X)\hookrightarrow (\bGSp(V,\psi),S^\pm).$$ Sometimes we will write $G$ for $\bG_{\BQ_p}$ for simplicity. 

Let $\bE=\bE(\bG,X)$ be the reflex field with ring of integers $\CO_{\bE}$. Let $p$ be a prime number. Let $\BA_f$ denote the ring of finite \adeles over $\BQ$, and $\BA^p_f $ denote the ring of prime-to-$p$ finite \adeles, which we consider as the subgroup of $\BA_f$ with trivial component at $p$.  Fix a place $v|p$ of $\bE$, and let $E$ denote the completion of $\bE$ at $v$. Denote by $\CO_{\bE,(v)}$ (resp. $\CO_E$) the localization (resp. completion) of $\CO_{\bE}$ at $v$. Let $\CG$ be the Bruhat--Tits stabilizer group scheme over $\BZ_p$ associated with some $x\in \CB(G,\BQ_p)$, whose neutral component $\CG^\circ$ is parahoric. Set $\rK_p=\CG(\BZ_p)$ or $\CG^\circ(\BZ_p)$ and $\rK=\rK_p\rK^p$ with $\rK^p\sset \bG(\BA^p_f)$ a sufficiently small open compact subgroup. 
By general theory of Shimura varieties, these data yield a quasi-projective smooth algebraic variety $\Sh_\rK(\bG,X)$ canonically defined over $\bE$, whose $\BC$-points are given by $$\Sh_\rK(\bG,X)(\BC)= \bG(\BQ)\backslash X\times \bG(\BA_f)/\rK.$$
Throughout the paper, we will always assume that $\rK^p$ is small enough that the above quotient exists as an algebraic variety.

We can also consider the projective limit of $\bE$-schemes \begin{flalign*}
	   \Sh(\bG,X) =\varprojlim_{\rK}\Sh_\bK(\bG,X), \ 
	   \text{resp.}\ \Sh_{\rK_p}(\bG,X)=\varprojlim_{\rK^p}\Sh_{\rK_p\rK^p}(\bG,X),
\end{flalign*} which carries a natural action of $\bG(\BA_f)$ (resp. $\bG(\BA_f^p)$). The projective limit exists since the transition maps are finite, hence affine.

\subsubsection{Integral models for level \texorpdfstring{$\CG(\BZ_p)$}{\mathcal G(Zp)}: construction} \label{subsubsec-intconstruction}
Assume that \begin{enumerate}[label=(\roman*)]
		\item $\rK_p=\CG(\BZ_p)$;
		\item $\iota_{\BQ_p}$ extends to a very good integral Hodge embedding $\wt{\iota}: (\CG,\mu_h)\hookrightarrow (\GL(V_{\BZ_p}),\mu_g)$, where $V_{\BZ_p}\sset V_{\BQ_p}$ is a \dfn{self-dual} $\BZ_p$-lattice with respect to $\psi$.
	\end{enumerate} 
% $\rK_p=\CG(\BZ_p)$ and $\iota_{\BQ_p}$ extends to a very good integral Hodge embedding $\wt{\iota}: (\CG,\mu_h)\hookrightarrow (\GL(V_{\BZ_p}),\mu_g)$, where $V_{\BZ_p} \sset V_{\BQ_p}$ is a $\BZ_p$-lattice of rank $2g$ and $V_{\BZ_p}$ is \dfn{self-dual} with respect to $\psi$. 
 We let $\mathcal{GSP}$ denote the parahoric group scheme associated to the self-dual lattice $V_{\BZ_p}$. 
Set $V_{\BZ_{(p)}}\coloneqq V\cap V_{\BZ_p}$. Denote by $G_{\BZ_{(p)}}$ the Zariski closure of $G$ in $\GL(V_{\BZ_{(p)}})$, then $\CG\simeq G_{\BZ_{(p)}}\otimes_{\BZ_{(p)}}\BZ_p$. Set $\rK'_p\coloneqq \mathcal{GSP}(\BZ_p)$. Let $\rK'^p$ be a small enough open compact subgroup of $\bGSp(\BA^p_f)$ containing $\rK^p$, which leaves $V_{\wh{\BZ}^p}$ stable. Here $\wh{\BZ}^p\coloneqq \prod_{\ell\neq p}\BZ_\ell$.  Set $\rK'=\rK'_p\rK'^p$. Then the embedding $\iota$ induces a closed immersion over $\bE$ $$\Sh_\rK(\bG,X)\hookrightarrow \Sh_{\rK'}(\bGSp,S^\pm)\otimes_\BQ\bE.$$ The choice of $V_{\BZ_{(p)}}$ gives rise to an interpretation of $\Sh_{\rK'}(\bGSp,S^\pm)$ as a moduli space of polarized abelian varieties, and hence to a natural integral model $\mathscr{S}_{\rK'}(\bGSp,S^\pm)$ over $\BZ_{(p)}$ (cf. \cite[\S 6.3]{zhou2020mod}).

\begin{defn}\label{defn-model}
	The integral model $\mathscr{S}_\rK(\bG,X)$ of $\Sh_\rK(\bG,X)$ is the normalization of the (reduced) Zariski closure $\mathscr{S}^-_\rK(\bG,X)$ of $\Sh_\rK(\bG,X)$ in $\mathscr{S}_{K'}(\bGSp,S^\pm)_{\CO_{\bE,{(v)}}}$. We set $$\sS_{\rK_p}(\bG,X)\coloneqq \varprojlim_{\rK^p}\sS_{\rK_p\rK^p}(\bG,X).$$ The $\bG(\BA_f^p)$-action on $\Sh_{\rK_p}(\bG,X)$ extends to $\sS_{\rK_p}(\bG,X)$. 
\end{defn}

\subsubsection{Hodge tensors and deformation theory}\label{subsubsec-hodgetensors}
Since $G_{\BZ_{(p)}}$ has reductive generic fiber, we can find a finite collection of tensors $(s_{\alpha})\sset V_{\BZ_{(p)}}^\otimes=(V_{\BZ_{(p)}}^\vee)^\otimes$ whose scheme-theoretic stabilizer is $G_{\BZ_{(p)}}$. Let $h: \CA\ra \mathscr S_\rK(\bG,X)$ denote the pullback of the universal abelian scheme over $\mathscr S_{\rK'}(\bGSp,S^\pm)$. Denote by $\CV=R^1h_*\Omega^\bullet$ the (relative) algebraic de Rham cohomology of $\CA$. Then the tensors $(s_\alpha)$, by the de Rham isomorphism, give rise to a collection of (absolute) Hodge cycles $s_{\alpha,\dR}\in\CV^\otimes_\BC$, where $\CV_\BC$ is the complex analytic vector bundle attached to $\CV$, and $s_{\alpha,\dR}$ descends to $\CV^\otimes$ by \cite[Proposition 4.2.6]{kisin2018integral} (i.e., $s_{\alpha,\dR}$ can be defined over $\CO_{\bE,(v)}$).

Recall that $\breve E$ denotes the completion of the maximal unramified extension of $E$ in $\ol{\BQ}_p$ with residue field $k$. Let $L/\breve E$ be a finite extension. For a point $x\in \Sh_\rK(\bG,X)(L)$ specializing to $\ol{x}\in\sS_\rK^-(\bG,X)(k)$, we write $\CA_x$ for the pullback of $\CA$ to $x$ and write $\sG_x$ for the $p$-divisible group associated with $\CA_x$. Then $s_{\alpha,\dR}$ pulls back to $s_{\alpha,\dR,x}\in H^1_\dR(\CA_x)^\otimes$. We can also obtain corresponding tensors $s_{\alpha,\et,x}$ in $T_p\sG_x^{\vee\otimes}$ by the Betti-\etale comparison theorem. Here $T_p\sG_x^\vee\coloneqq \Hom_{\BZ_p}(T_p\sG_x,\BZ_p)$. The tensors $s_{\alpha,\et,x}$ are Galois invariant and their scheme-theoretic stabilizer is isomorphic to $\CG$. Write $\sG_{\ol{x}}$ for the $p$-divisible group corresponding to $\ol{x}$ and $\BD_{\ol{x}}$ for $\BD(\sG_{\ol{x}})(W)$. Set $V\coloneqq T_p\sG_x^{\vee}\otimes_{\BZ_p}\BQ_p$. Then $V$ is a crystalline representation of $\Gamma_L\coloneqq \Gal(\ol{L}/L)$. The $p$-adic comparison isomorphism $$B_\cris\otimes_{\BZ_p} T_p\sG_x^{\vee}\simeq B_\cris\otimes_{K_0}D_\cris(V), \quad D_\cris(V)\coloneqq (B_\cris\otimes_{\BQ_p}V)^{\Gamma_L}, $$ takes the Galois invariant tensors $s_{\alpha,\et,x}$ to the Frobenius invariant tensors $s_{\alpha,0}\in D_\cris(V)^\otimes$. Note that (see e.g., \cite[(2.12.3)]{kim20162}) there exists a canonical isomorphism \begin{flalign*}
	D_\cris(V)\simeq \BD_{\ol{x}}[1/p].
\end{flalign*}
 Hence, we may view $\BD_{\ol{x}}^\otimes$ as a $W$-submodule of the $K_0$-vector space $(\BD_{\ol{x}}[1/p])^\otimes=D_\cris(V)^\otimes$.

\begin{prop}\label{prop-integraltensor}
	We have $s_{\alpha,0}\in \BD_{\ol{x}}^\otimes\sset D_\cris(V)^\otimes$. Moreover, we have the following properties. \begin{enumerate}
		\item The tensors $s_{\alpha,0}$ lift to Frobenius invariant tensors $\wt{s}_{\alpha,x}\in \Theta_{\CO_L}(\sG_x)^\otimes$, which map into $\Fil^0\BD(\sG_x)(\CO_L)^\otimes$ along the natural projection $\Theta_{\CO_L}(\sG_x)\ra \BD(\sG_x)(\CO_L)$ given by Theorem \ref{thm-classify} (2). Denote by $s_{\alpha,x}$ the image of $\wt{s}_{\alpha,x}$.
		\item There exists an isomorphism $\Theta_{\CO_L}(\sG_x)\simeq \BW(\CO_L)\otimes_{\BZ_p}T_p\sG_x^\vee$ taking $\wt{s}_{\alpha,x}$ to $s_{\alpha,\et,x}$. 
		In particular, there exists an isomorphism $$\BD_{\ol{x}}\simeq W\otimes_{\BZ_p}T_p\sG_x^\vee$$ taking $s_{\alpha,0}$ to $s_{\alpha,\et,x}$, and an isomorphism $$\BD(\sG_x)(\CO_L)\simeq \BD(\sG_{\ol{x}})(W) \otimes_W\CO_L $$ taking $s_{\alpha,x}$ to $s_{\alpha,0}$. Therefore, we can identify the group scheme $\CG_W\sset \GL(\BD_{\ol{x}})$ defined by $s_{\alpha,0}$ with $\CG\otimes_{\BZ_p}W$, and there exists a $G_{K_0}$($=\CG_W\otimes_WK_0$)-valued cocharacter $\mu_y$ such that \begin{enumerate}
		\item The filtration on $\BD_{\ol{x}}\otimes_WL$ induced by the canonical isomorphism $$\BD_{\ol{x}}\otimes_WL\simeq \BD(\sG_x)(\CO_L)\otimes_{\CO_L}L$$ is given by a $G_{K_0}$-valued cocharacter $G_{K_0}$-conjugate to $\mu_y$.
		\item $\mu_y$ induces a filtration on $\BD_{\ol{x}}$ which lifts the Hodge filtration on $\BD_{\ol{x}}\otimes_Wk=\BD(\sG_{\ol{x}})(k)$. 
	    \end{enumerate}
	\end{enumerate}  
\end{prop}
\begin{proof}
	As in \cite[Proposition 3.3.8]{kisin2018integral}, the tensors $(s_{\alpha,\et,x})\sset T_p\sG_x^{\vee\otimes}$ give rise to Frobenius invariant tensors $s_{\alpha,x}^{\fM}\sset \fM(\sG_x)^\otimes$. The tensors $s_{\alpha,x}^\fM$ map to tensors $\wt{s}_{\alpha,x}$ in $\Theta_{\CO_L}(\sG_x)^\otimes$ via the isomorphism \begin{flalign*}
	   \Theta_{\CO_K}(\sG_x)\simeq \varphi^* \fM(\sG_x) \otimes_{\fS,\kappa} \BW(\CO_K)
\end{flalign*} in Corollary \ref{coro-Mequi} (1). Since the above isomorphism respects the Hodge filtrations by Corollary \ref{coro-Mequi} (2), the tensors $\wt{s}_{\alpha,x}$ map into $\Fil^0\BD(\sG_x)(\CO_L)^\otimes$. The rest of the proof proceeds as in \cite[Proposition 3.3.8, Corollary 3.3.10]{kisin2018integral}.
\end{proof}

The above proposition implies that $\sG_x$ is a $(\CG_W,\mu_y)$-adapted deformation of $\sG_{\ol{x}}$ in the sense of Definition \ref{defn-adapted}.

\subsubsection{Integral models for level \texorpdfstring{$\CG(\BZ_p)$}{\mathcal G(Zp)}: properties}
Fix a parabolic subgroup $P\sset \GL(\BD_{\ol{x}})$ lifting $P_0$ corresponding to the Hodge filtration of $\BD(\sG_{\ol{x}})(k)= \BD_{\ol{x}}\otimes_Wk$. Let $y=y(x)\in \rbra{\GL(\BD_{\ol{x}})/P)}(L)$ correspond to the cocharacter $\mu_y$ as in Proposition \ref{prop-integraltensor} (2). Then as in \S \ref{subsec-crystensor}, we obtain from $y$ a closed subscheme $M_{G,y}^\loc\sset \GL(\BD_{\ol{x}})/P$ and formal local models $$\wh{M}^\loc=\Spf R,\quad \wh{M}^\loc_{G,y}=\Spf R_G.$$ Note that $R_G$ is a quotient of $R_E=R\otimes_W\CO_E$. By Proposition \ref{prop-integraltensor} (2) and the Betti-\etale comparison theorem, the scheme $\ud{\Isom}_{(s_\alpha,s_{\alpha,0})} (V_{\BZ_p}^\vee,\BD_{\ol{x}})$ of tensor-preserving isomorphisms is a trivial $\CG$-torsor. Then we may choose an isomorphism  $V_{\BZ_p}^\vee\simeq \BD_{\ol{x}}$ preserving tensors such that the very good Hodge embedding (by our assumption on $\wt{\iota}$) $$(\CG,\mu_h)\overset{\wt{\iota}} {\hookrightarrow} (\GL(V_{\BZ_p}),\mu_g)\simeq(\GL(V_{\BZ_p}^\vee),\mu_g) \simeq (\GL(\BD_{\ol{x}}),\mu_g)$$ induces a closed immersion $\BM^\loc_{\CG,\mu_h}\hookrightarrow (\GL(\BD_{\ol{x}})/P)_{\CO_E}\simtos \Gr(g,\BD_{\ol{x}})_{\CO_E}$. Note that the Hodge filtration on $\BD_{\ol{x}}\otimes_WL$ is induced by a $G$-valued cocharacter conjugate to $\mu_h\inverse$. Hence, we can identify ${M}^\loc_{G,y}$ with $\BM^\loc_{\CG,\mu_h}$ by Lemma \ref{lem-closure}, and so $R_G$ is normal.

 \begin{prop} \label{prop-hodgelocal}
 Suppose that conditions (i) and (ii) in the beginning of \S \ref{subsubsec-intconstruction} are satisfied.
 	Let $\wh{U}_{\ol{x}}$ be the completion of $\sS_{\mathrm{K}}^-(\bG,X)_{\CO_{\breve E}}$ at $\ol{x}$. Then the irreducible component of $\wh{U}_{\ol{x}}$ containing $x$ is isomorphic to $\wh{M}_{G,y}^\loc=\Spf R_G$ as formal schemes over $\CO_{\breve E}$.
 \end{prop}
 \begin{proof}
 	We follow the arguments of \cite[Proposition 4.2.2]{kisin2018integral}. 
 	
 	Note that $G_{K_0}\sset \GL(\BD_{\ol{x}}\otimes_{\BZ_p}\BQ_p)$ contains scalars, since $\bG\sset \GL(V_\BQ)$ contains the image of the weight homomorphism $w_h$. As $\iota_{\BQ_p}$ extend to a very good Hodge embedding, the constructions and results in \S \ref{sec-defm} can apply. In particular, by Proposition \ref{prop-torsor} (2), we can view $\Spf R_E$ as a universal deformation space of $\sG_{\ol{x}}$. Then the $p$-divisible group over $\wh{U}_{\ol{x}}$ arising from the universal abelian scheme $\CA$ gives rise to a natural map $\Phi: \wh{U}_{\ol{x}}\ra \Spf R_E$, which is a closed embedding by Serre-Tate theorem. 
 	
 	Let $Z\sset \wh{U}_{\ol{x}}$ be the irreducible component containing $x$. Let $x'\in Z(L')$ for some finite field extension $L'$ of $\breve E$. Then we can argue as in \cite[Proposition 4.2.2]{kisin2018integral} to show: $s_{\alpha,\et,x'}$ corresponds to $s_{\alpha,0}$ under the $p$-adic comparison isomorphism for the $p$-divisible group $\sG_{x'}$. Since the filtration on $\BD_{\ol{x}}\otimes_WK'$ corresponding to $\sG_{x'}$ is given by a $G$-valued cocharacter which is conjugate to $\mu_y$, by Proposition \ref{prop-integraltensor}, $\sG_{x'}$ is $(\CG_W,\mu_y)$-adapted. By our assumption on the integral Hodge embedding $\wt{\iota}$ and Proposition \ref{prop-integraltensor}, the assumptions in Proposition \ref{prop-adapted} are satisfied. Hence, $x'$ is induced by a point of $\wh{M}_{G,y}^\loc$ by Proposition \ref{prop-adapted}. Since $x'$ is arbitrary, it follows that $\Phi(Z)\sset \wh{M}^\loc_{G,y}$. They are equal, as $Z$ and $\wh{M}^\loc_{G,y}$ are of the same dimension. 	
 \end{proof}

\begin{thm}\label{thm-hodgemodel}
	Assume the following conditions: \begin{enumerate}[label=(\roman*)]
		\item $\rK_p=\CG(\BZ_p)$;
		\item $\iota_{\BQ_p}$ extends to a very good integral Hodge embedding $\wt{\iota}: (\CG,\mu_h)\hookrightarrow (\GL(V_{\BZ_p}),\mu_g)$, where $V_{\BZ_p}\sset V_{\BQ_p}$ is a \dfn{self-dual} $\BZ_p$-lattice with respect to $\psi$.
	\end{enumerate} 
	Then the $\CO_{\bE,(v)}$-schemes $\sS_\rK(\bG,X)$ and $\sS_{\rK_p}(\bG,X)$ constructed in Definition \ref{defn-model} satisfy the following properties.
	\begin{enumerate}
		\item $\sS_{\rK_p}(\bG,X)$ is an $\CO_{\bE,(v)}$-flat, $\bG(\BA_f^p)$-equivariant extension of $\Sh_{\rK_p}(\bG,X)$. 
		\item For any discrete valuation ring $R$ of mixed characteristic $0$ and $p$, the natural map $$\sS_{\rK_{p}}(\bG,X)(R)\ra \sS_{\rK_{p}}(\bG,X)(R[1/p])$$ is a bijection.  
		\item $\sS_{\rK}(\bG,X)$ fits into a local model diagram $$\xymatrix{
		   &\wt{\sS}_\rK(\bG,X)_{\CO_E}\ar[ld]_{\pi}\ar[rd]^q &\\ \sS_\rK(\bG,X)_{\CO_E} & &\BM^\loc_{\CG,\mu_h}
		}$$ of $\CO_E$-schemes, in which $\pi$ is a $\CG$-torsor and $q$ is $\CG$-equivariant and smooth of relative dimension $\dim G$. 
		\item If in addition, we have $\CG=\CG^\circ$, then for each $x\in \sS_\rK(\bG,X)(k')$ with $k'/k_E$ finite, there is a point $y\in \BM^\loc_{\CG,\mu_h}(k')$ such that we have an isomorphism of henselizations $$\CO^h_{\sS_\rK(\bG,X),x}\simeq \CO^h_{\BM^\loc_{\CG,\mu_h},y}.$$
	\end{enumerate} 
\end{thm}
\begin{proof}
	Note that under the assumptions of the above theorem, we have Proposition \ref{prop-hodgelocal}, which extends \cite[Proposition 4.2.2]{kisin2018integral} to the case $p=2$. Then the proofs of \cite[Proposition 4.2.2,4.2.7]{kisin2018integral} and \cite[Theorem 7.1.3]{kisin2024integralmodelsshimuravarieties} go through, and we obtain the theorem. We note that the assumption (B) in \cite[Theorem 7.1.3]{kisin2024integralmodelsshimuravarieties} is not used in the proof.  
\end{proof}

%\begin{remark}\label{schlmd}
%    Since our construction of the local model diagram is the same as the construction used in \cite[Theorem 7.1.3]{kisin2024integralmodelsshimuravarieties}, the same proof of \cite[Theorem A.3.3]{daniels2024conjecture} shows that the diagram in Theorem \ref{thm-hodgemodel} (3) is a scheme-theoretic local model diagram in the sense of \cite[Definition 4.9.1]{pappas2021p}.
%\end{remark}

\subsubsection{Integral models for parahoric level \texorpdfstring{$\CG^\circ(\BZ_p)$}{\mathcal G\^circ(Zp)}}  \label{sub414}
Now we use previous results to study integral models with parahoric level structure. That is, the level at $p$ is given by $\CG^\circ(\BZ_p)$. Write $\rK_p^\circ=\CG^\circ(\BZ_p)$ and $\rK^\circ=\rK_p^\circ \rK^p$. 
 
Note that there is a natural finite morphism of Shimura varieties $\Sh_{\rK^\circ}(\bG,X)\ra \Sh_\rK(\bG,X)$.

\begin{defn}\label{model2}
	The integral model $\sS_{\rK^\circ}(\bG,X)$ for parahoric level $\rK^\circ$ is the normalization of $\sS_\rK(\bG,X)$ in $\Sh_{\rK^\circ}(\bG,X)$. We also set $$\sS_{\rK_p^\circ}(\bG,X)\coloneqq \varprojlim_{\rK^p}\sS_{\rK_p^\circ \rK^p}(\bG,X). $$
\end{defn}

\begin{remark} \label{hodgecan}
    Since our construction (Definition \ref{defn-model} and \ref{model2}) agrees with the construction used in \cite[Theorem 4.5.2]{pappas2021p} and \cite[Theorem 4.2.3]{daniels2024conjecture}, the integral models $\sS_{\rK_p}(\bG,X)$ and $\sS_{\rK_p^\circ}(\bG,X)$ are canonical in the sense of \cite{pappas2021p}. 
    
    Moreover, the same argument of \cite[Theorem A.3.3]{daniels2024conjecture} shows that the diagram in Theorem \ref{thm-hodgemodel} (3) is a scheme-theoretic local model diagram in the sense of \cite[Definition 4.9.1]{pappas2021p}: the proof of \cite[Theorem A.3.3]{daniels2024conjecture} uses only assumptions (i) and (ii) in Theorem \ref{thm-hodgemodel}.
\end{remark}

Let $\bG^\mathrm{sc}$ denote the simply connected cover of $\bG^\der$ and set $\bC=\ker(\bG^{\mathrm{sc}}\ra \bG^\der)$. For a finite prime $\ell$ and $c\in H^1(\BQ,\bC)$, we write $c_\ell$ for the image of $c$ in $H^1(\BQ_\ell,\bC)$. We introduce the following assumption: \begin{flalign}
	     \text{If $c\in H^1(\BQ,\bC)$ satisfies $c_\ell=0$ for all $\ell\neq p$, then $c_p=0$. } \label{eq-c-assump}
\end{flalign}  
By \cite[Lemma 4.3.5]{kisin2018integral}, condition \eqref{eq-c-assump} implies that for a sufficiently small $\rK^p$, we have \begin{flalign}
	\bG(\BQ)\cap \rK\sset \rK^\circ. \label{eq10}
\end{flalign}

\begin{prop} \label{prop-paramodel}
	Assume that conditions $(i)$ and $(ii)$ in Theorem \ref{thm-hodgemodel} are satisfied. 
	\begin{enumerate}
	    \item The geometrically connected components of $\sS_{\rK_p^\circ}(\bG,X)$ are defined over the maximal extension of $\bE$ that is unramified at primes above $p$.
%	    \item The covering $\sS_{\rK^\circ}(\bG,X)\ra \sS_\rK(\bG,X)$ is finite \etale.
		\item Assume \eqref{eq-c-assump}. For any sufficiently small $\rK^p$ such that \eqref{eq10} holds, the covering $\sS_{\rK^\circ}(\bG,X)\ra \sS_\rK(\bG,X)$ is finite \etale and splits over an unramified extension of $\CO_E$. 
	\end{enumerate}
\end{prop}
\begin{proof}
	The proof follows the same argument as in \cite[Proposition 4.3.7, Corollary 4.3.9]{kisin2018integral}.
\end{proof}

\begin{remark}
	The covering $\sS_{\rK^\circ}(\bG,X)\ra \sS_\rK(\bG,X)$ is finite \etale for any sufficiently small $\rK^p$ such that $\Sh_\rK(\bG,X)$ is a variety. This follows from \cite[Theorem 4.1.15]{daniels2024conjecture} and the canonicity of $\sS_{\rK}(\bG,X)$.
\end{remark}

\subsection{Shimura varieties of abelian type}\label{subsect-cond}
Let $(\bG,X)$ be a Shimura datum of Hodge type with a Hodge embedding $\iota: (\bG,X)\hookrightarrow (\bGSp(V,\psi),S^\pm)$. Denote by $G$ the base change $\bG_{\BQ_p}$. Let $\CG$ be the Bruhat--Tits stabilizer group scheme associated to some point $x\in \CB(G,\BQ_p)$. Write $\CG^\circ$ for the neutral component of $\CG$. 
 Assume  \begin{enumerate}[label=(\roman*)]
	\item $\rK_p=\CG(\BZ_p)$;
	% \item $\iota_{\BQ_p}$ extends to a very good integral Hodge embedding $\wt{\iota}: (\CG,\mu_h)\hookrightarrow (\GL(V_{\BZ_p}),\mu_g)$, where $V_{\BZ_p}\sset V_{\BQ_p}$ is a \dfn{self-dual} $\BZ_p$-lattice with respect to $\psi$; 
	% \item $\bG$ satisfies condition (\ref{eq-c-assump});
	\item The center $Z_G$ of $G$ is an $R$-smooth torus (see \cite[2.1.4]{kisin2024integralmodelsshimuravarieties}).
\end{enumerate}
Let $(\bG_2, X_2)$ denote a Shimura datum of abelian type such that there is a central isogeny $\bG^\der\ra \bG_2^\der$ inducing an isomorphism of Shimura data $(\bG^\ad,X^\ad)\simto (\bG_2^\ad,X_2^\ad)$. Here, $X^\ad$ denotes the $\bG^\ad(\BR)$-conjugacy class of $h^\ad: \BS\xrightarrow{h} \bG_{\BR}\ra \bG^\ad_\BR$ for some $h\in X$; $X_2^\ad$ is similar. 

As usual, we denote $\rK_p^\circ\coloneqq \CG^\circ(\BZ_p)\sset G(\BQ_p)$ and $G_2\coloneqq \bG_{2,\BQ_p}$. Let $x_2\in \CB(G_2,\BQ_p)$ be a lift of $x_2^\ad=x^\ad$ in the identification $\CB(G_2^\ad,\BQ_p)=\CB(G^\ad,\BQ_p)$. Let $\CG_2^\circ$ be the parahoric group scheme associated to $x_2$. Write $\rK_{2,p}^\circ=\CG^\circ_2(\BZ_p)$. As discussed in \cite[\S 1.1.2]{kisin2018integral}, $\sG_2^\circ$ only depends on $x_2^\ad=x^\ad$.

Denote by $\bE_2$ the reflex field of $(\bG_2,X_2)$ and set $\bE'\coloneqq \bE\cdot \bE_2$, recall $\bE$ denotes the reflex field of $(\bG,X)$. We fix a place $v'$ of $\bE'$ above $v$. Denote by $E'$ the completion of $\bE'$ at $v'$.

Fix a connected component $X^+\sset X$. Denote by $\Sh_{\rK_p^\circ}(\bG,X)^+$ the geometrically connected component containing the image of $X^+\times \cbra{1}$ in $$\varprojlim_{\rK^p} \bG(\BQ)\backslash X\times \bG(\BA_f)/\rK_p^\circ \rK^p.$$ By Proposition \ref{prop-paramodel} (1), $\Sh_{\rK_p^\circ}(\bG,X)^+$ is defined over the maximal extension $\bE^p$ of $\bE$ that is unramified at primes above $p$. We denote by $\sS_{\rK_p^\circ}(\bG,X)^+$ the component of $\sS_{\rK_p^\circ}(\bG,X)$ extending $\Sh_{\rK_p^\circ}(\bG,X)^+$, which is defined over $\CO_{\bE^p,(v)}\coloneqq \CO_{\bE^p}\otimes_{\CO_\bE}\CO_{\bE,(v)}$. 

\subsubsection{Integral models of Shimura varieties of abelian type}
We recall the notation of \cite{deligne1979varietes}. Let $H$ be a group equipped with an action of a group $\Delta$, and let $\Gamma\sset H$ be a $\Delta$-stable subgroup. Suppose we are given a $\Delta$-equivariant map $\varphi: \Gamma\ra \Delta$ where $\Delta$ acts on itself by inner automorphisms, and suppose that for $\gamma\in\Gamma$, $\varphi(\gamma)$ acts on $H$ as conjugation by $\gamma$. Then the elements of the form $(\gamma,\varphi(\gamma)\inverse)$ form a normal subgroup of the semi-direct product $H\rtimes\Delta$. We denote by $$H*_{\Gamma}\Delta$$ the quotient of $H\rtimes \Delta$ by this normal subgroup. 

For a subgroup $H\sset \bG(\BR)$, denote by $H_+$ the preimage in $H$ of the connected component $\bG^\ad(\BR)^+$ of the identity in $\bG^\ad(\BR)$. We write $\bG^\ad(\BQ)^+=\bG^\ad(\BQ)\cap \bG^\ad(\BR)^+$.

\begin{lemma}
	\label{lem-groups} Suppose $S$ is an affine $\BQ$-scheme, and let $S_{\BZ_p}$ be a flat affine $\BZ_p$-scheme with generic fiber $S\otimes_\BQ\BQ_p$. Then there exists a $\BZ_{(p)}$-scheme $S_{\BZ_{(p)}}$, which is unique up to isomorphism, with generic fiber $S$ and $S_{\BZ_{(p)}}\otimes_{\BZ_{(p)}}\BZ_p=S_{\BZ_p}$.
\end{lemma}
\begin{proof}
	Let $A$ (resp. $B$) be the affine coordinate ring of $S_{\BZ_p}$ (resp. $S$). We have $A\otimes_{\BZ_p}\BQ_p=B\otimes_\BQ\BQ_p$. Then we can take $S_{\BZ_{(p)}}$ to be $\Spec A\cap B$, where the intersection happens in $A\otimes_{\BZ_p}\BQ_p=B\otimes_\BQ\BQ_p$. Any $\BZ_{(p)}$-scheme $T$ with generic fiber $S$ and $T\otimes_{\BZ_{(p)}}\BZ_p=S_{\BZ_p}$ is necessarily isomorphic to $\Spec A\cap B$. 
\end{proof}

By applying the above lemma to the group schemes $\CG$ and $\CG^\circ$ over $\BZ_p$, we obtain $\BZ_{(p)}$-smooth affine group schemes $G_{\BZ_{(p)}}$ and $G^\circ\coloneqq G_{\BZ_{(p)}}^\circ$. Similarly, let $G^{\ad\circ}=G^{\ad\circ}_{\BZ_{(p)}}$ be the  $\BZ_{(p)}$-model of the parahoric group scheme associated to $x^\ad\in \CB(G^\ad,\BQ_p)$. 
Let $G^\ad=G_{\BZ_{(p)}}/Z$, where $Z$ denotes the Zariski closure in $G_{\BZ_{(p)}}$ of the center $\bZ$ of the $\BQ$-group $\bG$.  As we assume that the center $Z_G$ of $G$ is an $R$-smooth torus,  $G^{\ad\circ}$ is the neutral component of $G^\ad$, see \cite[Lemma 4.6.2]{kisin2018integral} and \cite[Proposition 2.4.14]{kisin2024independenceellfrobeniusconjugacy}.

Following \cite[\S 4.6.3]{kisin2018integral}, we set \begin{flalign*}
	  \mathscr{A} (G_{\BZ_{(p)}})&\coloneqq  \bG(\BA_f^p)/Z(\BZ_{(p)})^-*_{G^\circ(\BZ_{(p)})_+/Z^\circ(\BZ_{(p)})} G^{\ad\circ}(\BZ_{(p)})^+, 
	  \\  \mathscr A(\bG) &\coloneqq  \bG(\BA_f)/\bZ(\BQ)^-*_{\bG(\BQ)_+/\bZ(\BQ)}\bG^\ad(\BQ)^+,
\end{flalign*}  and \begin{flalign*}
	\mathscr{A} (G_{\BZ_{(p)}})^\circ &\coloneqq  G^\circ(\BZ_{(p)})_+^- /Z^\circ(\BZ_{(p)})^-*_{G^\circ(\BZ_{(p)})_+/Z^\circ(\BZ_{(p)})} G^{\ad\circ}(\BZ_{(p)})^+,
	\\ \mathscr A(\bG)^\circ &\coloneqq  \bG(\BQ)_+^-/\bZ(\BQ)^-*_{\bG(\BQ)_+/\bZ(\BQ)}\bG^\ad(\BQ)^+.
\end{flalign*}  
Here, $G^\circ(\BZ_{(p)})_+^-$ is the closure of $G^\circ(\BZ_{(p)})_+$ in $\bG(\BA_f^p)$, and $Z^\circ$ is the Zariski closure of $\bZ$ in $G^\circ$.  
Similarly, we have $\mathscr{A} (G_{2,\BZ_{(p)}})$ and $\mathscr{A} (G_{2,\BZ_{(p)}})$. Since $G^{\ad\circ}$ is the neutral component of $G^\ad_{\BZ_{(p)}}$ (we assume $Z_G$ is an $R$-smooth torus), the action of $\mathscr A(G_{\BZ_{(p)}})$ on $\Sh_{\rK_p^\circ}(\bG,X)$ extends to $\sS_{\rK_p^\circ}(\bG,X)$. There is an injection by \cite[Lemma 4.6.10]{kisin2018integral}, $$\mathscr{A} (G_{\BZ_{(p)}})^\circ\backslash \mathscr{A} (G_{2,\BZ_{(p)}})\hookrightarrow \mathscr{A} (\bG)^\circ\backslash \mathscr A(\bG_2)/\rK_{2,p}^\circ.$$

Let $J\sset G_2(\BQ_p)$ be a set of coset representatives for the image of the above injection. Note that, via the homomorphism
$$
G_2(\mathbb Q_p)\longrightarrow G_2^{\mathrm{ad}}(\mathbb Q_p)\simeq G^{\mathrm{ad}}(\mathbb Q_p),
$$
the conjugation action (see \cite[2.0.2]{deligne1979varietes}) of $G^{\mathrm{ad}}(\mathbb Q_p)$ on $G(\mathbb Q_p)$ induces an action of $G_2(\mathbb Q_p)$ on $G(\mathbb Q_p)$. For each $j\in J$, denote by $j\rK^\circ_p j\inverse\sset G(\BQ_p)$ the parahoric subgroup obtained from $\rK^\circ_p$ via the action of $j$. Let $\sS_{j\rK^\circ_pj\inverse}(\bG,X)$ denote the integral model with parahoric level $j\rK^\circ_pj\inverse$ constructed as in \S \ref{sub414}. Define $$\sS_{j\rK^\circ_pj\inverse}(\bG,X)$$ as the $\CO_{\bE^p,(v)}$-scheme given by the schematic closure of $\Sh_{j\rK^\circ_pj\inverse}(\bG,X)^+$ in $\sS_{j\rK^\circ_pj\inverse}(\bG,X)_{\CO_{\bE^p,(v)}}$

\begin{defn}[{\cite[Corollary 4.6.15]{kisin2018integral}, \cite[\S 7.3.1 (3)]{kisin2024integralmodelsshimuravarieties}}] \label{defn-abelian}
	The integral model $\sS_{\rK_{2,p}^\circ}(\bG_2,X_2)$ for $\Sh_{\rK_{2,p}^\circ}(\bG_2,X_2)$ is $$\coprod_{j\in J} [\sS_{j\rK_p^\circ j\inverse}(\bG,X)^+\times \mathscr A(G_{2,\BZ_{(p)}})]/ \mathscr A(G_{\BZ_{(p)}})^\circ.$$
\end{defn}
The scheme $\sS_{\rK^\circ_{2,p}}(\bG_2,X_2)$ is priori defined over $\CO_{\bE'^p, (v)} $, but it descends to an $\CO_{\bE',(v')}$-scheme with a $\bG_2(\BA^p_f)$-action, see \cite[Corollary 4.6.15]{kisin2018integral}. 

\begin{thm} \label{thm-main}
	Assume that conditions $(i)$ and $(ii)$ at the beginning of \S \ref{subsect-cond} are satisfied. Suppose further that 
    \begin{enumerate}[label=(\roman*),resume]
        \item $\iota_{\BQ_p}$ extends to a very good integral Hodge embedding $\wt{\iota}: (\CG,\mu_h)\hookrightarrow (\GL(V_{\BZ_p}),\mu_g)$, where $V_{\BZ_p}\sset V_{\BQ_p}$ is a \dfn{self-dual} $\BZ_p$-lattice with respect to $\psi$; 
	    \item $\bG$ satisfies condition (\ref{eq-c-assump}).
    \end{enumerate}
	\begin{enumerate}  
        \item The $\bE'$-scheme $\Sh_{\rK_{2,p}^\circ}(\bG_2,X_2)_{\bE'}$ admits a $\bG_2(\BA_f^p)$-equivariant extension to a flat normal $\CO_{\bE',(v')}$-scheme $\sS_{\rK^\circ_{2,p}}(\bG_2,X_2)$. Any sufficiently small $\rK^p_2\sset \bG_2(\BA_f^p)$ acts freely on $\sS_{\rK_{2,p}^\circ}(\bG_2,X_2)$, and the quotient $$\sS_{\rK^\circ_2}(\bG_2,X_2)\coloneqq \sS_{\rK^\circ_{2,p}}(\bG_2,X_2)/\rK^p_2$$ is a flat normal $\CO_{\bE',(v')}$-scheme extending $\Sh_{\bK_2^\circ}(\bG_2,X_2)_{\bE'}$. 
		\item For any discrete valuation ring $R$ of mixed characteristic $0$ and $p$, the map $$\sS_{\rK_{2,p}^\circ}(\bG_2,X_2)(R)\ra \sS_{\rK_{2,p}^\circ}(\bG_2,X_2)(R[1/p])$$ is a bijection.
		\item There is a diagram of $\CO_{E'}$-schemes 
		    $$\xymatrix{
		   &\wt{\sS}_{\rK_{2,p}^\circ}^\ad\ar[ld]_{\pi}\ar[rd]^q &\\ \sS_{\rK_{2,p}^\circ}(\bG_2,X_2)_{\CO_{E'}} & &\BM^\loc_{\CG^\circ_2,\mu_{h_2}}\otimes_{\CO_{E_2}}\CO_{E'},
		}$$ where $\pi$ is a $\bG_2(\BA_f^p)$-equivariant $G_{2,\BZ_p}^{\ad}$-torsor, $q$ is $G_{2,\BZ_p}^{\ad}$-equivariant, and for any sufficiently small $\rK_2^p\sset \bG_2(\BA_f^p)$, the map $\wt{\sS}_{\rK_{2,p}^\circ}^\ad/\rK_2^p \ra \BM^\loc_{\CG^\circ_2,\mu_{h_2}}\otimes_{\CO_{E_2}}\CO_{E'}$ induced by $q$ is smooth of relative dimension $\dim \bG_2^\ad$. 
		If in addition, we have $\CG=\CG^\circ$, then $\pi$ reduces to a $G_{2,\BZ_p}^{\ad\circ}$-torsor. 
%		In such case, we have that if $\kappa'$ is a finite extension of $\kappa(v_2)$, and $y\in\sS_{K_{2,p}^\circ}(G_2,X_2)(\kappa')$, then there exists $z\in \BM^\loc_{\CG_2^\circ,\mu_{h_2}}(\kappa')$ such that we have an isomorphism of henselizations $$\CO_{\sS_{\rK_{2,p}^\circ}(\bG_2,X_2), y}^h\simeq \CO_{\BM^\loc_{\CG_2^\circ,\mu_{h_2}},z}^h.$$
	\end{enumerate}
\end{thm}

\begin{proof}
    Under the assumptions of the theorem, we can construct the integral model $\sS_{\rK_{2,p}^\circ}(\bG_2,X_2)$ as in Definition \ref{defn-abelian}. The properties of $\sS_{\rK_{2,p}^\circ}(\bG_2,X_2)$ are deduced from Theorem \ref{thm-hodgemodel} by following the arguments in \cite[Proposition 7.1.14]{kisin2024integralmodelsshimuravarieties} (cf. \cite[\S 4.4-4.6]{kisin2018integral}). Note that arguments in \cite[\S 4.4-4.6]{kisin2018integral} also work for $p=2$.
\end{proof}

\begin{prop} \label{abeliancan}
    Assume $(i)$ and $(ii)$ at the beginning of \S \ref{subsect-cond}. Assume further that \begin{enumerate}[label=(\roman*),start=5]
        \item any prime $v_2|p$ of $\bE_2$ splits completely in $\bE'=\bE\cdot\bE_2$;
        \item $G^\der$ is R-smooth (see \cite[2.1.4]{kisin2024integralmodelsshimuravarieties});
        \item $\pi_1(G^\der)_{I_{\BQ_p}}\ra \pi_1(G)_{I_{\BQ_p}}$ is injective, where $I_{\BQ_p}$ denotes the inertia subgroup of $\Gal(\ol{\BQ}_p/\BQ_p)$.
    \end{enumerate}
    Then for any prime $v_2|p$ of $\bE_2$, the construction in Definition \ref{defn-abelian} gives rise to a scheme $\sS_{\rK_{2,p}^\circ}(\bG_2,X_2)$ defined over $\CO_{E_2}=\CO_{\bE_2,v_2}$. Moreover, $\sS_{\rK_{2,p}^\circ}(\bG_2,X_2)$ is canonical in the sense of \cite{pappas2021p}.  
    
    In particular, $\sS_{\rK_2^\circ}(\bG_2,X_2)$ is independent of the choice of a Shimura datum $(\bG,X)$, as well as the choice of a symplectic embedding $(\bG,X)\hookrightarrow (\bGSp,S^\pm)$; see \cite[Corollary C]{daniels2024canonical}.
 % \begin{enumerate}[label=(\arabic*)]
 %        \item For any prime $v_2|p$ of $\bE_2$, the construction in Definition \ref{defn-model} gives rise to an $\CO_{\bE_2,(v_2)}$-scheme $\sS_{\rK_{2,p}^\circ}(\bG_2,X_2)$ with the properties as in Theorem \ref{thm-main}.
 %        \item  The integral model $\sS_{\rK_{2,p}^\circ}(\bG_2,X_2)$ is canonical in the sense of \cite{pappas2021p}, i.e., $\sS_{\rK_{2,p}^\circ}(\bG_2,X_2)$ satisfies the conditions in Conjecture \ref{conj}. 
 %    \end{enumerate} 
\end{prop}
\begin{proof}
    By condition (v), any prime $v_2|p$ of $\bE_2$ extends to a prime of $\bE'$, and $\CO_{\bE',v_2}=\CO_{\bE_2,v_2}$. Hence, we obtain an integral model  $\sS_{\rK_{2,p}^\circ}(\bG_2,X_2)$ over $\CO_{\bE_2,v_2}$. By definition and condition (v), the triple $(\bG,X,\CG^\circ)$ is a parahoric Shimura datum of Hodge type well-adapted to $(\bG_2,X_2,\CG_2^\circ)$ in the sense of \cite[Definition 3.8]{daniels2024canonical}. Since our construction (Definition \ref{defn-abelian}) of integral models is the same as the construction used in \cite[Proposition 7.1.14]{kisin2024integralmodelsshimuravarieties} (cf. \cite[\S 3.3]{daniels2024canonical}), the proof of \cite[Theorem 4.10]{daniels2024canonical} carry over verbatim. Note that the hypothesis $p>2$ in the proof of \cite[Theorem 4.10]{daniels2024canonical} is used only in \cite[Lemma 5.6]{daniels2024canonical} to ensure that $G^\der$ is R-smooth; we assume this directly here.
\end{proof}

% \begin{remark}
%     For a Shimura datum $(\bG_2,X_2)$ of abelian type as in Theorem \ref{thm-main}, the integral model $\sS_{\rK_2^\circ}(\bG_2,X_2)$ is canonical in the sense of \cite{pappas2021p}: our construction (Definition \ref{defn-abelian}) of integral models is the ``same" as the construction used in \cite[Proposition 7.1.14]{kisin2024integralmodelsshimuravarieties} (cf. \cite[\S 3.3]{daniels2024canonical}); the proofs of \cite[Theorem 4.10, Corollary 4.11]{daniels2024canonical} carry over verbatim.  Therefore, $\sS_{\rK_2^\circ}(\bG_2,X_2)$ is independent of the choice of a Shimura datum $(\bG,X)$, as well as the choice of a symplectic embedding $(\bG,X)\hookrightarrow (\bGSp,S^\pm)$, see \cite[Corollay C]{daniels2024canonical}. 
% %    \begin{enumerate}
% %    	\item Theorem \ref{thm1} in Case (A) of the Introduction follows from the above theorem. Note that the group $\bG$ in Theorem \ref{thm1} is denoted by $\bG_2$ here. 
% %    	\item For a Shimura datum $(\bG_2,X_2)$ of abelian type as in Theorem \ref{thm-main}, we expect that the integral model $\sS_{\rK_2^\circ}(\bG_2,X_2)$ is canonical in the sense of \cite{pappas2021p}, which would imply that $\sS_{\rK_2^\circ}(\bG_2,X_2)$ is independent of the choice of a Shimura datum $(\bG,X)$, as well as the choice of a symplectic embedding $(\bG,X)\hookrightarrow (\bGSp,S^\pm)$.    
% %    \end{enumerate}
% \end{remark}

\begin{remark}
    To construct integral models of Shimura varieties of abelian type as in Theorem \ref{thm-main}, it is enough to find a Hodge type lift satisfying conditions (i) -- (iv). For $p>2$, Kisin--Pappas--Zhou \cite[Proposition 7.2.19]{kisin2024integralmodelsshimuravarieties} achieve this for any Shimura datum of abelian type. Their arguments crucially rely on the hypothesis $p>2$.  For $p=2$, conditions (ii) and (iii) are subtle, especially for wildly ramified groups. 
    
    It is an interesting problem to determine when a very good embedding for a Hodge type lift exists for $p=2$. We verify it in cases (A) and (B) of Theorem \ref{thm1}. In general, one should not expect existence in all cases. For instance, even for tamely ramified groups, a necessary condition for the existence of a good integral Hodge embedding (see Definition \ref{goodrho}) is the normality of the associated Pappas--Zhu local model (cf. \cite[(3.3.12)]{kisin2018integral}); however, this normality may fail in general (see, e.g., \cite[Proposition 1.7]{haines2024normality}). Moreover, even if a good integral Hodge embedding exists, proving that it is very good also requires additional work: arguments such as \cite[Proposition 7.2.10]{kisin2024integralmodelsshimuravarieties} do not readily extend to $p=2$. 
\end{remark}

\subsubsection{Proof of Theorem \ref{thm1} in Case (A)} \label{caseA}

Now we start with a Shimura datum $(\bG_2,X_2)$ of abelian type with reflex field $\bE_2$, and denote by $\rK_{2,p}^\circ\sset G_2(\BQ_p)$ the parahoric subgroup associated to some $x_2\in \CB(G_2,\BQ_p)$.

\begin{prop}
	\label{lem-chooseGX} 
	Suppose that $(\bG_2^\ad,X_2^\ad)$ has no factor of type $D^\BH$, $G_2$ is unramified over $\BQ_p$, and $\rK_{2,p}^\circ$ is contained in some hyperspecial subgroup. Then there exists a Shimura datum $(\bG,X)$ of Hodge type, together with a central isogeny $\bG^\der\ra \bG_2^\der$ inducing an isomorphism $(\bG^\ad, X^\ad)\simeq (\bG_2^\ad,X_2^\ad)$, such that the following conditions hold.
	\begin{enumerate}[label=(\arabic*)]
		\item $\pi_1(G^\der)$ is trivial.  
		\item Any prime $v_2|p$ of $\bE_2$ splits completely in $\bE'=\bE\cdot \bE_2$. 
		\item $X_*(G^\ab)_{I_{\BQ_p}}$ is torsion free, where $G^\ab$ denotes the quotient $G/G^\der$.
		\item Conditions $(i)$ to $(iv)$ in Theorem \ref{thm-main} are satisfied.
	\end{enumerate}
\end{prop}
\begin{proof}
	As discussed in \cite[2.4.5]{kisin2024independenceellfrobeniusconjugacy}, the proof of \cite[Theorem 4.2]{edixhoven1992neron} shows that a tamely ramified torus is $R$-smooth. As we assume $G_2$ is unramified (in particular, $G_2$ is tamely ramified), we can choose a Hodge type lift $(\bG,X)$ as in \cite[Lemma 4.6.22]{kisin2018integral}. Then it remains to show that there exists a Hodge embedding $\iota: (\bG,X)\hookrightarrow (\bGSp(V,\psi),S^\pm)$ satisfying condition $(iii)$ of Theorem \ref{thm-main}. 
    
    By the proof of \cite[Lemma 4.6.22]{kisin2018integral}, $G=\bG_{\BQ_p}$ is also unramified. Let $x\in\CB(G,\BQ_p)$ be a lift of $x_2^\ad\in \CB(G^\ad_2,\BQ_p)=\CB(G^\ad,\BQ_p)$. Denote by $\CG=\CG_x$ be the associated Bruhat--Tits stabilizer group scheme.  Since $\pi_1(G^\der)$ is trivial by our choice of $(\bG,X)$, we may assume that, by Zarhin's trick and \cite[Corollary 2.3.16]{kisin2018integral}, there exists a good integral Hodge embedding $\wt{\iota}: (\CG,\mu_h)\hookrightarrow (\GL(\Lambda),\mu_g)$ extending $\iota_{\BQ_p}$, where $\Lambda\sset V_{\BQ_p}$ is a {self-dual} $\BZ_p$-lattice with respect to $\psi_{\BQ_p}$. Denote $\GSp\coloneqq \bGSp(V,\psi)_{\BQ_p}$.  By our assumption on $\rK_{2,p}^\circ$, the pair $(G,\CG)$ satisfies the assumptions in Theorem \ref{thm-tamefix}. Hence, there is a tame Galois extension $F/\BQ_p$ with Galois group $\Gamma$ such that in the diagram
	$$\xymatrix{
	    \CB(G,\BQ_p) \ar@{^{(}->}[r]^-{\iota}\ar@{^{(}->}[d] &\CB(\GSp,\BQ_p)\ar@{^{(}->}[d] \\ \CB(G_F,F)\ar@{^{(}->}[r] &\CB(\GSp_F,F)
	}$$
	of Bruhat--Tits buildings, we have \begin{itemize}
		\item the image of $x\in \CB(G,\BQ_p)$ in $\CB(G_F,F)$ is hyperspecial, and determines a reductive group $\CH$ over $\CO_F$ satisfying $\CG\simeq(\Res_{\CO_F/\BZ_p}\CH)^\Gamma$;
		\item the point $\iota(x)$ is hyperspecial corresponding to the self-dual lattice $\Lambda$, and its image in $\CB(\GSp_F,F)$ is hyperspecial corresponding to the  lattice $\wt{\Lambda}\coloneqq \Lambda\otimes_{\BZ_p}\CO_F$, which is self-dual with respect to the pairing $\psi_F$.
	\end{itemize}  
	By \cite[Lemma 3.1]{dokchitser2011}, there exist a totally real number field $\BF/\BQ$ and a place $w$ above $p$ such that $\BF_w\simeq F$. Let $\wt{V}$ denote the $\BQ$-vector space $V\otimes_\BQ\BF$. We pick an element $a\in \BF$ such that its image in $F$ generates the different ideal $\delta_{F/\BQ_p}$. Then $\wt{V}$ is equipped with a perfect alternating pairing given by $$\wt{\psi}(x,y)\coloneqq \Tr_{\BF/\BQ}(a\inverse\psi_F(x,y))$$ for $x,y\in \wt{V}$. Then $\wt{\Lambda}$ is self-dual with respect to $\wt{\psi}$, and the closed immersion \begin{flalign*}
		    \wt{\iota}: \CG\hookrightarrow \Res_{\CO_F/\BZ_p}\CH\hookrightarrow \GL(\wt{\Lambda})
	\end{flalign*} extends the Hodge embedding $G\hookrightarrow \GSp\hookrightarrow \bGSp(\wt{V},\wt{\psi})_{\BQ_p}\sset \GL(\wt{V}_{\BQ_p})$. As $\pi_1(G^\der)$ is trivial and $G$ is unramified over $\BQ_p$, the Pappas--Zhu local model for $(\CG,\mu_h)$ is isomorphic to $\BM^\loc_{\CG,\mu_h}$, and $\wt{\iota}$ is a good integral Hodge embedding by \cite[Proposition 2.3.7]{kisin2018integral}. As $(\bG_2^\ad,X_2^\ad)$ has no factor of type $D^\BH$, the closed immersion $\Res_{\CO_F/\BZ_p}\CH\hookrightarrow \GL(\wt{\Lambda})$ gives a very good integral Hodge embedding by \cite[Proposition 5.3.11, Theorem 1.2.3]{kisin2024integralmodelsshimuravarieties}. Since $\CG=(\Res_{\CO_F/\BZ_p}\CH)^\Gamma$, we obtain that $\wt{\iota}$ is also very good by \cite[Corollary 5.3.4]{kisin2024integralmodelsshimuravarieties}. We then obtain a desired Hodge embedding by replacing $\iota$ by the Hodge embedding $(\bG,X)\hookrightarrow (\bGSp(\wt{V},\wt{\psi}),S^\pm(\wt{V},\wt{\psi}))$.
\end{proof}

\begin{corollary}\label{coro-abelian}
	Keep the assumptions in Proposition \ref{lem-chooseGX}. Then the integral model $\sS_{\rK^\circ_{2,p}}(\bG_2,X_2)$ (constructed by using $(\bG,X)$ in Proposition \ref{lem-chooseGX}) is defined over $\CO_{\bE_2,v_2}$ for some prime $v_2|p$ of $\bE_2$. 
    
    Moreover, we have $\CG=\CG^\circ$, and the conclusions of Theorem \ref{thm-main} and Proposition \ref{abeliancan} hold. In particular, if $\kappa$ is a finite extension of $\kappa(v_2)$ and $y\in\sS_{\rK_{2,p}^\circ}(\bG_2,X_2)(\kappa)$, then there exists $z\in \BM^\loc_{\CG_2^\circ,\mu_{h_2}}(\kappa)$ such that we have an isomorphism of henselizations $$\CO_{\sS_{\rK_{2,p}^\circ}(\bG_2,X_2), y}^h\simeq \CO_{\BM^\loc_{\CG_2^\circ,\mu_{h_2}},z}^h.$$ 
\end{corollary}
\begin{proof}
	By Proposition \ref{lem-chooseGX} (2), there exists a prime  $v_2|p$ of $\bE_2$ extending to the prime $v'$ of $\bE'$, and we have $\CO_{\bE_2,v_2}\simeq \CO_{\bE',v'}$. Hence, the scheme $\sS_{\rK^\circ_{2,p}}(\bG_2,X_2)$ is defined over $\CO_{\bE_2,v_2}$. Since $\pi_1(G^\der)$ is trivial by Proposition \ref{lem-chooseGX} (1), we have $\pi_1(G)=\pi_1(G^\ab)=X_*(G^\ab)$ (cf. \cite[Corollary 1.7]{borovoi98}). Then $\pi_1(G)_{I_{\BQ_p}}$ is torsion-free by Proposition \ref{lem-chooseGX} (3). In particular, we have $\CG=\CG^\circ$. Since $G$ is unramified and $\pi_1(G^\der)$ is trivial, the assumptions in Proposition \ref{abeliancan} are satisfied. Hence, the conclusions in Proposition \ref{abeliancan} hold.
\end{proof}

By Theorem \ref{thm-main} and Corollary \ref{coro-abelian}, we obtain Theorem \ref{thm1} in Case (A). Note that the group $\bG$ in Theorem \ref{thm1} is denoted by $\bG_2$ here.

\subsection{Integral models of unitary Shimura varieties}  \label{sec-ulm}
In this subsection, we consider Shimura varieties in Case (B) of \S \ref{intro-results}. Note that the corresponding Shimura variety is of Hodge type. We show that, in this case, the assumptions in Theorem \ref{thm-hodgemodel} are satisfied, allowing us to construct integral models of Shimura varieties for which the conclusions of Theorem \ref{thm-hodgemodel} hold.  

\subsubsection{} \label{subsubsec-unitary}
Let $n=2m+1\geq 3$ be an odd integer. Let $\bF/\BQ$ be an imaginary quadratic extension such that $2$ is ramified in $\bF$. Then $F\coloneqq \bF\otimes_\BQ\BQ_2$ is a ramified quadratic extension of $\BQ_2$ with residue field $\BF_2$.
 Let $(\bV,\bh)$ be an $n$-dimensional non-degenerate $\bF/\BQ$-hermitian space of signature $(n-1,1)$. Denote by $$\bG\coloneqq \GU(\bV,\bh)$$ the unitary similitude group over $\BQ$ attached to $(\bV,\bh)$. 
% The vector space $V\coloneqq \bV\otimes_{\bF}F$ equipped with the $F/\BQ_2$-hermitian form $h\coloneqq \bh_{\BQ_2}$ defines a unitary similitude group $G=\bG_{\BQ_2}$ over $\BQ_2$. 
%\begin{lemma} \label{lem-splith}
%	For any non-degenerate hermitian form $h'$ on $V$, we have $G\simeq \GU(V,h')$.
%\end{lemma}
%\begin{proof}
%	By the classification of hermitian spaces over local fields (see, for example, \cite[Theorem 3.1]{jacobowitz1962hermitian}), there are two isomorphism classes of $n$-dimensional non-degenerate hermitian spaces over $\BQ_2$, classified by discriminants in $\BQ_2\cross/N_{F/\BQ_2}(F\cross)$. Let $a\in\BQ_2\cross$ be an element not in $N_{F/\BQ_2}(F\cross)$. Define a hermitian form $h_a$ on $V$ by setting $h_a(x,y)\coloneqq ah(x,y)$ for $x,y\in V$. Since $\disc(h_a)=a^{n}\disc(h)$ and $n$ is odd, the hermitian spaces $(V,h)$ and $(V,h_a)$ represent the two isomorphism classes of $n$-dimensional non-degenerate hermitian spaces over $\BQ_2$. Moreover, multiplication by $a$ induces an isomorphism between $\GU(V,h)$ and $\GU(V,h_a)$. Hence, the lemma follows. 
%\end{proof}
% By Lemma \ref{lem-splith}, we may assume that the hermitian form $h$ is split, that is, there exists an $F$-basis $e_1,\ldots,e_n$ of $V$ such that $h(e_i,e_j)=\delta_{i,n+1-j}$. 
 Suppose that $$\rK_2\sset \bG(\BQ_2)$$ is a special parahoric subgroup in the sense of Bruhat--Tits theory.  For an open compact subgroup of the form $\rK=\rK_2\rK^2\sset \bG(\BA_f)$, where $\rK^2\sset \bG(\BA_f^2)$ is open compact and sufficiently small, we can associate a Shimura variety $\Sh_\rK(\bG,X)$ of level $\rK$ as in \cite[\S 1.1]{pappas2009local}. Then $\Sh_\rK(\bG,X)$ is a quasi-projective smooth variety of dimension $n-1$ over the reflex field $\bF$. Denote by $\Sh_\rK(\bG,X)_F$ the base change of $\Sh_\rK(\bG,X)$ to $F$.  

%Let $a\in \bF\cross$ be an element such that $a=-\ol{a}$. Then the hermitian form $\bh$ on $\bV$ induces a perfect alternating $\BQ$-bilinear form $\psi$ on $\bV$ by setting $$\psi(x,y)\coloneqq  \Tr_{\bF/\BQ}(a\inverse \bh(x,y)), \text{\ for $x,y\in \bV$}.$$ 
%Denote by $\bGSp(\bV,\psi)$ the symplectic similitude group over $\BQ$ associated with the above pairing. Then we obtain an embedding $\iota_1: \bG\hookrightarrow \bGSp(\bV,\psi)$, which also induces an embedding of Shimura data \begin{flalign*}
%	   \iota_1: (\bG,X)\hookrightarrow (\bGSp(\bV,\psi), S^\pm(\bV,\psi)).
%\end{flalign*} 

\subsubsection{Unitary local models} \label{subsubunitary}
Before proving Theorem \ref{thm1} in Case (B), we first review the results from \cite{yang24} and then apply them to show that the tangent spaces of unitary local models at closed points are spanned by smooth formal curves. This property is a key ingredient in the construction of very good Hodge embeddings. 

Let $F_0/\BQ_2$ be a finite extension\footnote{In our application, we take $F_0$ to be $\BQ_2$.} and $F$ be a (wildly) ramified quadratic extension of $F_0$. Let $k$ denote the algebraic closure of the finite field $\BF_2$. Let $(V,h)$ be a hermitian space, where $V$ is an $F$-vector space of dimension $n=2m+1 \geq 3$ and $h: V\times V\ra F$ is a non-degenerate hermitian form. Assume that $h$ is split, i.e., there exists an $F$-basis $(e_i)_{1\leq i\leq n}$ of $V$ such that $h(e_i,e_j)=\delta_{i,n+1-j}$ for $1\leq i,j\leq n$. Let $G\coloneqq \GU(V,h)$ denote the unitary similitude group over $F_0$ attached to $(V,h)$. There are two conjugacy classes of \dfn{special} parahoric subgroups of $G$. As in \cite{yang24}, we index them by $I=\cbra{0}$ or $\cbra{m}$. Let $\CG_I$ denote the special parahoric group scheme corresponding to $I=\cbra{0}$ or $\cbra{m}$. By \cite[1.2.3]{pappas2009local}, $\sG_I$ is a Bruhat--Tits stabilizer group scheme. Let $\mu$ denote the geometric cocharacter $\BG_{m,\ol{F}}\ra G_{\ol{F}}\simeq \GL_{n,\ol{F}}\times \BG_{m,\ol{F}}$ given by $z\mapsto (\diag(z,1^{(n-1)}),z)$. Let $\BM^\loc_{\CG_I,\mu}$ be the local model attached to $(\CG_I,\mu)$ by Theorem \ref{thm-localmodels}. The following theorem is essentially contained in \cite{yang24}.
\begin{thm} \label{thm-goodlattice}
	Let $\RM^\loc_I$ denote the unitary local model in \cite{yang24} corresponding to $I=\cbra{0}$ or $\cbra{m}$. 
	\begin{enumerate}
		\item There exists a good integral Hodge embedding $(\CG_I,\mu)\hookrightarrow (\GL(\Lambda_I),\mu_n)$ for some lattice $\Lambda_I\sset V$. The group $$\cbra{g\in G(F_0)\ |\ g\Lambda_I=\Lambda_I}$$ is a special parahoric subgroup of $G(F_0)$. 
		\item The scheme $\RM^\loc_I$ is isomorphic to $\BM^\loc_{\CG_I,\mu}$.
		\item If $I=\cbra{0}$ (resp. $I=\cbra{m}$), then $\RM^\loc_I$ is $\CO_F$-smooth on the complement of a single closed point (resp. $\CO_F$-smooth). 
	\end{enumerate}
\end{thm}
\begin{proof}
	(2) and (3) follow from \cite[Theorem 1.2, 1.4]{yang24}. It remains to prove (1). Choose $\Lambda_I$ as in \cite[Theorem 1.1]{yang24}. By the concrete description of the parahoric group scheme $\CG_I$ in \cite[Theorem 3.9]{yang24}, there is a closed immersion $\iota: \CG_I\hookrightarrow \GL(\Lambda_I)$. The base change $\iota_{\BQ_p}$ is the standard Hodge embedding $G=\GU(V,h) \hookrightarrow \GL(V)$, which sends the conjugacy class $\cbra{\mu}$ to $\cbra{\mu_n}$. As $G$ contains the scalars, $\iota$ is an integral Hodge embedding. Moreover, $\iota$ is good, since it induces a closed immersion $\BM^\loc_{\CG_I,\mu}\simeq \RM^\loc_I\hookrightarrow \Gr(n,\Lambda_I)_{\CO_F}$ (cf. the proof of \cite[Theorem 8.4]{yang24}).
\end{proof}

\begin{thm} \label{thm-tangent}
   For any closed point $x\in \RM^\loc_I(k)$, the tangent space of the special fiber $\RM^\loc_I\otimes_{\CO_F}k$ at $x$ is spanned by smooth formal curves.
\end{thm}

The proof of Theorem \ref{thm-tangent} is divided into the following two cases. 

\subsubsection*{The case \texorpdfstring{$I=\cbra{m}$}{I={m}}}
By Theorem \ref{thm-goodlattice} (3), the local model $\RM^\loc_{\cbra{m}}$ is smooth over $\CO_F$. Clearly Theorem \ref{thm-tangent} holds in this case by the infinitesimal lifting property of smooth morphisms. 

\subsubsection*{The case \texorpdfstring{$I=\cbra{0}$}{I={0}}}
By Theorem \ref{thm-goodlattice} (3), $\RM^\loc_{\cbra{0}}$ is $\CO_F$-smooth on the complement of a single closed point, which we will call the \dfn{worst point}. To prove Theorem \ref{thm-tangent} in this case, it suffices to prove the tangent space of $\RM^\loc_{\cbra{0}}\otimes_{\CO_F}k$ at the worst point is spanned by smooth formal curves.
The following proposition directly follows from \cite[Theorem 1.2, 1.3]{yang24}. 

\begin{prop} \label{prop-localmod}
	Let $A$ (\resp $B$) be a $2m\times 2m$ (resp. $2m\times 1$) matrix with variables as entries. Let $H_{2m}$ denote the $2m\times 2m$ anti-diagonal unit matrix. There is an open affine subscheme $\RU_{\cbra{0}}^\loc$ of $\RM^\loc_{\cbra{0}} \otimes_{\CO_F} k$ which contains the worst point such that $\RU^\loc_{\cbra{0}}$ is isomorphic to 
	\begin{flalign*}
		  \Spec \frac{\CO_F[A|B]}{(\wedge^2(A|B), A-A^t, \frac{\tr(H_{2m}A)}{2}A+BB^t)}\otimes_{\CO_F}k.
	\end{flalign*} 
	(We remark that under the relation $A-A^t=0$, the polynomial $\tr(H_{2m}A)$, which is the sum of the anti-diagonal entries of $A$, is indeed divisible by $2$ in $\CO_F[A]$.) Moreover, the worst point corresponds to the point defined by $A=B=0$ and $\RU^\loc_{\cbra{0}}$ is smooth on the complement of the worst point. 
\end{prop}

\begin{defn} \label{defn-tangent}
	Let $X$ be an affine scheme of finite type over $k$. Let $x\in X(k)$ be a $k$-point. We may express $X$ as a closed subscheme of $\BA^d=\Spec k[T_1,\ldots,T_d]$ defined by an ideal $\fa\sset k[T_1,\ldots,T_d]$ such that $x$ is the origin of $\BA^d$. 
	\begin{enumerate}
		\item For a polynomial $f\in k[T_1,\ldots,T_d]$, write $f=\sum_{i=r}^Nf_i$ as a decomposition into homogeneous polynomials with $f_r\neq 0$. Denote by $f^*$ (resp. $f^{(1)}$) the lowest degree term $f_r$ (resp. $f_1$). If $r\geq 2$, set $f_1=0$.   
		\item Denote by $\fa^*$ (resp. $\fa^{(1)}$) the ideal in $k[T_1,\ldots,T_d]$ generated by $f^*$ (resp. $f^{(1)}$), for all $f\in \fa$. The \dfn{tangent cone} $TC_xX$ (resp. \dfn{schematic tangent space} $T_x^{sch}X$) of $X$ at $x$ is the scheme $\Spec k[T_1,\ldots,T_d]/\fa^*$ (resp. $\Spec k[T_1,\ldots,T_d]/\fa^{(1)}$).
	\end{enumerate}
\end{defn}
Note that the definition of $TC_xX$ (resp. $T_x^{sch}X$) is independent of the embeddings of $X$ in affine spaces. See \cite[Chapter III, \S 3, 4]{mumford1999red}. Clearly $T_x^{sch}X$ is a linear subspace of $\BA^d$ and there is a closed immersion $TC_xX\hookrightarrow T_x^{sch}X$. Note that there is a natural bijection between the $k$-points $T_x^{sch}X(k)$ and the tangent space $T_xX$, see \cite[\S 4]{mumford1999red}. Concretely, for any $z\in T^{sch}_xX(k)$ corresponding to a $k$-algebra homomorphism $z: k[T_1,\ldots,T_d]/\fa^{(1)} \ra k$, we can associate a $k$-algebra homomorphism $t_z: k[T_1,\ldots,T_d]/\fa \ra k[t]/(t^2)$ via $T_i\mapsto z(T_i)t$. The morphism $t_z$ defines a tangent vector of $X$ at $x$.

%\begin{defn} \label{defn-tangent}
%	Let $X$ be an affine scheme of finite type over $k$. Let $x\in X(k)$ be a $k$-point. We may express $X$ as a closed scheme of $\BA^d=\Spec k[T_1,\ldots,T_d]$ defined by an ideal $\fa\sset k[T_1,\ldots,T_d]$ such that $x$ is the origin of $\BA^d$. 
%	\begin{enumerate}
%		\item For a polynomial $f\in k[T_1,\ldots,T_d]$, write $f=\sum_{i=r}^Nf_i$ as a decomposition into homogeneous polynomials with $f_r\neq 0$. Denote by $f^*$ the lowest degree term $f_r$.    
%		\item Denote by $\fa^*$ the ideal in $k[T_1,\ldots,T_d]$ generated by $f^*$, for all $f\in \fa$. The \dfn{tangent cone} $TC_xX$ of $X$ at $x$ is the scheme $\Spec k[T_1,\ldots,T_d]/\fa^*$.
%	\end{enumerate}
%\end{defn}
%Note that the definition of $TC_xX$ is independent of the embeddings of $X$ into affine spaces. See \cite[Chapter III, \S 3]{mumford1999red}. 

\begin{lemma} \label{lem-spantc}
	Let $X$ be a reduced affine scheme of finite type over $k$. Let $x\in X(k)$. Assume that there exists a closed immersion $i: X\hookrightarrow \BA^d$ such that $X$ is defined by a homogeneous ideal $\fa$ and $i(x)$ is the origin $O$ of $\BA^d$. Then the set $TC_xX(k)$ spans the $k$-vector space $T_xX$. 
\end{lemma}
\begin{proof}
	Without loss of generality, we may assume that $i$ does not factor through any (proper) linear subspace of $\BA^d$. As $X$ is reduced, the image $i(X)$ is not contained in any (proper) linear subspace of $\BA^d$. Since $\fa$ is homogeneous, $X$ is isomorphic to the tangent cone $TC_xX$ and $i$ is identified with the embedding $TC_xX\hookrightarrow T_x^{sch}X\hookrightarrow T_O^{sch}\BA^d$. Let $W$ denote the subspace in $T_xX$ spanned by $TC_xX(k)$. Then we have a linear subspace $W^{sch}\sset \BA^d$ such that $W^{sch}(k)=W$. We obtain a factorization $$i: X\hookrightarrow W^{sch}\hookrightarrow T_x^{sch}X\hookrightarrow \BA^d.$$ Since $i: X\hookrightarrow\BA^d$ does not factor through any proper linear subspace of $\BA^d$, it forces that $W^{sch}=T_x^{sch}X=\BA^d$, and hence, $W=T_xX$. 
\end{proof}

\begin{corollary} \label{coro-span}
	Under the same assumptions as in Lemma \ref{lem-spantc}, the tangent space $T_xX$ is spanned by smooth formal curves.
\end{corollary}
\begin{proof}
	Denote $X=\Spec R=\Spec k[T_1,\ldots,T_d]/\fa$. By assumption, the tangent cone $TC_xX$ is isomorphic to $X$. Recall that for a $k$-point $z\in TC_xX(k)$ corresponding to $z: R=k[T_1,\ldots,T_d]/\fa\ra k$, the associated tangent vector $t_z\in X(k[t]/(t^2))$ is given by the $k$-algebra homomorphism $R\ra k[t]/(t^2)$ sending $T_i\mapsto z(T_i)t$. Define a $k$-algebra homomorphism $\wt{t}_z: k[T_1,\ldots,T_d] \ra k[[t]]$ via $T_i\mapsto z(T_i)t$. For any homogeneous polynomial $f\in \fa$, we have \begin{flalign*}
		    \wt{t}_z(f)=f(z(T_1)t,\ldots, z(T_d)t)= t^{\deg f}f(z(T_1),\ldots,z(T_d))=0.
	\end{flalign*} Hence, the map $\wt{t}_z$ factors through $R/\fa$. In other words, the tangent vector $t_z$ lifts to the smooth formal curve $\wt{t}_z\in X(k[[t]])$. Now the corollary follows from Lemma \ref{lem-spantc} immediately.  
\end{proof}

By Proposition \ref{prop-localmod}, the scheme $\RU^\loc_{\cbra{0}}\otimes_{\CO_F}k$ is defined by a homogeneous ideal under the obvious embedding $\RU^\loc_{\cbra{0}}\otimes_{\CO_F}k \hookrightarrow \Spec k[A|B]$, which sends the worst point to the origin. By Corollary \ref{coro-span}, we obtain the following.

\begin{corollary} \label{cor-cone}
	The tangent space of $\RM^\loc_{\cbra{0}}\otimes_{\CO_F}k$ at the worst point is spanned by smooth formal curves.  
\end{corollary}

This proves Theorem \ref{thm-tangent} in the case $I=\cbra{0}$.

\subsubsection{Proof of Theorem \ref{thm1} in Case (B)} \label{sec-application}
Let us keep the notation as in \S \ref{subsubsec-unitary}.
%Let $\bF/\BQ$ be an imaginary quadratic extension such that $2$ is ramified in $\bF$. Then $F\coloneqq \bF\otimes_\BQ\BQ_2$ is a ramified quadratic extension of $\BQ_2$ with residue field $\BF_2$.
%Let $n=2m+1\geq 3$ be an odd integer. 
% Let $(\bV,\bh)$ be an $n$-dimensional non-degenerate $\bF/\BQ$-hermitian space of signature $(n-1,1)$. Denote by $\bG\coloneqq \GU(\bV,\bh)$ the unitary similitude group over $\BQ$ attached to $(\bV,\bh)$. 
 The vector space $V\coloneqq \bV\otimes_{\bF}F$ equipped with the $F/\BQ_2$-hermitian form $h\coloneqq \bh_{\BQ_2}$ defines a unitary similitude group $G=\bG_{\BQ_2}$ over $\BQ_2$. 
\begin{lemma} \label{lem-splith}
	For any non-degenerate hermitian form $h'$ on $V$, we have $G\simeq \GU(V,h')$.
\end{lemma}
\begin{proof}
	By the classification of hermitian spaces over local fields (see, for example, \cite[Theorem 3.1]{jacobowitz1962hermitian}), there are two isomorphism classes of $n$-dimensional non-degenerate hermitian spaces over $\BQ_2$, classified by discriminants in $\BQ_2\cross/N_{F/\BQ_2}(F\cross)$. Let $a\in\BQ_2\cross$ be an element not in $N_{F/\BQ_2}(F\cross)$. Define a hermitian form $h_a$ on $V$ by setting $h_a(x,y)\coloneqq ah(x,y)$ for $x,y\in V$. Since $\disc(h_a)=a^{n}\disc(h)$ and $n$ is odd, the hermitian spaces $(V,h)$ and $(V,h_a)$ represent the two isomorphism classes of $n$-dimensional non-degenerate hermitian spaces over $\BQ_2$. Moreover, multiplication by $a$ induces an isomorphism between $\GU(V,h)$ and $\GU(V,h_a)$. Hence, the lemma follows. 
\end{proof}

 By Lemma \ref{lem-splith}, we may assume that the hermitian form $h$ is split, that is, there exists an $F$-basis $e_1,\ldots,e_n$ of $V$ such that $h(e_i,e_j)=\delta_{i,n+1-j}$. Up to conjugation, we may assume the special parahoric subgroup $\rK_2\sset \bG(\BQ_2)$ corresponds to $I=\cbra{0}$ or $\cbra{m}$ as in \S \ref{subsubunitary}.  
% For an open compact subgroup of the form $\rK=\rK_2\rK^2\sset \bG(\BA_f)$, where $\rK^2\sset \bG(\BA_f^2)$ is open compact and sufficiently small, we can associate a Shimura variety $\Sh_\rK(\bG,X)$ of level $\rK$ as in \cite[\S 1.1]{pappas2009local}. Then $\Sh_\rK(\bG,X)$ is a quasi-projective smooth variety of dimension $n-1$ over the reflex field $\bF$. Denote by $\Sh_\rK(\bG,X)_F$ the base change of $\Sh_\rK(\bG,X)$ to $F$.  

Let $a\in \bF\cross$ be an element such that $a=-\ol{a}$. Then the hermitian form $\bh$ on $\bV$ induces a perfect alternating $\BQ$-bilinear form $\psi$ on $\bV$ by setting $$\psi(x,y)\coloneqq  \Tr_{\bF/\BQ}(a\inverse \bh(x,y)), \text{\ for $x,y\in \bV$}.$$ 
Denote by $\bGSp(\bV,\psi)$ the symplectic similitude group over $\BQ$ associated with the above pairing. Then we obtain an embedding $\iota_1: \bG\hookrightarrow \bGSp(\bV,\psi)$, which also induces an embedding of Shimura data \begin{flalign*}
	   \iota_1: (\bG,X)\hookrightarrow (\bGSp(\bV,\psi), S^\pm(\bV,\psi)).
\end{flalign*} 
Let $\Lambda_I\sset V$ be the lattice as in Theorem \ref{thm-goodlattice}. Then by Theorem \ref{thm-goodlattice} (1), there exists a good integral Hodge embedding $\wt{\iota}_1: (\CG_I,\mu)\hookrightarrow (\GL(\Lambda_I),\mu_n)$ extending $\iota_{1,\BQ_2}$. By Theorem \ref{thm-tangent} and Lemma \ref{lem-verygood}, $\wt{\iota}_1$ is very good. Denote by $\Lambda_I^\#\sset V$ the dual lattice of $\Lambda_I$ with respect to $\psi$.  Set $\Lambda\coloneqq (\Lambda_I)^4\oplus(\Lambda_I^\#)^4\sset V^{8}$. Using Zarhin's trick as in the proof of \cite[Proposition 7.2.10 (3)]{kisin2024integralmodelsshimuravarieties}, there exists a non-degenerate alternating pairing $\psi'$ on $\bV^8$ such that $\Lambda$ is self-dual with respect to $\psi'_{\BQ_2}$, and an embedding of Shimura data \begin{flalign*}
	     \iota: (\bG,X)\hookrightarrow (\bGSp(\bV^8,\psi'), S^\pm(\bV^8,\psi'))
\end{flalign*} such that $\iota$ extends to a very good integral Hodge embedding $(\CG_I,\mu)\hookrightarrow (\GL(\Lambda),\mu_{8n})$. 

Denote $(\bGSp,S^\pm)\coloneqq (\GSp(\bV^8,\psi'), S^\pm(\bV^8,\psi'))$. Then we obtain an embedding of Shimura data \begin{flalign*}
	   \iota: (\bG,X)\hookrightarrow (\bGSp, S^\pm).
\end{flalign*} Moreover, the embedding $\iota_{\BQ_2}$ extends to a very good integral Hodge embedding by previous discussion. Note that for odd unitary similitude groups, the parahoric group scheme corresponding to $\rK_2$ is connected by \cite[1.2.3]{pappas2009local}. In particular, the assumptions in Theorem \ref{thm-hodgemodel} are satisfied and we obtain the following theorem.

\begin{thm}
	There exists a normal flat $\CO_F$-scheme $\sS_\rK(\bG,X)$ extending $\Sh_\rK(\bG,X)$ such that the conclusions of Theorem \ref{thm-hodgemodel} hold for $\sS_\rK(\bG,X)$.  
\end{thm}

This finishes the proof of Theorem \ref{thm1} in Case (B). Note that $\sS_\rK(\bG,X)$ is canonical in the sense of \cite{pappas2021p}, see Remark \ref{hodgecan}.

\appendix
\section{Bruhat--Tits stabilizer group schemes and tame Galois fixed points}  \label{appsec}
In this appendix, we establish a result on Bruhat--Tits stabilizer group schemes, which is used in the proof of Proposition \ref{lem-chooseGX} to construct very good Hodge embeddings in Case (A).

Let $F$ be a complete discrete valued field with residue characteristic $p=2$. Let $G$ be a connected reductive group over $F$. Denote by $\CB(G,F)$ (resp. $\ol{\CB}(G,F)$) the extended (resp. ``classical") Bruhat--Tits building. Recall that for a finite tame Galois extension $K/F$ with Galois group $\Gamma$, the inclusion $$\CB(G,F)\hookrightarrow \CB(G,K)$$ of buildings identifies the image with the fixed point set $\CB(G,K)^\Gamma$. For $x\in \CB(G,F)$, we use $\CG^K_x$ to denote the Bruhat--Tits stabilizer group scheme over $\CO_K$ attached to the image of $x$ in $\CB(G,K)$. 

\begin{thm} \label{thm-tamefix}
	Assume that $G$ is unramified. Let $\CG=\CG_\bbf$ be the Bruhat--Tits stabilizer group scheme attached to some facet $\bbf$ in $\CB(G,F)$ whose closure contains a hyperspecial point.
	
	Then there exist a point $x\in \CB(G,F)$ and a finite tame Galois extension $K/F$ with Galois group $\Gamma$ such that $G\otimes_FK$ is split, $\CG=\CG_x$, and (the image of) $x$ is hyperspecial in $\CB(G,K)$. Moreover, we have an isomorphism of (smooth) $\CO_F$-group schemes $$\CG\simeq (\Res_{\CO_K/\CO_F}\CG_x^K)^\Gamma$$ extending the isomorphism $G\simeq (\Res_{K/F}G_K)^\Gamma$.
\end{thm}

Theorem \ref{thm-tamefix} may be viewed as the $p=2$ analogue of \cite[Proposition 2.2.2]{kisin2024integralmodelsshimuravarieties} and \cite[Theorem 1.1]{Pappas_Rapoport_2024}, where the corresponding results require $p$ to be odd.

\subsection{The proof of Theorem \ref{thm-tamefix}}
We first consider the case when $G$ is split, absolutely simple, and simply connected. Fix a maximal torus $T$ and a Borel subgroup $B$ containing $T$. Let $\Delta=\cbra{\alpha_1,\ldots,\alpha_n}$ be the subset of simple roots with respect to $(T,B)$ in the root system $\Phi=\Phi(T,B)$. Denote by $\Phi^+=\Phi\cap \BZ_{\geq 0}\Delta$ the set of positive roots. Note that there is a perfect pairing $$\pair{-,-}: X_*(T)\times X^*(T)\ra \BZ $$ between the cocharacter group $X_*(T)$ and the character group $X^*(T)$ of $T$. There is an isomorphism between the apartment $\CA$ of $\CB(G,F)$ corresponding to $T$ and $V\coloneqq X_*(T)_\BR$ such that the origin in $V$ corresponds to a special vertex, which is also hyperspecial, in $\CA$. Moreover, a chamber $C$ of $\CA$ is given by \begin{flalign*}
	    C= \cbra{x\in V\ |\ 0<\pair{x,\alpha}<1 \text{\ for all\ }\alpha\in\Phi^+ }.
\end{flalign*}
For $1\leq i\leq n$, denote by ${\omega_i}\in V$ the fundamental coroot corresponding to $\alpha_i\in \Delta$. 
Then the chamber $C$ has $n+1$ vertices $v_0,\ldots,v_n$, where $v_0=0$ and $v_i=\omega_i/c_i$ for $1\leq i\leq n$, where $c_i$ is a positive integer such that $\sum_{i=1}^nc_i\alpha_i$ is the highest root in $\Phi$. Since $G(F)$ acts transitively on the set of chambers in $\CA$ (see, for example, \cite[\S 1.8]{tits1979reductive}), we may assume that $\bbf$ is contained in the closure of $C$. By assumption, the closure $\ol{\bbf}$ of $\bbf$ contains a hyperspecial vertex $v_\bbf$. Note that $v_\bbf$ is some vertex $v_i$ for which $c_i=1$. If $\ol{\bbf}$ consists of only a single point, there is nothing to prove. Hence, we may assume that $\ol{\bbf}$ strictly contains $v_\bbf$.   Let $y\in V$ be the barycenter of the (sub)facet determined by the vertices in $\ol{\bbf}$ except $v_\bbf$. Then $y$ is of the form $$y=\frac{1}{m2^d}y_1,$$ where $m$ is an odd integer, $d\geq 0$ is an integer, and $y_1\in \BZ\Delta$. Set $$x\coloneqq \frac{1}{m2^{d+1}+1}v_\bbf + \frac{m2^{d+1}}{m2^{d+1}+1}y= \frac{1}{m2^{d+1}+1}v_\bbf + \frac{2}{m2^{d+1}+1}y_1$$ 
%\begin{flalign*}
%	    x =\begin{cases}
%	    	  \frac{1}{2m+1}v_\bbf +\frac{2m}{2m+1}y= \frac{1}{2m+1}v_\bbf + \frac{2}{2m+1}y_1 \quad &\text{if $d=0$},
%	    	  \\ \frac{1}{m2^d+1}v_\bbf + \frac{m2^d}{m2^d+1}y= \frac{1}{m2^d+1}v_\bbf + \frac{1}{m2^d+1}y_1    &\text{if $d>0$.}
%	    \end{cases} 
%\end{flalign*}
Then $x$ lies in the line segment between $v_\bbf$ and $y$, and hence in $\bbf$. Since $G$ is simply connected, we have $$\CG_x=\CG_\bbf.$$ Let $F_1$ be a finite extension of $F$ with ramification index $m2^{d+1}+1$. Denote by $$\rho: \CB(G,F)\hookrightarrow \CB(G,F_1),$$ the natural inclusion of buildings. Then we see that \begin{flalign*}
	    \rho(x)= v_\bbf + 2y_1\in v_\bbf+X_*(T).
\end{flalign*}
%\begin{flalign*}
%	    \rho(x) =\begin{cases}
%	    	mv_\bbf +2y_1 \quad &\text{if $d=0$},\\ v_\bbf +
%	    \end{cases}
%\end{flalign*} 
Thus, $\rho(x)$ is a hyperspecial point in $\CB(G,F_1)$. As $p=2$, the extension $F_1/F$ is tame. Let $K$ be the Galois closure of $F_1/F$. Then $K$ is a tame Galois extension of $F$. Note that the image of $\rho(x)$ in $\CB(G,K)$ is also hyperspecial. The pair $(K,x)$ satisfies the conclusion of Theorem \ref{thm-tamefix}. 

Next we consider the case when $G$ is unramified, absolutely simple and simply connected. Let $F_1/F$ be an unramified Galois extension over which $G$ is split. Denote by $\Gamma_1$ the Galois group of $F_1/F$. Then the facets in $\CB(G,F)$ correspond to $\Gamma_1$-invariant facets in $\CB(G,F_1)$. Let $\bbf_1$ be the $\Gamma_1$-invariant facet in $\CB(G,F_1)$ corresponding to $\bbf$. The closure of the facet $\bbf_1$ contains a hyperspecial point, which is the image of $v_\bbf$ in $\CB(G,F_1)$. Let $y_1$ be the barycenter of $\bbf_1$. Then $y_1$ is a fixed point of $\Gamma_1$ and we have $$\CG=(\Res_{\CO_{F_1}/\CO_F}\CG_{y_1}^{F_1})^{\Gamma_1}.$$
Note that $y_1$ is of the form \begin{flalign*}
	   y_1= \frac{1}{m2^d}(v_\bbf + y_2),
\end{flalign*}
where $m$ is odd and $y_2\in X_*(T)$ for a maximal torus $T$ in the split group $G_{F_1}$. 
Since $y_1$ and $v_\bbf$ are fixed by $\Gamma_1$, so is any point in the line segment of $y_1$ and $v_\bbf$.
Set \begin{flalign*}
	   x\coloneqq \frac{1}{m2^{d+1}+1}v_\bbf + \frac{m2^{d+1}}{m2^{d+1}+1}y_1= \frac{3}{m2^{d+1}+1}v_\bbf + \frac{2}{m2^{d+1}+1}y_2.
\end{flalign*}
Then $x$ lies in the line segment between $y_1$ and $v_\bbf$, and hence is fixed by $\Gamma_1$. We obtain that $x$ corresponds to a point in $\CB(G,F)$ and $\CG_x=\CG_\bbf$. 
Let $F_2$ be a finite (tame) extension of $F_1$ with ramification index $m2^{d+1}+1$. Then the image of $x$ in $\CB(G,F_2)$ is of the form $3v_\bbf+2y_2\in 3v_\bbf+X_*(T)$. Since $3v_\bbf$ is hyperspecial, $x$ is hyperspecial in $\CB(G,F_2)$. Let $K$ be the Galois closure of $F_2/F$. Then $K$ is a tame Galois extension of $F$ and the pair $(K,x)$ satisfies the conclusion of Theorem \ref{thm-tamefix}. In particular, Theorem \ref{thm-tamefix} holds when $G$ is unramified, absolutely simple and simply connected. 

Following the proof of \cite[Proposition 2.2.2]{kisin2024integralmodelsshimuravarieties}, we see that Theorem \ref{thm-tamefix} holds when $G$ is any unramified group over $F$.

\bibliographystyle{plain}
\bibliography{myref}

%\printbibliography

%\begin{thebibliography}{99}
%	\bibitem{stack} https://stacks.math.columbia.edu/tag/0CC6
%\end{thebibliography}

\end{document}